\documentclass[a4paper,12pt,oneside]{article}

\usepackage[utf8]{inputenc}
\usepackage[english]{babel}
\usepackage{amsthm}
\usepackage{amssymb}
\usepackage{amsmath}
\usepackage{hyperref}
\usepackage{mathrsfs}

\usepackage{pstricks}

 \usepackage{tikz}

\usepackage{graphicx}
\graphicspath{{images/}}

\usepackage{longtable,geometry}
\usepackage{color}

\newtheoremstyle{note}{12pt}{12pt}{}{}{\bfseries}{.}{.5em}{}
\title{\LARGE\textbf{Coexistence of Non-Periodic Attractors}}
\author{Liviana Palmisano}

\makeatletter
\newtheorem{theo}[equation]{Theorem}
\newtheorem{prop}[equation]{Proposition}

\newtheorem{defin}[equation]{Definition}

\newtheorem{rem}[equation]{Remark}

\numberwithin{equation}{section}
\newtheorem{lem}[equation]{Lemma}

\usepackage{babel}

\newcommand{\N}{{\mathbb N}}
\newcommand{\Z}{{\mathbb Z}}
\newcommand{\R}{{\mathbb R}}

\newcommand{\Co}{{{\mathcal C}^0}}

\newcommand{\Cd}{{{\mathcal C}^2}}
\newcommand{\Ct}{{{\mathcal C}^3}}
\newcommand{\Cq}{{{\mathcal C}^4}}

\newcommand{\Cuno}{{{\mathcal C}^1}}

\begin{document}
\maketitle
\author
\textcolor{blue}{}\global\long\def\sbr#1{\left[#1\right] }
\textcolor{blue}{}\global\long\def\cbr#1{\left\{  #1\right\}  }
\textcolor{blue}{}\global\long\def\rbr#1{\left(#1\right)}
\textcolor{blue}{}\global\long\def\ev#1{\mathbb{E}{#1}}
\textcolor{blue}{}\global\long\def\R{\mathbb{R}}
\textcolor{blue}{}\global\long\def\E{\mathbb{E}}
\textcolor{blue}{}\global\long\def\norm#1#2#3{\Vert#1\Vert_{#2}^{#3}}
\textcolor{blue}{}\global\long\def\pr#1{\mathbb{P}\rbr{#1}}
\textcolor{blue}{}\global\long\def\qq{\mathbb{Q}}
\textcolor{blue}{}\global\long\def\aa{\mathbb{A}}
\textcolor{blue}{}\global\long\def\ind#1{1_{#1}}
\textcolor{blue}{}\global\long\def\pp{\mathbb{P}}
\textcolor{blue}{}\global\long\def\cleq{\lesssim}
\textcolor{blue}{}\global\long\def\ceq{\eqsim}
\textcolor{blue}{}\global\long\def\Var#1{\text{Var}(#1)}
\textcolor{blue}{}\global\long\def\TDD#1{{\color{red}To\, Do(#1)}}
\textcolor{blue}{}\global\long\def\dd#1{\textnormal{d}#1}
\textcolor{blue}{}\global\long\def\eqdef{:=}
\textcolor{blue}{}\global\long\def\ddp#1#2{\left\langle #1,#2\right\rangle }
\textcolor{blue}{}\global\long\def\En{\mathcal{E}_{n}}
\textcolor{blue}{}\global\long\def\Z{\mathbb{Z}}
\textcolor{blue}{{} }

\textcolor{blue}{}\global\long\def\nC#1{\newconstant{#1}}
\textcolor{blue}{}\global\long\def\C#1{\useconstant{#1}}
\textcolor{blue}{}\global\long\def\nC#1{\newconstant{#1}\text{nC}_{#1}}
\textcolor{blue}{}\global\long\def\C#1{C_{#1}}
\textcolor{blue}{}\global\long\def\meas{\mathcal{M}}
\textcolor{blue}{}\global\long\def\cSpace{\mathcal{C}}
\textcolor{blue}{}\global\long\def\pspace{\mathcal{P}}

\begin{abstract}
In the space of polynomial maps of $\mathbb R^2$ of degree at least two, there are codimension $3$ laminations of maps with at least $3$ period doubling Cantor attractors. The leafs of the laminations are real-analytic and they have uniform diameter. The closure of each lamination contains the codimension one tangency locus of a saddle point. Asymptotically, the leafs of each lamination align with the leafs of the eigenvalue foliation.  This is an example of general coexistence theorems valid for higher dimensional real-analytic unfoldings of two dimensional homoclinic tangencies.
\end{abstract}

\section{Introduction}
In order to understand the long term behavior of a dynamical system, one possible approach is to study the set where a lot of orbits spend most of the time. This set is called the attractor of the system. Moreover, as soon as an attractor is detected, one would like to know for which other parameters a similar attractor occurs in a family of systems of the same type. This says how much and in which form an attractor is stable.
Attractive periodic orbits and hyperbolic attractors, for example, persist by changing parameters in an open set. They have the strongest form of stability.

We study here real-analytic two-dimensional unfoldings of maps with a strong homoclinic tangency, see Definition \ref{unfolding}. What makes such families special is that, like in most systems with frictions, the first return maps to a neighborhood near the homoclinic tangency, are close to H\'enon maps, see \cite{Berger,PT}. The set of parameters where this holds, consists of the so-called H\'enon strips. The main theorems rely on  a refinement of this fact. In the usual real setting, H\'enon-like maps are defined as smooth perturbation of a certain form of the H\'enon maps. In our real-analytic context, we refer to {\it H\'enon-like maps} as holomorphic perturbations, defined on a bi-disk in $\mathbb C^2$, of quadratic maps, i.e.
\begin{eqnarray*}
HF\left(\begin{matrix}
x\\y
\end{matrix}\right)
=\left(\begin{matrix}
x^2+\nu +\epsilon (x,y)\\
x
\end{matrix}\right).
\end{eqnarray*}
The precise definition is given in \eqref{eq:Henonlikemapsdef}. In particular, given a real-analytic unfolding, the first return maps are H\'enon-like maps and they have the following properties, see Theorem \ref{firstreturnmapanalytic}, Theorem \ref{prop:HFclosetoFa}, and Theorem \ref{prop:dftdt}.
\vskip .2 cm
{\it
\paragraph{\bf Theorem A.} Let $F_{t,a}$ be a real-analytic two-dimensional unfolding of the homoclinic tangency at $a=0$. Then, for any $n$ large enough, there exists a H\'enon strip $\mathcal H_n$ in parameter space, such that, for all $(t,a)\in\mathcal H_n$, the first return map $F^n_{t,a}$, to an appropriate domain, becomes a H\'enon-like map $HF_{t,a}$, after a holomorphic coordinate change. Furthermore, the parameter dependence is given by
\begin{eqnarray*}
\frac{\partial HF_{t,a}}{\partial a}&=&\left(L(t,a)+O\left(\frac{n}{\mu^n(t,a)}\right)\right)\mu^{2n}(t,a),\\
\frac{\partial HF_{t,a}}{\partial t}&=&\left(M(t,a)+O\left(\frac{1}{\mu^n(t,a)}\right)\right)n\mu^{n}(t,a),
\end{eqnarray*}
where $\mu(t,a)$ is the unstable eigenvalue of the saddle point and the constants $L(t,a)$ and $M(t,a)$ are real-analytic functions of the parameters. }
\vskip .2 cm
Theorem A shows a universal aspect of the parameter dependence of the return maps, namely a parameter change results in essentially a translation.  Theorem A will play a crucial role in finding maps in real-analytic unfoldings with period doubling Cantor attractors\footnote{See Definition \ref{cantorA}}. In fact, the manifold structure of the period doubling locus is only known in the  H\'enon-like setting, see  \cite{CLM}.

\bigskip

The study of part of the local dynamics in unfoldings is reduced to the study of the dynamics of H\'enon-like maps.  
In the H\'enon family and in other two-dimensional H\'enon-like families the following attractors has been detected:
\begin{itemize}
\item[-] there are maps, for an open set of parameters, having one attractive periodic orbit, a sink,
\item[-] there are maps, for a positive Lebesgue set of parameters, having a strange attractor\footnote{See Definition \ref{strange}}, see \cite{BC,MV},
\item[-] there are maps having a period doubling Cantor attractor and they form a smooth curve in the parameter space, see \cite{CLM}.
\end{itemize} 
A natural question is to ask if there are maps in the H\'enon family having two or more of these attractors simultaneously and in which form they are "observable". A strategy to find them is to start with a map having a periodic attracting orbit and try to find parameters in the open set where the periodic orbit persists which have also another attractor. This approach has been introduced originally by Newhouse in \cite{Newhouse}. One difference from this classical construction is that we need to make a selection of parameters. This can be a very sophisticated procedure, which has been carried out in for example \cite{BMP, BP, Berger, Bi, Bu, Ro} using different strategies. In \cite{BMP, Ro}, the authors find parameters in the H\'enon family corresponding to maps having either coexistence of periodic attractors (sinks), or periodic attractors and one non-periodic attractor. Here, using the method introduced in \cite{BMP}, we solve the more delicate problem to find parameters in finite dimensional unfoldings whose corresponding maps have multiple non-periodic attractors combined with sinks. Non-periodic attractors are much less stable than the periodic ones, they can be easily destroyed by changing parameters. Despite this we can prove that multiple period doubling Cantor attractors coexist along real-analytic leafs of a lamination. We summarize in the following statement Proposition \ref{pdanalytic}, Theorem \ref{theo:2PDattarctors} and Theorem \ref{theo:KPDlaminations2dimensional}.
\vskip .2 cm
{\it
\paragraph{\bf Theorem B.}
Let $M$, $\mathcal P$ and $\mathcal T$ be real-analytic manifolds and $F:\left(\mathcal P\times\mathcal T\right)\times M\to M$ be a real-analytic   family with $\text{dim}(\mathcal P)=2$. If there exists $\tau_0\in\mathcal T$ such that $F_0:\left(\mathcal P\times\left\{\tau_0\right\}\right)\times M\to M$ is an unfolding of a map $f_{\tau_0}$ with a strong homoclinic tangency, then for $k=1,2$, there exists a codimension $k$ lamination of maps with at least $k$ period doubling Cantor attractors which persist along the leafs. The homoclinic tangency persists along a global codimension one manifold in $\mathcal P\times \mathcal T$ and this tangency locus is contained in the closure of the lamination. Moreover, the leafs of the lamination are real-analytic and they have a uniform positive diameter when  $\text{dim}(\mathcal T)\geq 1$.}
\vskip .2 cm
 Theorem B states that, when  $\text{dim}(\mathcal T)=0$, there is a set of maps with two period doubling Cantor attractors which start to move creating a real-analytic lamination when  $\text{dim}(\mathcal T)\geq 1$.
Our method combined with the method in \cite{BMP}, allows also to find parameters in which finitely many sinks and multiple period doubling Cantor attractors coexist, see Theorem \ref{sinksand2PDattarctors} and Theorem \ref{theo:SsinksKperioddoubling}.

Furthermore, we find other coexistence phenomena persisting along codimension three laminations. In these cases we are able to give a description of the asymptotic direction of the leafs of the laminations. This occurs in the so called {\it saddle deforming} unfoldings, see Definition \ref{defn:sdunfoldings}. In saddle deforming  unfoldings the level sets of the eigenvalue pair of a saddle point define the codimension two {\it eigenvalue foliation} of the tangency locus. A saddle deforming unfolding is such that it contains a three dimensional subfamily transversal to this foliation. In these unfoldings, the leafs of each codimension three coexistence lamination align with the leafs of the eigenvalue foliation associated to a saddle point. 
The coexistence phenomena and their stability stated in Theorem \ref{theo:KPDlaminations} and Theorem \ref{theo:infsinksPDlaminations}  are summarized in the following.
\vskip .2 cm
{\it
\paragraph{\bf Theorem C.}
Let $M$, $\mathcal P$ and $\mathcal T$ be real-analytic manifolds and $F:\left(\mathcal P\times\mathcal T\right)\times M\to M$ be a saddle deforming unfolding, then the following holds:
\begin{itemize}
\item[-]  there exists a codimension $3$ lamination $3PD$ of maps with at least $3$ period doubling Cantor attractors which persist along the leafs,
\item[-]  there exists a codimension $3$ lamination $NHPD$ of maps with infinitely many sinks and at least $1$ period doubling Cantor attractor which persist along the leafs.
\end{itemize}
The leafs of the laminations are real-analytic and they have a uniform positive diameter when their dimension is at least one.  Moreover, the laminations align with the eigenvalue foliation. In particular, for each leaf of the eigenvalue foliation, there is a sequence of leafs of $3PD$ and a sequence of leafs of $NHPD$ which accumulate at this eigenvalue leaf.  }
\vskip .2 cm
Laminations of coexisting attractors have been found already in \cite{BMP}. Here we are able to find laminations of non-periodic attractors and to describe the asymptotic direction of the leafs. This reveals further universal and global aspects of the bifurcation pattern. 

We would like to stress that using Theorem A it is possible to find coexistence of $3$ coexisting period doubling Cantor attractors near points with a triple homoclinic tangency. The existence of triple homoclinic points is shown in \cite{Turaev}. Such a construction would give a rather small set of this parameters. An integral part of our method is the simultaneous creation of new attractors together with new tangencies. The process is given by an explicit algorithm. This allows for a precise geometrical description of the laminations. In particular, the asymptotic of the laminations is controlled by the eigenvalue foliation. 

Our method allows also to find coexistence of period doubling Cantor attractors and a strange one. In this case, we do not have laminations because the stability of the strange attractors in families with at least three parameters is not yet understood. Theorem \ref{theo:KPD1Strange} states the following.
\vskip .2 cm
{\it
\paragraph{\bf Theorem D.}
Let $M$, $\mathcal P$ and $\mathcal T$ be real-analytic manifolds and $F:\left(\mathcal P\times\mathcal T\right)\times M\to M$ be a saddle deforming unfolding, then the set of maps with at least $2$ period doubling Cantor attractors and one strange attractor has Hausdorff dimension at least $\text{dim}(\mathcal P\times\mathcal T)-2$.   }
\vskip .2 cm
The previous theorems and the method of the proofs apply in particular to the H\'enon family or any family of polynomial maps of $\mathbb R^2$ of degree at least $2$, see Theorem \ref{theo:Henonapplication} and Theorem \ref{theo:polyapplication}.

\paragraph{Acknowledgements}
The author was supported by the Trygger foundation, Project CTS 17:50, and partially by the NSF grant 1600503. 

\section{Preliminaries}\label{section:preliminaries}
In this section we collect definitions and facts needed in the sequel. A more elaborate exposition can be found in \cite{BMP}. The following well-known linearization result is due to Sternberg.
\begin{theo}\label{Ctlinearization}
Given $\left(\lambda,\mu\right)\in\R^{2}$, there exists $\kappa\in\N$ such that the following holds. 
Let $M$ be a two-dimensional real-analytic manifold and let $f:M\to M$ be a diffeomorphism with saddle point $p\in M$ having unstable eigenvalue $|\mu|>1$ and stable eigenvalue $\lambda$. If
\begin{equation}\label{nonresonance}
\lambda\neq\mu^{k_1} \text{ and } \mu\neq\lambda^{k_2}
\end{equation}
for $k=\left(k_1,k_2\right)\in\N^{2}$ and $2\leq |k|=k_1+k_2\leq \kappa$, then $f$ is $\Cq$ linearizable.
\end{theo}
\begin{defin}
Let $M$ be a two-dimensional real-analytic manifold and let $f:M\to M$ be a diffeomorphism with a saddle point $p\in M$. We say that $p$ satisfies the $\Cq$ non-resonance condition if (\ref{nonresonance}) holds.
\end{defin}

\begin{theo}\label{familydependence}
Let $M$ be a two-dimensional real-analytic manifold and let $f:M\to M$ be a diffeomorphism with a saddle point $p\in M$ which satisfies the $\Cq$ non-resonance condition. Let $0\in \mathcal P\subset\R^{n}$ and let $F:M\times \mathcal P\to M$ be a real-analytic family with $F_0=f$. Then, there exists a neighborhood $U$ of $p$ and a neighborhood $V$ of $0$ such that, for every $t\in V$, $F_t$ has a saddle point $p_t\in U$ satisfying the $\Cq$ non-resonance condition. Moreover $p_t$ is $\Cq$ linearizable in the neighborhood $U$ and the linearization depends $\Cq$ on the parameters. 
\end{theo}
The proofs of Theorem \ref{Ctlinearization} and Theorem \ref{familydependence} can be found in \cite{BrKo, IlaYak}.
In the sequel we introduce the concept of a map with a strong homoclinic tangency which appears already in \cite{BMP}. This is a map on a two-dimensional manifold with a saddle point, whose eigenvalues satisfy a contraction condition as in $(f2)$. Moreover, a map with a strong homoclinic tangency also has a non-degenerate homoclinic tangency and a transversal homoclinic intersection satisfying $(f6)$, $(f7)$ and $(f8)$, see Figure \ref{Fig1}. All conditions defining a map with a strong homoclinic tangency are open in the space of maps with an homoclinic and transversal tangency. Also, except for $(f2)$, all conditions are dense. An example is the H\'enon family which contains maps with a strong homoclinic tangency. A map with these properties will then be "unfolded" to create a two-dimensional family.

\begin{figure}[h]
\centering
\includegraphics[width=0.6\textwidth]{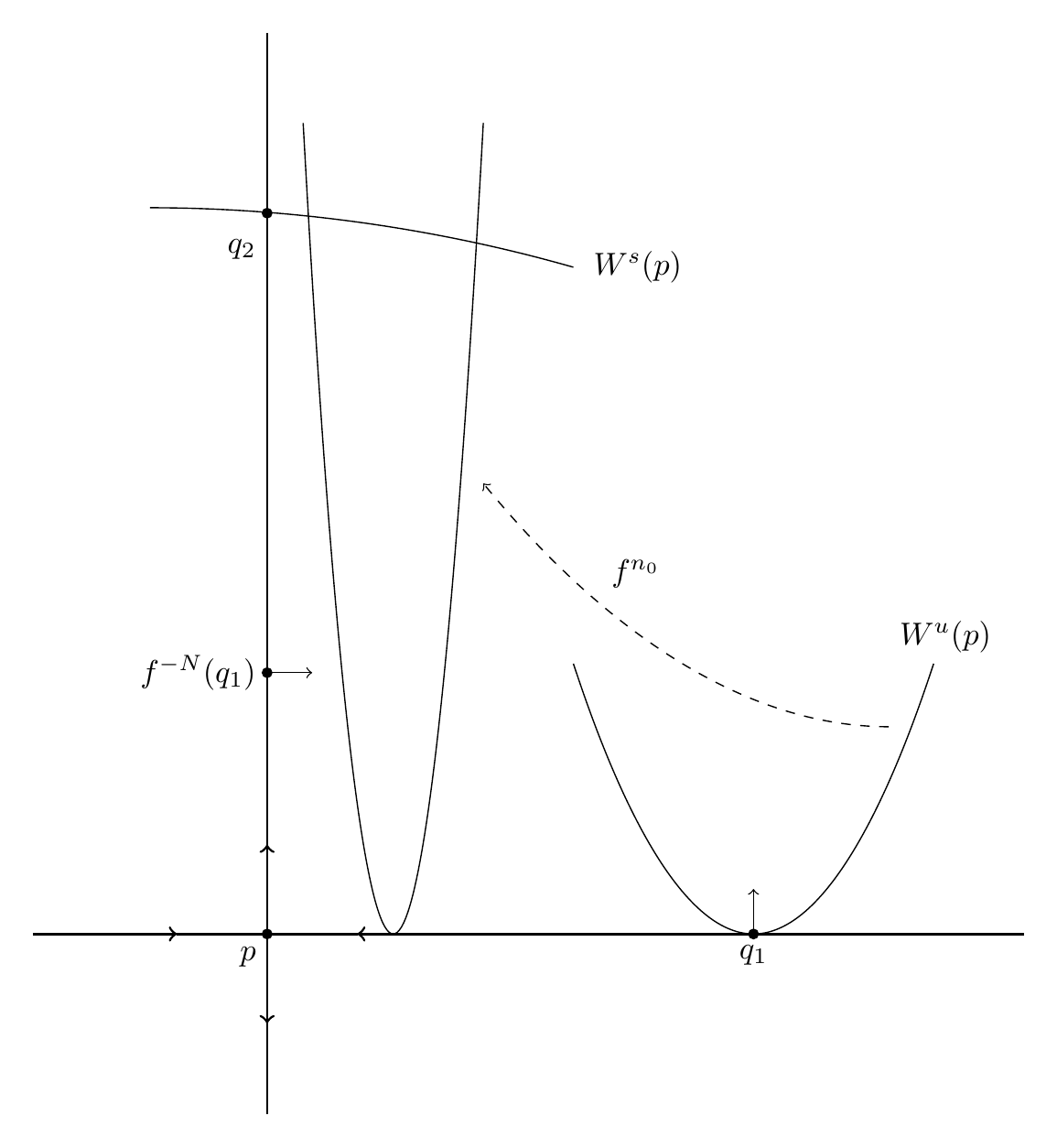}
\caption{A map with a strong homoclinic tangency}
\label{Fig1}
\end{figure}

\begin{defin}\label{stronghomtang}
Let $M$ be a two-dimensional real-analytic manifold and let $f:M\to M$ be a local diffeomorphism satisfying the following conditions:
\begin{itemize}
\item[$(f1)$] $f$ has a saddle point $p$ with unstable eigenvalue $|\mu|>1$ and stable eigenvalue $\lambda$,
\item[$(f2)$] $|\lambda||\mu|^3<1$,
\item[$(f3)$] $p$ satisfies the $\Cq$ non-resonance condition,
\item[$(f4)$] $f$ has a non degenerate homoclinic tangency, $q_1\in W^u(p)\cap W^s(p)$, 
\item[$(f5)$] $f$ has a transversal homoclinic intersection,  $q_2\in W^u(p)\pitchfork W^s(p)$,
\item[$(f6)$] let $[p,q_2]^u\subset W^u(p)$ be the arc connecting $p$ to $q_2$, then there exist arcs  $W^u_{\text{\rm loc},n}(q_2)=[q_2, u_n]^u\subset W^u(q_2)$ such that $[p,q_2]^u\cap [q_2,u_n]^u=\left\{q_2\right\}$ and 
$$
\lim_{n\to\infty}f^n\left(W^u_{\text{\rm loc},n}(q_2)\right)=[p,q_2]^u,
$$
\item[$(f7)$] there exist neighborhoods $W^u_{\text{\rm loc},n}(q_1)\subset W^u(q_1)$ such that 
$$
\lim_{n\to\infty}f^n\left(W^u_{\text{\rm loc},n}(q_1)\right)=[p,q_2]^u,
$$
\item[$(f8)$] there exists $N\in\N$ such that 
$$
f^{-N}(q_1)\in [p,q_2]^u.
$$
\end{itemize}

A map $f$ with these properties is called a map with a \rm{strong homoclinic tangency}, see Figure \ref{Fig1}.
\end{defin}

\begin{rem} If the unstable eigenvalue is negative, $\mu<-1$, then $(f6)$, $(f7)$, and $(f8)$ are redundant.
\end{rem}

Following \cite{PT}, we define an unfolding of a map $f$ with a strong homoclinic tangency. This is a family created by "adding" to $f$ two parameters. Details follow.
Let  $\mathcal P=[-r,r]^2$ with $r>0$.
Given a map $f$ with a strong homoclinic tangency, we consider a real-analytic family $F:\mathcal P\times M\to M$ trough $f$ with the following properties:
\begin{itemize}
\item[$(F1)$] $F_{0,0}=f$,
\item[$(F2)$] $F_{t, a}$ has a saddle point $p(t, a)$ with unstable eigenvalue $|\mu(t, a)|>1$, with stable one $\lambda(t,a)$ and
$$
\frac{\partial \mu}{\partial t}\ne 0 ,
$$
\item[$(F3)$] let 
$\mu_{\text{max}}=\max_{(t,a)}|\mu(t,a)|$,
 $\lambda_{\text{max}}=\max_{(t,a)}|\lambda(t,a)|$ and assume $$\lambda_{\text{max}}\mu_{\text{max}}^3<1,$$
\item[$(F4)$] there exists a real-analytic function $[-r,r]\ni t\mapsto q_1(t)\in W^u(p(t,0))\cap W^s(p(t,0))$ such that $q_1(t)$ is a non degenerate homoclinic tangency of $F_{t,0}$.
\item[$(F5)$] there exists a real-analytic function $[-r,r]^2\ni (t,a)\mapsto q_2(t,a)\in W^u(p(t,a))\cap W^s(p(t,a))$ such that $q_2(t,a)$ is a transversal homoclinic intersection of $F_{t,a}$.
\end{itemize}
According to Theorem \ref{familydependence} we may make a change of coordinates to ensure  that the family $F$ is $\Cq$ and that, for all $(t,a)\in [-r_0,r_0]^{2}$ with $0<r_0<r$, $F_{t,a}$ is linear on the ball $[-2,2]^2$, namely $$F_{t,a}=\left(\begin{matrix}
\lambda(t,a)&0\\
0&\mu(t,a)\\
\end{matrix}\right).$$ 
Moreover, the saddle point $p(t,a)=(0,0)$ and the local stable and unstable manifolds satisfy:
\begin{itemize}
\item[-] $W^s_{\text{loc}}(0)=[-2,2]\times \left\{0\right\}$,
\item[-] $W^u_{\text{loc}}(0)=\left\{0\right\}\times [-2,2]$ ,
\item[-] $q_1(t)\subset W^s_{\text{loc}}(0)$,
\item[-] $q_2(t,a)\in \left\{0\right\}\times \left(\frac{1}{\mu},1\right)\subset W^u_{\text{loc}}(0)$.
\end{itemize}
Consider the function $y\mapsto \left(F_{t,a}^{N}(0,y)\right)_y$. Let $q_3(t,a)\in \left\{0\right\}\times(0,2)\subset W^u_{\text{loc}}(0)$ be the solution of  
$$\frac{\partial \left(F_{t,a}^{N}(0,y)\right)_y}{\partial y}=0$$
with $F_{t,a}^{N}(q_3(t,0))=q_1(t)$. Observe that $q_3$ depends in a $\Ct$ manner on the parameters. Without loss of generality we may assume that $q_{3}(t,a)=(0,1)$. We define $q_1(t,a)=F_{t,a}^{N}(q_3(t,a))$
and it points in the positive $y$ direction.
\begin{defin}\label{unfolding}
A family $F_{t,a}$ is called an \emph{unfolding} of $f$ if it can be reparametrized such that
$$
\frac{\partial (q_1(t,a))_y}{\partial a}\neq 0.
$$
\end{defin}
\begin{rem}
A generic two-dimensional family trough $f$ can locally be reparametrized to become an unfolding.
\end{rem}

\section{H\'enon-like normalization}\label{sec:firstreturnmap}
In the sequel we prove that, given a real-analytic unfolding $F$, then the first return maps to an appropriately chosen domain can be {\it straightened} to obtain a family consisting of {\it H\'enon-like maps}, see \eqref{eq:Henonlikemapsdef}. The straightened first return maps are holomorphic maps defined on a uniform domain in $\mathbb C^2$ and they are arbitrarily close to degenerate H\'enon maps. The idea of the proof is inspired by the classical construction contained for example in \cite{Berger, PT}. However, in order to apply the main theorem in \cite{CLM}, we need first return maps with uniform holomorphic extensions. For this reason and for the study of the parameter dependence we include the following self-contained construction.

Consider a real-analytic unfolding $F_{t,a}$ of a strong homoclinic tangency, say with $(t,a)\in (-1,1)\times (-1,1)$. We may assume that this family extends to an holomorphic family of the form
$$\mathbb{D}\times \mathbb{D}\ni (t,a)\mapsto F_{t,a}:U\times U\to\mathbb C^2,
$$
where $U$ is a domain in $\mathbb C$.
There is a local holomorphic change of coordinates such that the saddle point becomes $(0,0)$ and the local stable manifold contains,  in the $x$-axis, the disc of radius $4$ centered around $0$, denoted by $\mathbb D_4$. Similarly, the local unstable manifold contains the disc $\mathbb D_4$ in the $y$-axis. Moreover, the restriction of each map to the invariant manifolds is linearized, that is 
\begin{equation*}\label{semilinearization}
F_{t,a}(x,0)=(\lambda x,0) \text{ and } F_{t,a}(0,y)=(0, \mu y).
\end{equation*}
From now on we will study the map $F_{t,a}^{n+N} $ in this new coordinates. The following lemma is stated and proved in \cite{BMP}, see Lemma 17.
\begin{lem}\label{DFnC} If $(x,y)\in \mathbb D_4\times \mathbb{D}_4$ and $F^i(x,y)\in \mathbb{D}_4\times \mathbb{D}_4$, for $i\le n$, then
$$
DF^n(x,y)=\left(\begin{matrix}
a_{11} \lambda_1^n \mu^n& a_{12} \lambda_1^n \mu^n\\
a_{21} & a_{22} \mu^n
\end{matrix}\right),
$$
where $a_{kl}$ are uniformly bounded holomorphic functions and $a_{22}\ne 0$ and uniformly away from zero.
\end{lem}
From $(6.17)$ in \cite{BMP}, there exists an holomorphic function $\mathbb D\ni t\mapsto sa_n(t)$ such that, $|sa_n(t)|=O\left(1/\mu^n\right)$ and in the parameter $(t, sa_n(t))$ the periodic point $p_t$, called the strong sink of period $n+N$, has trace zero. 
Choose $E>0$ and define, for $n$ large enough, the {\it $n^{th}$ H\'enon strip} by
\begin{equation}\label{Hnstrip}
\mathcal{H}_n=\left\{(t,a)\in [-t_0,t_0]\times [-a_0,a_0] \left|\right. |a-sa_n(t)|\leq\frac{E}{|\mu(t,sa_n(t))|^{2n}}\right\}.
\end{equation}
The precise choice of $E$ will be made in \eqref{eq:Echoice}.
For $(t, a)\in\mathcal{H}_n$ the coordinates of the strong sink are given by $p_t=(p_{x}(t),p_{y}(t))$. Without loss of generality we may assume that $p_x(t)=2$ for all $t\in\mathbb D$. 

\subsection{The straightened map}
Choose $(t,a)\in\mathcal{H}_n$. In this section we apply a coordinate change to the restriction of $F_{t,a}^{n+N}$ to a relevant domain such that, the resulting map resembles a H\'enon map in the sense that vertical lines go to horizontal lines, see \eqref{eq:straightnedmap}. This map is called the {\it straightened} map. First we construct the coordinate change $\sigma_{t,a}$, called the {\it straightening} map, defined on the domain $HB(t,a)$, called the  {\it straightening box}. Let $\mathbb D_1\subset\mathbb C$ be the disk in the complex plane centered around $p_x(t)=2$. Consider the vertical foliation in $\mathbb D_1\times\mathbb C$. The domain of the {straightening map} $$\sigma_{t,a}:HB(t,a)\to\mathbb D_1\times\mathbb D_1$$ is the {straightening box} $$HB(t,a)=\left(F_{t,a}^{n+N}\right)^{-1}\left(\mathbb D_1\times\mathbb C\right)\cap \left(\mathbb D_1\times\mathbb C\right)$$
and it is defined by 
$$
\sigma_{t,a}(x,y)=\left(\left(F_{t,a}^{n+N}(x,y)\right)_x,x\right)
$$
where $\left(F_{t,a}^{n+N}(x,y)\right)_x$ denotes the $x$ component of $F_{t,a}^{n+N}(x,y)$. Refer to Figure \ref{fig:HBbox}
\begin{rem}\label{rem:verthorizfoliations}
Observe that the box $HB(t,a)$ has two natural foliations. Namely, one given by $\mathbb D_1\times\mathbb C$, called the {\it vertical foliation} and the other given by  $\left(F_{t,a}^{n+N}\right)^{-1}\left(\mathbb D_1\times\mathbb C\right)$, called the {\it almost horizontal foliation}. The straightening map turns the vertical foliation into the horizontal foliation and the almost horizontal foliation into the vertical foliation.
\end{rem}

\begin{figure}[h]
\centering
\includegraphics[width=0.9\textwidth]{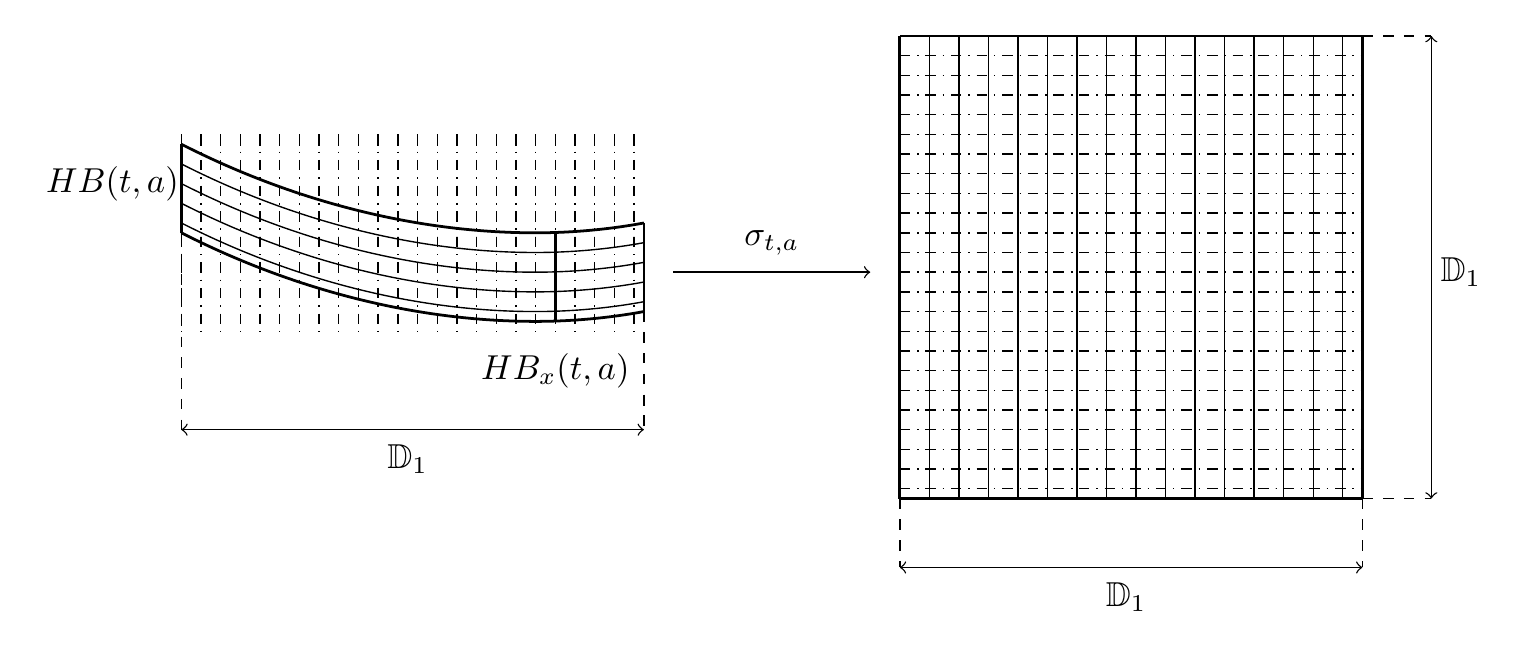}
\caption{Straightening box and straightening map}
\label{fig:HBbox}
\end{figure}

\begin{lem}\label{lem:hbx}
For every $x\in\mathbb D_1$, $$HB_x(t,a)=HB(t,a)\cap
\left(\left\{x\right\}\times\mathbb C\right)$$ is a simply connected domain of diameter proportional to $1/\mu^n(t,a)$. If $y\in HB(t,a)$ then $y$ is proportional to $1/\mu^n(t,a)$. Moreover the straightening map $\sigma_{t,a}$ is a biholomorphism.
\end{lem}
\begin{proof}
 By adjusting the radius of the disc $\mathbb D_1$, we may assume that there exists a simply connected domain $D$ in the $y$-axis, containing $q_3(t,0)$, such that $\left(F_{t,0}^N(D)\right)_x$ contains strictly $\mathbb D_1$. Moreover the projection $\pi_x:F^N_{t,0}(D)\to\mathbb C$ is univalent. Let $x\in\mathbb D_1$. Then $F^n_{t,a}\left(\left\{x\right\}\times\mathbb C\right)$ contains a simply connected domain exponentially close to $D$ in the $y$-axis. In particular $HB_x(t,a)$ is a simply connected domain with the projection $$\pi_x:F^{n+N}_{t,a}\left(HB_x(t,a)\right)\to\mathbb D_1$$ is univalent and onto. The diameter estimate for $HB_x(t,a)$ follows from Lemma \ref{DFnC}. Observe that 
 $$
 \sigma_{t,a}^{-1}(x,y)=\left(y,\left(\left(F^{n+N}_{t,a}\right)^{-1}\left(\left\{x\right\}\times\mathbb C\right)\cap \left(\left\{y\right\}\times\mathbb C\right)\right)_y\right).
 $$
 In particular the straightening map $\sigma_{t,a}$ is a biholomorphism.
\end{proof}
\begin{rem}
The construction of the straightening box implies that for every $t\in\mathbb D$, $p(t)\in HB\left(t,sa_n(t)\right)$.
\end{rem}
Let $\pi_x$ and $\pi_y$ be the orthogonal projections to the axes. The following holds. 
\begin{lem}\label{lem:intersectionnumber}
Let $\gamma_1:\mathbb D\to\mathbb C^2$ and $\gamma_2:\mathbb D\to\mathbb C^2$ be holomorphic discs such that 
\begin{itemize}
\item[-] $\pi_x\left(\gamma_1(\mathbb D)\right)\subset \mathbb D$,
\item[-] $\pi_y\circ\gamma_1:\mathbb D\to \mathbb D$ is a $2$-to-$1$ covering map,
\item[-] $\pi_x\circ\gamma_2:\mathbb D\to \mathbb D$ is univalent
\end{itemize}
then for $i=1,2$,
$$
\#\left(\gamma_i^{-1}\left(\gamma_1(\mathbb D)\cap \gamma_2(\mathbb D)\right)\right)\in\left\{1,2\right\}.
$$
In particular, the intersection number is $1$ if and only if the two discs have a tangency. 
\end{lem}
\begin{proof}
Observe that the intersection number is an homotopy invariant. Without loss of generality we may assume that $\pi_x\left(\gamma_1(\mathbb D)\right)=\left\{0\right\}$ and $\pi_y\left(\gamma_2(\mathbb D)\right)=\left\{0\right\}$. The lemma follows.
\end{proof}
\begin{lem}\label{lem:piyFN}
For every $(t,a)\in\mathcal{H}_n$ there exists a simply connected domain $D_{t,a}$ in the $y$-axis such that the map 
$$
\pi_y\circ F^N_{t,a}:D_{t,a}\to\mathbb D
$$
is a $2$-to-$1$ covering map.
\end{lem}
\begin{proof}
When $a=0$, the fact that $F_{t,0}^N(q_3(t,0))=q_1(t,0)$ where $q_1(t,0)$ is the homoclinic tangency, the existence of a disc $D_{t,0}$ is assured. The parameter continuity implies the existence of the disc $D_{t,a}$ when $n$ is large enough.
\end{proof}

\begin{lem}
For every $(t,a)\in\mathcal{H}_n$ and $x\in\mathbb D_1$, 
$$
\pi_x\circ\sigma_{t,a}\circ F_{t,a}^{n+N}:D_x=\left(F_{t,a}^{n+N}\right)^{-1}\left(HB(t,a)\right)\cap HB_x(t,a)\to\mathbb D_1 
$$
is a $2$-to-$1$ covering map. In particular, either $D_x$ is simply connected or $D_x=D_x^-\cup D_x^+$ is the union of two simply connected domains.
\end{lem}
\begin{proof}
Let $x\in\mathbb D_1$ and consider $F^n_{t,a}:\left\{x\right\}\times\mathbb D\to\mathbb C^2$. Because of Lemma \ref{DFnC} and Lemma \ref{lem:piyFN}, there exists a disc $ D\subset \left\{x\right\}\times\mathbb D$ such that $F_{t,a}^n{\left( D\right)}$ is $(\lambda\mu)^n$ close to $D_{t,a}$ and the projection 
$$
\pi_y\circ F^{n+N}_{t,a}: D\to\mathbb D
$$
is a $2$-to-$1$ covering map. Let $x_1\in\mathbb D_1$. Observe that each point $(x,y)\in D$ with $$\left(\pi_x\circ\sigma_{t,a}\circ F_{t,a}^{n+N}\right)(x,y)=x_1$$ corresponds to an intersection point of $F_{t,a}^{n+N}(D)$ with $ F_{t,a}^{-(n+N)}\left(\left\{x_1\right\}\times\mathbb C\right)$. Hence, from Lemma \ref{lem:intersectionnumber}, we get 
$$
\#\left\{(x,y)\in D\left|\right.\left(\pi_x\circ\sigma_{t,a}\circ F_{t,a}^{n+N}\right)(x,y)=x_1\right\}\in\left\{1,2\right\}.
$$
Moreover, observe that 
$$
HB_x(t,a)=\left\{(x,y)\in D\left|\right.\left(\pi_x\circ\sigma_{t,a}\circ F_{t,a}^{n+N}\right)(x,y)\in\mathbb D_1\right\}.
$$
The lemma follows.
\end{proof}
\begin{lem}\label{lem:modDxHBx}
For every $(t,a)\in\mathcal{H}_n$ and $x\in\mathbb D_1$, the domain $D_x$ satisfies 
$$
\text{\rm mod}\left(D_x, HB_x(t,a)\right)\geq \frac{n}{4\pi}\log\mu(t,a)+O(1).
$$
\end{lem}
\begin{proof}
Choose $x\in\mathbb D_1$ and consider the map $$
\pi_y\circ F^{n+N}_{t,a}: HB_x(t,a)\to\mathbb C.
$$ The image of this map contains a disk of definite radius centered around zero. We may assume that it contains the disc of radius $2$. Moreover, the map has a unique critical point and the corresponding critical value is at distance $O\left(1/\mu^n\right)$ to zero, using that $(t,a)\in\mathcal H_n$. Hence, there is a disc $D\subset HB_x(t,a)$ such that 
\begin{equation}\label{eq:piyfn2to1}
\pi_y\circ F^{n+N}_{t,a}: D\to\mathbb D
\end{equation}
is a $2$-to-$1$ covering map. 
By Lemma \ref{lem:hbx}, there exists a positive constant $K$, such that
$$
\pi_y\left(HB(t,a)\right)\subset K_n
$$
where $K_n$ is the disc of radius $K/\mu^n$. Observe that $$D_x\subset \left(\pi_y\circ F^{n+N}_{t,a}\right)^{-1}\left(K_n\right)\cap HB_x(t,a).$$ The lemma follows from \eqref{eq:piyfn2to1}.
\end{proof}
We define now the domain of the straightened map $\tilde F_{t,a}$ as
$$
\text{Dom}(\tilde F_{t,a})=\sigma_{t,a}\left(\cup_{x\in\mathbb D_1}D_x\right)=\sigma_{t,a}\left(\left(F_{t,a}^{n+N}\right)^{-1}\left(HB(t,a)\right)\cap HB(t,a)\right)\subset \mathbb D_1\times\mathbb D_1
$$
where 
\begin{equation}\label{eq:straightnedmap}
\tilde F_{t,a}=\sigma_{t,a}\circ F_{t,a}^{n+N}\circ\sigma_{t,a}^{-1}:\text{Dom}(\tilde F_{t,a})\to\mathbb D_1\times\mathbb D_1.
\end{equation}
By construction 
$$
\tilde F_{t,a}\left(\begin{matrix}
x\\y
\end{matrix}\right)
=\left(\begin{matrix}
g_{t,a}(x,y)\\
x
\end{matrix}\right)
$$
where, for each $y\in \mathbb D_1$, the map $x\to g_{t,a}(x,y)$ is $2$-to-$1$ onto $\mathbb D_1$.
\subsection{Local existence of the critical point}
We introduce now the critical point of the return map $F_{t,a}^{n+N}:HB(t,a)\to\mathbb R^2$. This point has the same meaning as the critical point of the one dimensional part of the H\'enon map, i.e. $x\mapsto a+x^2$, see \eqref{Henonmap} In this subsection we prove the local existence of the critical point.
\begin{defin}\label{def:criticalpointofF}
A point $c=(c_x,c_y)\in HB(t,a)$ is a critical point of the map $
 F^{n+N}_{t,a}$ with critical value $v=(v_x,v_y)= F^{n+N}_{t,a}(c)\in HB(t,a)$ if $v_x=c_x$ and $F^{n+N}_{t,a}(HB_{c_x}(t,a))$ is tangent to the leaf of the almost horizontal foliation passing trough $v$. Namely, 
 \begin{enumerate}
\item $\pi_x\circ  F^{n+N}_{t,a}(c_x,c_y)=c_x$,
\item $\frac{\partial}{\partial y}\left(\pi_x\circ \sigma_{t,a}\circ  F^{n+N}_{t,a}\right)(c_x,c_y)=0$.
\end{enumerate}
\end{defin}
Observe that, for $a=sa_n(t)$, the strong sink $p(t)$ is a critical point. In the next proposition we prove that the strong sink persists as critical point in a neighborhood of the strong sink locus in $\mathcal H_n$.
\begin{prop}\label{prop:existenceofcriticalpoint}
For $n$ large enough, there exist a maximal neighborhood $\mathcal C_n\subset \mathcal H_n$ of the real part of the graph of the function $sa_n$ and a real analytic function $c:\mathcal C_n\to\mathbb R^2$ such that, for all $(t,a)\in \mathcal C_n$, the point $c(t,a)$ is a critical point of $F^{n+N}_{t,a}$. 
\end{prop}
The proof of this proposition needs some preparation. Fix $(t,a)\in\mathcal H_n$.
Let $h_{t,a}$ be a $\Cq$ diffeomorphism which locally linearizes the saddle point $(0,0)$ of $F_{t,a}$ and let $\hat F_{t,a}=h_{t,a}\circ F_{t,a}\circ h_{t,a}^{-1}$ be the corresponding smooth unfolding of a strong homoclinic tangency. 
\begin{lem}\label{lem:Dh}
 The coordinate change $h_{t,a}$ restricted to $HB(t,a)\cap\mathbb R^2$ satisfies
\begin{equation}\label{eq:Dhta}
Dh_{t,a}(x,y)=\left(
\begin{matrix}
1+{\alpha}/{\mu^n(t,a)}& {\beta}\\
{\gamma}/{\mu^n(t,a)}&1+\delta_x\left( x-2\right)+{\delta_y}/{\mu^n(t,a)}
\end{matrix}
\right)
\end{equation}
where $(x,y)\in HB(t,a)\cap\mathbb R^2$ and $\alpha, \beta,\gamma,\delta_x,\delta_y$ are $\Ct$ functions. 
\end{lem}
\begin{proof}
Because $h$ is $\Cq$ and preserves the $x$ and $y$ axis we have that $h(x,0)$ and $h(0,y)$ are linear. Hence, without loss of generality we may assume that $h(x,0)=x$. By rescaling in the $y$ direction we may assume that that $\partial (h)_y/\partial y(2,0)=1$. Observe that, using that $h$ preserves the $y$-axis $$\left(h\right)_x=x+y( x\varphi_{1,x}+y\varphi_{1,y}),$$ where $\varphi_{1,x}$ and $\varphi_{1,y}$ are $\Cuno$ functions. It follows that 
$$
\frac{\partial \left(h\right)_x}{\partial x}=1+y\left(\varphi_{1,x}+x\frac{\partial \varphi_{1,x}}{\partial x}+y\frac{\partial \varphi_{1,y}}{\partial x}\right)=1+\frac{\alpha}{\mu^n},
$$
where we used that, because $(x,y)\in HB(t,a)$,  $y$ is proportional to $1/\mu^n$. Moreover,
$$
\frac{\partial \left(h\right)_x}{\partial y}=x\varphi_{1,x}+y\varphi_{1,y}+y\left(x\frac{\partial \varphi_{1,x}}{\partial y}+\varphi_{1,y}+y\frac{\partial \varphi_{1,y}}{\partial y}\right)=\beta.
$$
Write now the Taylor expansion of $\left(h\right)_y$ centered around the point $(2,0)$. Then 
$$\left(h\right)_y=y( 1+x\varphi_{2,x}+y\varphi_{2,y}),$$
where $\varphi_{2,x}$ and $\varphi_{2,y}$ are $\Cuno$ functions. 
It follows that,
$$
\frac{\partial \left(h\right)_y}{\partial x}=y\left(\varphi_{2,x}+x\frac{\partial \varphi_{2,x}}{\partial x}+y\frac{\partial \varphi_{2,y}}{\partial x}\right)=\frac{\gamma}{\mu^n}.
$$
Finally,
$$
\frac{\partial \left(h\right)_y}{\partial y}=1+x\varphi_{2,x}+y\varphi_{2,y}+y\left(x\frac{\partial \varphi_{2,x}}{\partial y}+\varphi_{2,y}+y\frac{\partial \varphi_{2,y}}{\partial y}\right)=1+\delta_xx+\frac{\delta_y}{\mu^n},
$$
where we used again that $y$ is proportional to $1/\mu^n$.
\end{proof}
\begin{rem}\label{rem:alphabetagammadeltasmall}
Because $Dh(0)=\text{id}$, by shrinking the domain of linearization we may assume that  $\alpha, \beta,\gamma,\delta_x,\delta_y$ are sufficiently small. 
\end{rem}
To describe the critical point incorporating the linearized map $\hat F$ we 
consider the function $\Phi_{t,a}:HB(t,a)\cap\mathbb R^2\to\mathbb R^2$ defined by
$$
\Phi_{t,a}(x,y)=\left(
\left(\pi_x\circ h^{-1}\circ \hat F^{n+N}_{t,a}\circ h\right)(x,y)-x,
\frac{\partial}{\partial y}\left(\pi_x\circ h^{-1}\circ \hat F^{n+N}_{t,a}\circ\hat F^{n+N}_{t,a}\circ h\right)(x,y)
\right).
$$
Observe that $c$ is a critical point of $F_{t,a}$ if and only if $\Phi_{t,a}(c)=0$. Observe that if $p(t)$ is the strong sink of $F_{t,sa_n(t)}$, then $p(t)$ is a critical point and in particular $\Phi_{t,sa_n(t)}(p(t))=0$. By applying the implicit function theorem we will prove that the solution $p(t)$ can be extended over a domain of the form $\mathcal C_n$ as stated in Proposition \ref{prop:existenceofcriticalpoint}. In order to apply the implicit function theorem, we need to prove that $D\Phi_{t,a}$ is non singular in any critical point. The following notation and lemmas are required.

Let $(x,y)\in HB(t,a)\cap\mathbb R^2$ and 
\begin{itemize}
\item[-] $\underline x=(x,y)$ with $Dh_{t,a}(\underline x)=\left(
\begin{matrix}
1+{\alpha}/{\mu^n}& {\beta}\\
{\gamma}/{\mu^n}&\delta
\end{matrix}
\right)$,
\item[-] $\underline x_1=h_{t,a}\left(\underline x\right)$,
\item[-] $\underline x_2=\hat F_{t,a}^n\left(\underline x_1\right)$ with $D\hat F^N_{t,a}(\underline x_2)=\left(
\begin{matrix}
A_2&B_2\\
C_2&D_2
\end{matrix}
\right)$,
\item[-]$\underline x_3=\hat F_{t,a}^N\left(\underline x_2\right)$ with $Dh^{-1}_{t,a}(\underline x_3)=\left(
\begin{matrix}
1+{\alpha_3}/{\mu^n}& {\beta_3}\\
{\gamma_3}/{\mu^n}&{\delta_3}
\end{matrix}
\right)$,
\item[-]$\underline x_4=\hat F_{t,a}^n\left(\underline x_3\right)$  with $D\hat F^N_{t,a}(\underline x_4)=\left(
\begin{matrix}
A_4&B_4\\
C_4&D_4
\end{matrix}\right)$ ,
\item[-]$\underline x_5=\hat F_{t,a}^N\left(\underline x_4\right)$ with $Dh^{-1}_{t,a}(\underline x_5)=\left(
\begin{matrix}
1+{\alpha_5}/{\mu^n}& {\beta_5}\\
{\gamma_5}/{\mu^n}& {\delta_5}
\end{matrix}
\right)$.
\end{itemize}
\begin{rem}\label{rem:deltaBawayfromzero}
By Lemma \ref{lem:Dh} and Remark \ref{rem:verthorizfoliations},  $\delta$, $\delta_3$ and $\delta_5$ are non zero and close to one. Observe that, by the definition of unfolding $$D\hat F^N_{t,0}(q_3(t,0))=\left(\begin{matrix}A&B\\C&0\end{matrix}\right)$$ with $B$ away from zero. Because $\underline x_2$ and $\underline x_4$ are near $q_3$, without loss of generality we may assume that $B_2$ and $B_4$ are away from zero. 
\end{rem}
 Observe that all coefficients introduced, $\alpha,\beta,\dots,A_2,B_2,\dots,\alpha_3,\beta_3,\dots,A_4,B_4,\dots,\alpha_5,\dots$ are functions of the point $(x,y)\in HB(t,a)\cap\mathbb R^2$. Moreover we will use the notation 
$$
D\hat F^N_{t,a}(\underline x)=\left(
\begin{matrix}
A(\underline x,t,a)&B(\underline x,t,a)\\
C(\underline x,t,a)&D(\underline x,t, a)
\end{matrix}
\right).
$$
The following lemmas will give estimates on these coefficients and their partial derivatives.
\begin{lem}\label{lem:xpartialderivatives}
The partial derivatives at $\underline x=(x,y)\in HB(t,a)$ satisfy the following.
\begin{itemize}
\item[-]$\partial\alpha/\partial x,\partial \beta/\partial x,\partial\gamma/\partial x,\partial \delta/\partial x=O(1)$,
\item[-]$\partial A_2/\partial x,\partial B_2/\partial x,\partial C_2/\partial x=O(1)$,
\item[-]$\partial A_4/\partial x,\partial B_4/\partial x,\partial C_4/\partial x,\partial D_4/\partial x=O(\mu^n(t,a))$,
\item[-]$\partial\alpha_5/\partial x,\partial \beta_5/\partial x,\partial\gamma_5/\partial x,\partial \delta_5/\partial x=O(\mu^n(t,a)))$.
\end{itemize}
\end{lem}
\begin{proof}
The first set of estimates follows by the smoothness of $h_{t,a}$. For the second set of estimates, observe that 
$$
\frac{\partial\underline x_2}{\partial x}=\left(\begin{matrix}
\lambda^n& 0\\
0&\mu^n
\end{matrix}
\right)Dh_{t,a}(\underline x)\left(\begin{matrix}
1\\
0
\end{matrix}
\right)=\left(\begin{matrix}
\lambda^n+\alpha\left(\lambda/\mu\right)^n\\
\gamma
\end{matrix}
\right)=O(1).
$$
Moreover, 
\begin{equation}\label{eq:dx3dx}
\frac{\partial\underline x_3}{\partial x}=D\hat F^N_{t,a}(\underline x_2)\frac{\partial\underline x_2}{\partial x}=\left(\begin{matrix}
B_2\gamma+A_2\lambda^n+A_2\alpha\left(\lambda/\mu\right)^n\\
D_2\gamma+C_2\lambda^n+C_2\alpha\left(\lambda/\mu\right)^n
\end{matrix}
\right)=O(1).
\end{equation}
The third set of estimates follows from the fact that, 
$$
\frac{\partial\underline x_4}{\partial x}=\left(\begin{matrix}
\lambda^n& 0\\
0&\mu^n
\end{matrix}
\right)\frac{\partial\underline x_3}{\partial x}=\left(\begin{matrix}
B_2\gamma\lambda^n+A_2\lambda^{2n}+A_2\alpha\left(\lambda^2/\mu\right)^n\\
D_2\gamma\mu^n+C_2\left(\lambda\mu\right)^n+C_2\alpha\lambda^n
\end{matrix}
\right)=O(\mu^n).
$$
Finally,
\begin{eqnarray}\label{eq:dx5dx}\nonumber
\frac{\partial\underline x_5}{\partial x}&=&D\hat F^N_{t,a}(\underline x_4)\frac{\partial\underline x_4}{\partial x}\\\nonumber&=&\left(\begin{matrix}
A_4B_2\gamma\lambda^n+A_4A_2\lambda^{2n}+A_4A_2\alpha\left(\lambda^2/\mu\right)^n\\\nonumber
C_4B_2\gamma\lambda^n+C_4A_2\lambda^{2n}+C_4A_2\alpha\left(\lambda^2/\mu\right)^n\end{matrix}\right)\\
&+&\left(\begin{matrix}
B_4D_2\gamma\mu^n+B_4C_2\left(\lambda\mu\right)^n+B_4C_2\alpha\lambda^n\\
D_4D_2\gamma\mu^n+D_4C_2\left(\lambda\mu\right)^n+D_4C_2\alpha\lambda^n
\end{matrix} 
\right)=O(\mu^n).
\end{eqnarray}
The last set follows.
\end{proof}
\begin{lem}\label{lem:ypartialderivatives}
The partial derivatives at $\underline x=(x,y)\in HB(t,a)$ satisfy the following.
\begin{itemize}
\item[-]$\partial\alpha/\partial y,\partial \beta/\partial y,\partial\gamma/\partial y,\partial \delta/\partial y=O(1)$,
\item[-]$\partial A_2/\partial y,\partial B_2/\partial y,\partial C_2/\partial y=O(\mu(t,a)^n)$,
\item[-]$\partial A_4/\partial y,\partial B_4/\partial y,\partial C_4/\partial y,\partial D_4/\partial y=O(\mu(t,a)^{2n})$,
\item[-]$\partial\alpha_5/\partial y,\partial \beta_5/\partial y,\partial\gamma_5/\partial y,\partial \delta_5/\partial y=O(\mu(t,a)^{2n}))$.
\end{itemize}
\end{lem}
\begin{proof}
The first set of estimates follows by the smoothness of $h_{t,a}$. For the second set of estimates, observe that,
\begin{equation}\label{eq:dunderlinex2dy}
\frac{\partial\underline x_2}{\partial y}=\left(\begin{matrix}
\lambda^n& 0\\
0&\mu^n
\end{matrix}
\right)Dh_{t,a}(\underline x)\left(\begin{matrix}
0\\
1
\end{matrix}
\right)=\left(\begin{matrix}
\beta\lambda^n\\
\delta\mu^n
\end{matrix}
\right)=O(\mu^n).
\end{equation}
Moreover, 
\begin{equation}\label{eq:dx3dy}
\frac{\partial\underline x_3}{\partial y}=D\hat F^N_{t,a}(\underline x_2)\frac{\partial\underline x_2}{\partial y}=\left(\begin{matrix}
B_2\delta\mu^n+A_2\beta\lambda^n\\
D_2\delta\mu^n+C_2\beta\lambda^n
\end{matrix}
\right)=O(\mu^n).
\end{equation}
The third set of estimates follows from the fact that,
\begin{equation}\label{eq:dx4dy}
\frac{\partial\underline x_4}{\partial y}=\left(\begin{matrix}
\lambda^n& 0\\
0&\mu^n
\end{matrix}
\right)\frac{\partial\underline x_3}{\partial y}=\left(\begin{matrix}
B_2\delta\left(\lambda\mu\right)^n+A_2\beta\lambda^{2n}\\
D_2\delta\mu^{2n}+C_2\beta\left(\lambda\mu\right)^n
\end{matrix}
\right)=O(\mu^{2n}).
\end{equation}
Finally,
\begin{eqnarray}\label{eq:dx5dy}
\nonumber\frac{\partial\underline x_5}{\partial y}&=&D\hat F^N_{t,a}(\underline x_4)\frac{\partial\underline x_4}{\partial y}\\\nonumber&=&\left(\begin{matrix}
A_4B_2\delta\left(\lambda\mu\right)^n+A_4A_2\beta\lambda^{2n}
\\\nonumber
C_4B_2\delta\left(\lambda\mu\right)^n+C_4A_2\beta\lambda^{2n}\end{matrix}\right)\\
&+&\left(\begin{matrix}
B_4D_2\delta\mu^{2n}+B_4C_2\beta\left(\lambda\mu\right)^n\\
D_4D_2\delta\mu^{2n}+D_4C_2\beta\left(\lambda\mu\right)^n
\end{matrix}
\right)=O(\mu^{2n}).
\end{eqnarray}
The last set follows.
\end{proof}
\begin{lem}\label{lem:D2xypartialderivatives}
The partial derivatives at the critical point $c$ of $F_{t,a}$ satisfy the following,
\begin{itemize}
\item[-]$D_2=D\left(\underline x_2(c),t,a\right)=O\left(\left(\lambda(t,a)/\mu(t,a)\right)^n\right)$,
\item[-]$\partial D_2/\partial x=\gamma\partial D/\partial y=O(\gamma)$,
\item[-]$\partial D_2/\partial y=\delta \mu^n(t,a){\partial D}/{\partial y}\left(1+O\left(\left(\lambda(t,a)/\mu(t,a)\right)^n\right)\right)$
\end{itemize}
where $\delta{\partial D}/{\partial y}$ is uniformly bounded and uniformly away from zero and $\gamma$ can be taken arbitrarly small by shrinking the domain of linearization. 
\end{lem}
\begin{proof}
Observe that 
$$
\frac{\partial\underline x_2}{\partial x}=\left(\begin{matrix}
\lambda^n& 0\\
0&\mu^n
\end{matrix}
\right)Dh_{t,a}(\underline x)\left(\begin{matrix}
1\\
0
\end{matrix}
\right)=\left(\begin{matrix}
\lambda^n\left(1+\frac{\alpha}{\mu^n}\right)\\
\gamma
\end{matrix}
\right).
$$
As a consequence, if $\underline x_2=(x_2,y_2)$,
$$
\frac{\partial D_2}{\partial x}=\frac{\partial D}{\partial x}\frac{\partial x_2}{\partial x}+\frac{\partial D}{\partial y}\frac{\partial y_2}{\partial x}=O(\gamma).
$$
The second estimate follows by applying Remark \ref{rem:alphabetagammadeltasmall}. 
Using \eqref{eq:dunderlinex2dy}, we get
\begin{equation}\label{eq:dD2dyexact}
\frac{\partial D_2}{\partial y}=\frac{\partial D}{\partial x}\frac{\partial x_2}{\partial y}+\frac{\partial D}{\partial y}\frac{\partial y_2}{\partial y}=O\left({\lambda}^n\right)+\frac{\partial D}{\partial y}\delta\mu^n,
\end{equation}
where, by Remark \ref{rem:deltaBawayfromzero}, $\delta$ is uniformly away from zero. Moreover ${\partial D}/{\partial y}\neq 0$ is also uniformly away from zero because $\underline x_2$ is close to $q_3(t,a)$. In fact $\partial D/\partial y(q_3(t,0))$ is bounded away from zero since $q_1(t,0)$ is a non degenerate homoclinic tangency. The third estimate follows. It is left to prove the first estimate. Let $c$ be a critical point of $F_{t,a}$ with critical value $v$. Let $L$ be the almost horizontal leaf passing trough $v$ and let $w=(1,w_2)$ be the tangent vector to $h(L)$ at $h(v)$. Then $$w_2=O\left(\left(\frac{\lambda}{\mu}\right)^n\right).$$ Observe that, because $\underline x_3=h(v)$, we have $\partial\underline x_3/\partial y=\theta w$ with $\theta\neq 0$.
In particular, using \eqref{eq:dx3dy}, we get
$$
\frac{C_2\beta\lambda^n+D_2\delta\mu^n}{A_2\beta\lambda^n+B_2\delta\mu^n}=O\left(\left(\frac{\lambda}{\mu}\right)^n\right).
$$
Using the fact that $B_2$ is away from zero, see Remark \ref{rem:deltaBawayfromzero}, the first estimate follows.
\end{proof}
\noindent
{\it Proof of Proposition \ref{prop:existenceofcriticalpoint}.}
We will show that, if $\underline x$ is a critical point of $F_{t,a}$,
\begin{equation}\label{eq:Dphi}
D\Phi_{t,a}(\underline x)=\left(\begin{matrix}
\phi_{11}& \phi_{12}\\
\phi_{21}&\phi_{22}
\end{matrix}
\right)=\left(\begin{matrix}
-1+O(\gamma)& B_{\phi}\mu(t,a)^n\\
C_{\phi}\mu(t,a)^{2n} &D_{\phi}\mu(t,a)^{3n}
\end{matrix}
\right)
\end{equation}
where $B_{\phi}$ and $D_{\phi}$ are uniformly bounded and uniformly away from zero and $C_{\phi}$ can be taken arbitrarily small by shrinking the domain of linearization. Moreover, according to Remark \ref{rem:alphabetagammadeltasmall}, $\gamma$ can also be taken small. The proposition is proved by applying the implicit function theorem.

The top row of the matrix in \eqref{eq:Dphi} is given by
\begin{eqnarray*}
\phi_{11}&=&\left[Dh_{t,a}^{-1}(\underline x_3)DF_{t,a}^N(\underline x_2)\left(\begin{matrix}
\lambda^n& 0\\
0&\mu^n
\end{matrix}
\right)Dh_{t,a}(\underline x)\right]_{11}-1,\\
\phi_{12}&=&\left[Dh_{t,a}^{-1}(\underline x_3)DF_{t,a}^N(\underline x_2),\left(\begin{matrix}
\lambda^n& 0\\
0&\mu^n
\end{matrix}
\right)Dh_{t,a}(\underline x)\right]_{12}.
\end{eqnarray*}
Using the notation introduced above, a calculation shows that 
\begin{eqnarray}\label{eq:phi11}
\nonumber \phi_{11}&=&\left(1+\frac{\alpha_3}{\mu^n}
\right)\left(A_2\lambda^n+A_2\alpha\left(\frac{\lambda}{\mu}\right)^n+B_2\gamma\right)+{\beta_3}
\left(C_2\lambda^n+C_2\alpha\left(\frac{\lambda}{\mu}\right)^n+D_2\gamma\right)-1\\
&=&-1+O(\gamma),
\end{eqnarray}
where we also used that $D_2=O(\lambda^n/\mu^n)$, see Lemma \ref{lem:D2xypartialderivatives}.
Similarly, 
\begin{eqnarray}\label{eq:phi12}
\nonumber \phi_{12}&=&\left(1+\frac{\alpha_3}{\mu^n}
\right)\left(A_2\beta\lambda^n+B_2\delta\mu^n\right)+{\beta_3}
\left(C_2\beta\lambda^n+D_2\delta\mu^n\right)\\
&=&B_2\delta\mu^n+O(1)= B_{\phi}\mu^n, 
\end{eqnarray}
where $B_{\phi}$ is uniformly bounded and uniformly away from zero, see Remark \ref{rem:deltaBawayfromzero}. Observe that, the second component of the function $\Phi_{t,a}(\underline x)$ is 
\begin{eqnarray}\label{eq:phisecondcomponent}
\nonumber &&\left[Dh_{t,a}^{-1}(\underline x_5)D\hat F_{t,a}^N(\underline x_4)\left(\begin{matrix}
\lambda^n& 0\\\nonumber
0&\mu^n
\end{matrix}
\right)D\hat F_{t,a}^N(\underline x_2)\left(\begin{matrix}
\lambda^n& 0\\\nonumber
0&\mu^n
\end{matrix}
\right)Dh_{t,a}(\underline x)\right]_{12}\\\nonumber
&=&\delta D_2B_4\mu^{2n}+\delta D_2D_4\beta_5\mu^{2n}+\delta D_2B_4\alpha_5\mu^{n}+\delta B_2 A_4\left(\lambda\mu\right)^n+\delta B_2 C_4\beta_5 \left(\lambda\mu\right)^n\\\nonumber &+&\delta B_2 A_4\alpha_5\lambda^n+\beta C_2 B_4\left(\lambda\mu\right)^n+\beta C_2 D_4\beta_5\left(\lambda\mu\right)^n+\beta C_2 B_4\alpha_5\lambda^n+\beta A_2 A_4\lambda^{2n}\\&+&\beta A_2 C_4\beta_5\lambda^{2n}+\beta A_2 A_4\alpha_5\left(\frac{\lambda^{2}}{\mu}\right)^n.
\end{eqnarray}
In the following calculation we often use that $\lambda\mu^3<1$. Using Lemma \ref{lem:xpartialderivatives} and Lemma \ref{lem:D2xypartialderivatives} we have
\begin{eqnarray}\label{eq:phi21}
\nonumber \phi_{21}&=&\frac{\partial}{\partial x}\left[\delta D_2B_4\mu^{2n}+\delta D_2D_4\beta_5\mu^{2n}+\delta D_2B_4\alpha_5\mu^{n}\right]+O\left(\left(\lambda\mu^2\right)^n\right)\\
\nonumber &=&\frac{\partial D_2}{\partial x}B_4\delta\mu^{2n}+\frac{\partial D_2}{\partial x}D_4\delta\beta_5\mu^{2n}+O(\mu^n)\\
&=&C_{\phi}\mu^{2n}.
\end{eqnarray}
Observe that $C_{\phi}=O({\partial D_2}/{\partial x})=O(\gamma)$, see Lemma \ref{lem:D2xypartialderivatives}.
Similarly, by using Lemma \ref{lem:ypartialderivatives} and Lemma \ref{lem:D2xypartialderivatives} we have
\begin{eqnarray}\label{eq:phi22}
\nonumber \phi_{22}&=&\frac{\partial}{\partial y}\left[\delta D_2B_4\mu^{2n}+\delta D_2D_4\beta_5\mu^{2n}+\delta D_2B_4\alpha_5\mu^{n}\right]+O\left(\left(\lambda\mu^3\right)^n\right)\\
\nonumber &=&\frac{\partial D_2}{\partial y}B_4\delta\mu^{2n}+\frac{\partial D_2}{\partial y}D_4\delta\beta_5\mu^{2n}+O(\mu^{2n})\\
&=&D_{\phi}\mu^{3n}.
\end{eqnarray}
By Remark \ref{rem:alphabetagammadeltasmall}, Remark \ref{rem:deltaBawayfromzero} and Lemma \ref{lem:D2xypartialderivatives}, $D_{\phi}$ is uniformly bounded away from zero.
From \eqref{eq:phi11}, \eqref{eq:phi12}, \eqref{eq:phi21} and \eqref{eq:phi22} we have that $D\Phi_{t,a}(\underline x)$ is non singular. 
The proposition follows by applying the implicit function theorem. 
\qed

\subsection{Global existence of the critical point}
In this subsection we prove that all maps in $\mathcal H_n$ have a critical point, namely $\mathcal C_n=\mathcal H_n$. To do that we need to study the $a$-dependence of the critical point and critical value. 
From propositions \ref{prop:existenceofcriticalpoint} and \ref{prop:HncontainedinCn}, for all $(t,a)\in\mathcal C_n$, the map $F_{t,a}$ has a critical point $(c_{x}(t,a),c_{y}(t,a))$ with critical value $(v_{x}(t,a),v_{y}(t,a))$, where $v_x(t,a)=c_x(t,a)$. 
\begin{prop}\label{prop:aderivativescriticalpointscriticalvalue}
For every $(t,a)\in\mathcal C_n$, 
\begin{eqnarray*}
\frac{\partial c_{x}}{\partial a}&=&\frac{\partial v_{x}}{\partial a}=O(n), \\\frac{\partial c_{y}}{\partial a}&=&O\left(\frac{n}{\mu^n(t,a)}\right),\\ \frac{\partial v_{y}}{\partial a}&=&\delta_3+O\left(\frac{n}{\mu^n(t,a)}\right).
\end{eqnarray*}
In particular, 
\begin{equation*}
| v_{y}-c_{y}|=O\left(\frac{1}{\mu^{2n}(t,a)}\right).
\end{equation*}
\end{prop}
The proof of the previous proposition needs some preparation. Fix $(t,a)\in\mathcal C_n$. Let $(x,y)\in HB(t,a)\cap\mathbb R^2$ and  $\underline x=(x,y)$, $\underline x_1$,
$\underline x_2$, $\underline x_3$, $\underline x_4$, $\underline x_5$ as introduced before. Observe that all coefficients introduced, $\alpha,\beta,\dots,A_2,B_2,\dots,\alpha_3,\beta_3,\dots,A_4,B_4,\dots,\alpha_5,\dots$ are functions of the point $(x,y)\in HB(t,a)\cap\mathbb R^2$ and $a$.

\begin{lem}\label{lem:apartialderivatives}
The partial derivatives at the critical point $c$ of $F_{t,a}$ satisfy the following.
\begin{itemize}
\item[-]$\partial\alpha/\partial a,\partial \beta/\partial a,\partial\gamma/\partial a,\partial \delta/\partial a=O(1)$,
\item[-]$\partial A_2/\partial a,\partial B_2/\partial a,\partial C_2/\partial a,\partial D_2/\partial a=O(n)$,
\item[-]$\partial\alpha_3/\partial a,\partial \beta_3/\partial a,\partial\gamma_3/\partial a,\partial \delta_3/\partial a=O(n)$,
\item[-]$\partial A_4/\partial a,\partial B_4/\partial a,\partial C_4/\partial a,\partial D_4/\partial a=O(\mu^n(t,a))$,
\item[-]$\partial\alpha_5/\partial a,\partial \beta_5/\partial a,\partial\gamma_5/\partial a,\partial \delta_5/\partial a=O(\mu^n(t,a)))$.
\end{itemize}
\end{lem}
\begin{proof}
The first set of estimates follows by the smoothness of the family of maps $h_{t,a}$. For the second set of estimates, observe that, because $h_{t,a}(x,0)=(x,0)$ and $y=O\left(1/\mu^n\right)$, we have 
$$
\frac{\partial\underline x_1}{\partial a}=\frac{\partial{h}}{\partial a}=O\left(\frac{1}{\mu^n}\right).
$$
Recall that $\underline x_2=\hat F^n_{t,a}\left(\underline x_1,t,a\right)$. It follows that
\begin{equation}\label{eq:x2aderivatives}
\frac{\partial x_2}{\partial a}=n\lambda^{n-1}\frac{\partial\lambda}{\partial a}x_1+\lambda^n\frac{\partial x_1}{\partial a}=O\left(n\lambda^n\right),
\end{equation}
and 
\begin{equation}\label{eq:y2aderivatives}
\frac{\partial y_2}{\partial a}=n\mu^{n-1}\frac{\partial\mu}{\partial a}y_1+\mu^n\frac{\partial y_1}{\partial a}=O\left(n\right),
\end{equation}
where we used that $y_1=O\left(1/\mu^n\right)$.
As a consequence,
$$
\frac{\partial A_2}{\partial a}=\frac{\partial A}{\partial x}\frac{\partial x_2}{\partial a}+\frac{\partial A}{\partial y}\frac{\partial y_2}{\partial a}+\frac{\partial A}{\partial a}=O\left(n\right).
$$
Similarly the other estimates in the second statement follows. 
Recall that $\underline x_3=\hat F^N_{t,a}\left(\underline x_2,t,a\right)$. It follows that
\begin{equation}\label{eq:x3aderivatives}
\frac{\partial x_3}{\partial a}=A_2\frac{\partial x_2}{\partial a}+B_2\frac{\partial y_2}{\partial a}+O(1)=O\left(n\right),
\end{equation}
and 
\begin{equation}\label{eq:y3aderivatives}
\frac{\partial y_3}{\partial a}=C_2\frac{\partial x_2}{\partial a}+D_2\frac{\partial y_2}{\partial a}+O(1)=O(1),
\end{equation}
where we used \eqref{eq:x2aderivatives}, \eqref{eq:y2aderivatives} and that $D_2=O\left(\lambda^n/\mu^n\right)$, see Lemma \ref{lem:D2xypartialderivatives}. The third set of estimates follow.
Recall now that $\underline x_4=\hat F^n_{t,a}\left(\underline x_3,t,a\right)$. It follows that,
\begin{equation}\label{eq:x4aderivatives}
\frac{\partial x_4}{\partial a}=n\lambda^{n-1}\frac{\partial\lambda}{\partial a}x_3+\lambda^n\frac{\partial x_3}{\partial a}=O\left(n\lambda^n\right),
\end{equation}
and 
\begin{equation}\label{eq:y4aderivatives}
\frac{\partial y_4}{\partial a}=n\mu^{n-1}\frac{\partial\mu}{\partial a}y_3+\mu^n\frac{\partial y_3}{\partial a}=O\left(\mu^n\right),
\end{equation}
where we used \eqref{eq:x3aderivatives}, \eqref{eq:y3aderivatives} and that $y_3=O\left(1/\mu^n\right)$.
As a consequence,
$$
\frac{\partial A_4}{\partial a}=\frac{\partial A}{\partial x}\frac{\partial x_4}{\partial a}+\frac{\partial A}{\partial y}\frac{\partial y_4}{\partial a}+\frac{\partial A}{\partial a}=O\left(\mu^n\right).
$$
Similarly the other estimates in the fourth statement follows. 
Finally, recall that $\underline x_5=\hat F^N_{t,a}\left(\underline x_4,t,a\right)$. It follows that,
\begin{equation}\label{eq:x5aderivatives}
\frac{\partial x_5}{\partial a}=A_4\frac{\partial x_4}{\partial a}+B_4\frac{\partial y_4}{\partial a}+O(1)=O\left(\mu^n\right),
\end{equation}
and 
\begin{equation}\label{eq:y5aderivatives}
\frac{\partial y_5}{\partial a}=C_4\frac{\partial x_4}{\partial a}+D_4\frac{\partial y_4}{\partial a}+O(1)=O\left(\mu^n\right),
\end{equation}
where we used \eqref{eq:x4aderivatives}, \eqref{eq:y4aderivatives}.
As before the last set follows.
\end{proof}
\begin{lem}\label{lem:D2apartialderivatives}
The partial derivative of $y_3$ at the critical point $c$ of $F_{t,a}$ satisfy the following:
$$\frac{\partial y_3}{\partial a}=1+O\left(\frac{1}{\mu^n(t,a)}\right).$$
\end{lem}
\begin{proof}
For $a=0$, let $q_3(t,0)$ be the point where $D(q_3,t,0)=0$. Recall that $\hat F_{t,0}^N(q_3)=q_1$ where $q_1$ is the tangency at $a=0$. Because $\partial D(q_3,t,0)/\partial y\neq 0$, there is a curve trough $q_3$ transversal to the $y$-axis where $D(\underline x,t,0)=0$ for all points $\underline x$ in the curve. The transversality implies that there is a curve $D(\underline x,t,a)=0$, transversally intersecting the $y$-axis in the point $q_3(t,a)$ at the distance of $O(1/\mu^n)$ to the curve $D(\underline x,t,0)=0$, because $(t,a)\in\mathcal C_n$. From Lemma \ref{lem:D2xypartialderivatives} we know that $D_2=O\left(\lambda^n/\mu^n\right)$. This implies that the distance from $\underline x_2$ to this curve is of $O\left(\lambda^n/\mu^n\right)$. Moreover $x_2=O(\lambda^n)$. Hence, 
$$
\text{dist}\left(\underline x_2,q_3(t,0)\right)=O\left(\frac{1}{\mu^n}\right).
$$
Because $\partial\left(\hat F^N_{t,0}(q_3(t,0))\right)_y/\partial a=1$, we get 
$$
\frac{\partial\left(\hat F^N_{t,a}(\underline x_2)\right)_y}{\partial a}=1+O\left(\frac{1}{\mu^n}\right).
$$
Now \eqref{eq:y3aderivatives}, becomes 
$$
\frac{\partial y_3}{\partial a}=C_2\frac{\partial x_2}{\partial a}+D_2\frac{\partial y_2}{\partial a}+1+O\left(\frac{1}{\mu^n}\right)=1+O\left(\frac{1}{\mu^n}\right),
$$
where we also used \eqref{eq:x2aderivatives}, \eqref{eq:y2aderivatives} and Lemma \ref{lem:D2xypartialderivatives}.
\end{proof}
\noindent
{\it Proof of Proposition \ref{prop:aderivativescriticalpointscriticalvalue}.}
Observe that 
$$
\frac{\partial c}{\partial a}=\left(D\Phi\right)^{-1}\frac{\partial\Phi}{\partial a}.
$$
Let $\Phi=(\Phi_1,\Phi_2)$. We start by calculating ${\partial\Phi_1}/{\partial a}$. Observe that $\Phi_1=\left(h_{t,a}^{-1}(\underline x_3)-x\right )_x$. As a consequence,
\begin{equation}\label{eq:dphi1da} 
\frac{\partial\Phi_1}{\partial a}=\left(1+\frac{\alpha_3}{\mu^n}\right)\frac{\partial x_3}{\partial a}+{\beta_3}\frac{\partial y_3}{\partial a}+\frac{\partial h^{-1}}{\partial a}=O(n),
\end{equation}
where we used \eqref{eq:x3aderivatives} and \eqref{eq:y3aderivatives}. In order to calculate ${\partial\Phi_2}/{\partial a}$ we take the $a$-derivative of \eqref{eq:phisecondcomponent}. Using Lemma \ref{lem:apartialderivatives}, the partial derivatives of all terms in \eqref{eq:phisecondcomponent} with a factor $\lambda^n$ will give a contribution of $O\left(\left(\lambda\mu^2\right)^n\right)$. We get
\begin{eqnarray*}
\frac{\partial\Phi_2}{\partial a}&=&\frac{\partial}{\partial a}\left[\delta D_2B_4\mu^{2n}+\delta D_2D_4\beta_5\mu^{2n}+\delta D_2B_4\alpha_5\mu^{n}\right]\\&+&O\left(\left(\lambda\mu^2\right)^n\right).
\end{eqnarray*}
Using Lemma \ref{lem:apartialderivatives} and the fact that $D_2=O\left(\lambda^n/\mu^n\right)$, see Lemma \ref{lem:D2xypartialderivatives}, we get 
\begin{eqnarray}\label{eq:dphi2da}
\frac{\partial\Phi_2}{\partial a}&=&O\left(n\mu^{2n}\right).
\end{eqnarray}
Because the determinant of $D\Phi$ is proportional to $\mu^{3n}$, see \eqref{eq:Dphi}, we have,
$$
\frac{\partial c}{\partial a}=\left(D\Phi\right)^{-1}\frac{\partial\Phi}{\partial a}=\left(\begin{matrix}
O(1)& O\left(1/\mu^{2n}\right)\\
O\left(1/\mu^{n}\right) &O\left(1/\mu^{3n}\right)
\end{matrix}
\right)\left(\begin{matrix}
O(n)\\
O\left(n\mu^{2n}\right)\end{matrix}
\right)=\left(\begin{matrix}
O(n)\\
O\left(n/\mu^{n}\right)
\end{matrix}
\right),
$$
where we used \eqref{eq:dphi1da} and \eqref{eq:dphi2da}. The first two equations follow. Observe that 
$v_y=\left(h^{-1}_{t,a}(\underline x_3)\right)_y.$ Hence,
$$
\frac{\partial v_y}{\partial a}=\frac{\gamma_3}{\mu^n}\frac{\partial x_3}{\partial a}+{\delta_3}\frac{\partial y_3}{\partial a}+O\left(\frac{1}{\mu^n}\right)=\delta_3+O\left(\frac{n}{\mu^n}\right),
$$
where we used \eqref{eq:x3aderivatives} and Lemma \ref{lem:D2apartialderivatives}. The last statement is an immediate consequence of the fact that $(t,a)\in\mathcal C_n\subset\mathcal H_n$.
\qed
\bigskip
\begin{lem}\label{prop:HncontainedinCn}
For $n$ large enough, $\mathcal H_n=\mathcal C_n$.
\end{lem}
\begin{proof}
Let $(t,sa_n(t))\in\mathcal C_n$. Then $p(t)=(2,p_y(t))\in\left\{2\right\}\times\left[\delta^-,\delta^+\right]=HB_2(t,sa_n(t))\cap\mathbb R^2$. 
Observe that
$$
\left|\left(F^{n+N}_{(t,sa_n(t))}\left(2,\delta^{\pm}\right)\right)_x-2\right|=1.
$$
As a consequence, by Lemma \ref{DFnC}, there exists a constant $K>0$ such that 
$$
1=\left|\left(F^{n+N}_{(t,sa_n(t))}\left(p(t)\right)-F^{n+N}_{(t,sa_n(t))}\left(\delta^{\pm}\right)\right)_x\right|\leq \frac{1}{K}\left|p_y(t)-\delta^{\pm}\right|\mu^n.
$$
Hence, 
\begin{equation}\label{eq2:propHncontainedinCn}
\left|p_y(t)-\delta^{\pm}\right|\geq \frac{K}{\mu^n}.
\end{equation}
For $(t,a)\in\mathcal H_n$ let $\partial^{\pm}HB(t,a)$ be the two almost horizontal curves in the boundary of $HB(t,a)\cap\mathbb R^2$. Notice that $\partial^{\pm}HB$ are graphs of functions satisfying, by Lemma $12$ in \cite{BMP}, 
\begin{equation}\label{eq4:propHncontainedinCn}
\frac{\partial}{\partial a}\left(\partial^{\pm}HB\right)=O\left(\frac{n}{\mu^n}\right),
\end{equation}
and by Lemma \ref{DFnC},
\begin{equation}\label{eq5:propHncontainedinCn}
\frac{\partial}{\partial x}\left(\partial^{\pm}HB\right)=O\left(\frac{1}{\mu^n}\right).
\end{equation}
Choose $(t,a)\in\partial\mathcal C_n$ and suppose  $(t,a)\notin\partial\mathcal H_n$. Then either $c=c(t,a)$ or $v=v(t,a)$ are in the boundary of $HB(t,a)$. Because of Proposition \ref{prop:aderivativescriticalpointscriticalvalue}, 
\begin{equation}\label{eq6:propHncontainedinCn}
|c_x-2|=|v_x-2|=O\left(\frac{n}{\mu^{2n}}\right).
\end{equation}
Hence,
\begin{equation}\label{eq3:propHncontainedinCn}
\left\{c,v\right\}\cap \left\{\partial^{+}HB\cup\partial^{-}HB\right\}\neq\emptyset.
\end{equation}
Observe that $c(t,sa_n(t))=v(t,sa_n(t))=p(t)$ and by Proposition \ref{prop:aderivativescriticalpointscriticalvalue},
\begin{equation}\label{eq1:propHncontainedinCn}
|c-p(t)|,|v-p(t)|=O\left(\frac{n}{\mu^{2n}}\right).
\end{equation}
Consider the case $c\in\partial^{+}HB$. The other cases described in \eqref{eq3:propHncontainedinCn} can be treated in exactly the same way. Because of  \eqref{eq4:propHncontainedinCn}, \eqref{eq5:propHncontainedinCn}, and \eqref{eq6:propHncontainedinCn},  we get, 
$$
|\delta^+-c|=O\left(\frac{n}{\mu^{2n}}\right).
$$
This last equation together with \eqref{eq2:propHncontainedinCn} and \eqref{eq1:propHncontainedinCn} lead to a contradiction. Namely,
$$
\frac{K}{\mu^n}\leq |p_y(t)-\delta^+|\leq |p(t)-c|+|c-(2,\delta^+)|=O\left(\frac{n}{\mu^{2n}}\right).
$$
Hence $(t,a)\in\partial\mathcal H_n$ and $\mathcal C_n=\mathcal H_n$.
\end{proof}
\subsection{The normalized map}
In this subsection we rescale the domain of the straightened map to unit size. 
Fix $(t,a)\in\mathcal H_n$. Recall that the return map after straightening became
$$
\tilde F_{t,a}\left(\begin{matrix}
x\\y
\end{matrix}\right)
=\left(\begin{matrix}
g_{t,a}(x,y)\\
x
\end{matrix}\right).
$$
In particular, the map $x\mapsto g_{t,a}(x,c_x(t,a))$ has a critical point at $c_{t,a}=c_x(t,a)$. Similarly as in Definition \ref{def:criticalpointofF}, the point $(c_{t,a},c_{t,a})$ is a critical point of $\tilde F_{t,a}$ in the sense that
$$
\tilde F_{t,a}\left(\begin{matrix}
c_{t,a}\\
c_{t,a}
\end{matrix}\right)=\left(\begin{matrix}
v_{t,a}\\
c_{t,a}
\end{matrix}\right)
$$
and the image of the horizontal slice at $y=c_{t,a}$ is tangent to the vertical slice at $x=v_{t,a}$.
\begin{lem}\label{lem:vminusc}
For all $(t,a)\in \mathcal H_n$,
$$
\frac{\partial v_{t,a}}{\partial a}=\left(B_4+D_4\beta_5\right)\mu^n(t,a)+O\left(n\right).
$$
In particular,
$$
| v_{t,a}-c_{t,a}|=O\left(\frac{1}{\mu^{n}(t,a)}\right).
$$
\end{lem}
\begin{proof}
Observe that $v_{t,a}=\left(h^{-1}_{t,a}\left(\underline x_5\right)\right)_x$. To calculate its $a$-derivative we improve the derivatives in the proof of Lemma \ref{lem:apartialderivatives}. By Lemma \ref{lem:D2apartialderivatives}, we get 
$$
\frac{\partial y_4}{\partial a}=n\mu^{n-1}\frac{\partial\mu}{\partial a}y_3+\mu^n\frac{\partial y_3}{\partial a}=\mu^n+O\left(n\right).
$$
By this, by \eqref{eq:x4aderivatives} and by \eqref{eq:y4aderivatives} we have, 
$$
\frac{\partial x_5}{\partial a}=A_4\frac{\partial x_4}{\partial a}+B_4\frac{\partial y_4}{\partial a}+O(1)=B_4\mu^n+O(n),
$$
and 
$$
\frac{\partial y_5}{\partial a}=C_4\frac{\partial x_4}{\partial a}+D_4\frac{\partial y_4}{\partial a}+O(1)=D_4\mu^n+O(n).
$$
Finally, 
$$
\frac{\partial v_{t,a}}{\partial a}=\left(1+\frac{\alpha_5}{\mu^n}\right)\frac{\partial x_5}{\partial a}+{\beta_5}\frac{\partial y_5}{\partial a}+O\left(\frac{1}{\mu^n}\right)=\left(B_4+\beta_5D_4\right)\mu^n+O(n).
$$
The last equality follows from the fact that $(t,a)\in\mathcal H_n$.
\end{proof}
We rewrite 
$$
\tilde F_{t,a}\left(\begin{matrix}
x\\y
\end{matrix}\right)
=\left(\begin{matrix}
q_{t,a}(x)+\tilde\epsilon_{t,a}(x,y)\\
x
\end{matrix}\right)
$$
where $q_{t,a}(x)=g_{t,a}(x,c_{t,a})$. 
Observe that, the domain of $q_{t,a}$ is the simply connected domain $\pi_x\left(\sigma_{t,a}\left(D_{c_{t,a}}\right)\right)=\tilde D\subset\mathbb D_1$. Moreover, because $p_x(t)=c_x(t,sa_n(t))=2$ and because of Proposition \ref{prop:aderivativescriticalpointscriticalvalue}, we have $$|c_{t,a}-2|=O\left(\frac{n}{\mu^{2n}(t,a)}\right).$$
In particular, $c_{t,a}$ is exponentially close to the center of $\mathbb D_1$. Because $|v_y|=O\left(1/{\mu^{n}(t,a)}\right)$ and for every $(x,y)\in HB(t,a)$, by Lemma \ref{lem:hbx}, $|y|=O\left(1/{\mu^{n}(t,a)}\right)$ we get 
\begin{equation}\label{eq:sgrtmun}
\text{diam}\left(F^n_{t,a}(D_{c_{t,a}})\right)=O\left(\frac{1}{\mu^{n/2}(t,a)}\right).
\end{equation}
In fact, $D_{c_{t,a}}\subset HB(t,a)$ with $F_{t,a}^{n+N}\left(D_{c_{t,a}}\right)\subset HB(t,a)$. Because $\pi_y\left(F^N_{t,a}\left(F^n_{t,a}\left(D_{c_{t,a}}\right)\right)\right)$ is contained in a disc of radius $1/\mu^n(t,a)$ centered in zero and $\pi_y\left(F^N_{t,a}\left(F^n_{t,a}\left(HB(t,a)\right)\right)\right)$ contains a disc of uniform size centered in zero, the estimate in \eqref{eq:sgrtmun} follows. Since $\tilde D=\pi_x\left(F^N_{t,a}\left(F^n_{t,a}\left(D_{c_{t,a}}\right)\right)\right)$ we have
\begin{equation}\label{eq:diameterD1}
\text{diam}(\tilde D)=O\left(\frac{1}{\mu^{n/2}(t,a)}\right).
\end{equation}

\begin{figure}[h]
\centering
\includegraphics[width=1.06\textwidth]{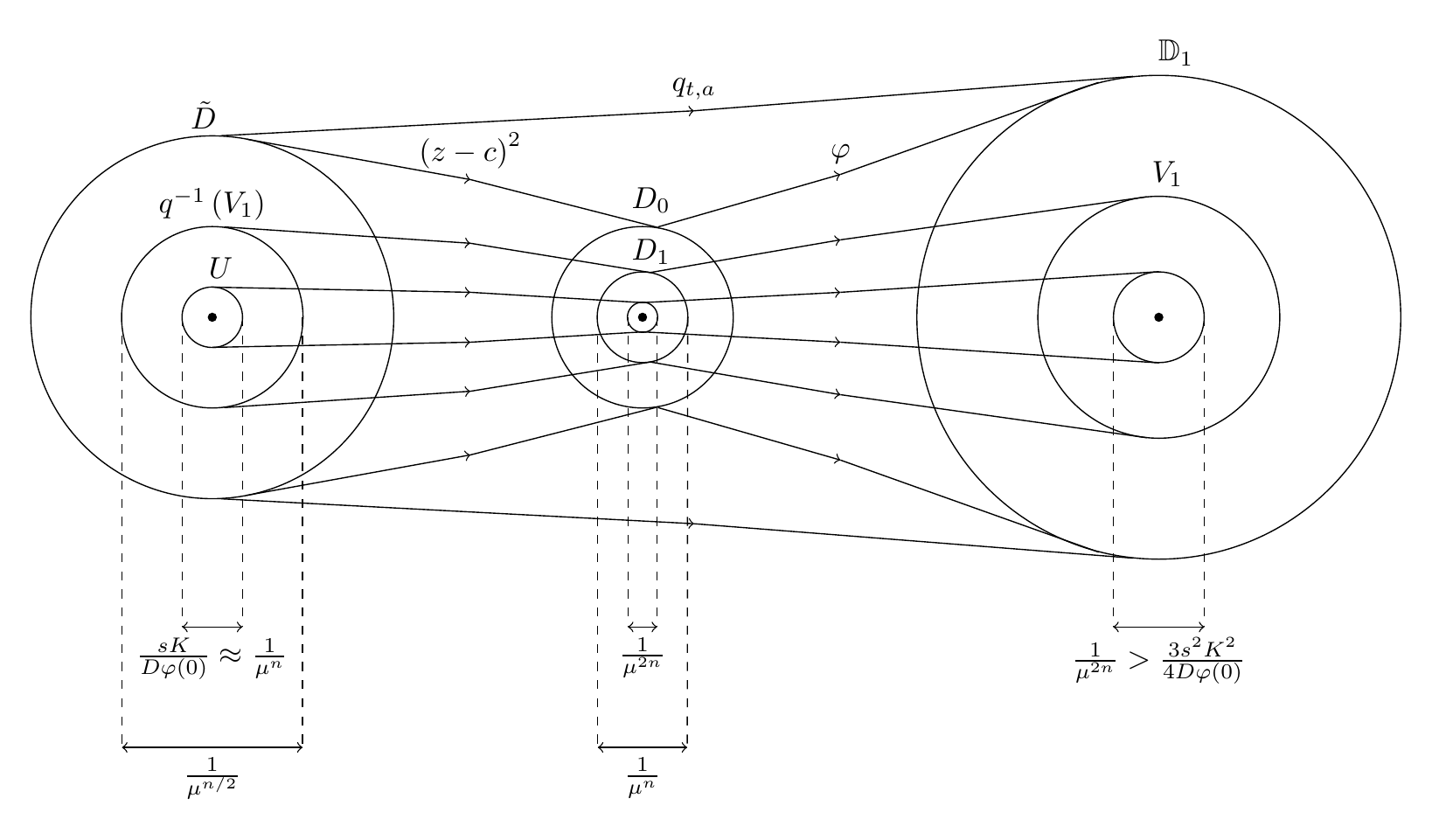}
\caption{The factorization}
\label{fig:constnormalizedmap}
\end{figure}

We will need to introduce a rescaling to bring this domain back to unit size. 
The following construction is illustated in Figure \ref{fig:constnormalizedmap}.
 We factorize the map $q_{t,a}$ as 
 \begin{equation}\label{eq:qtaphi}
 q_{t,a}=\varphi\left(\left(z- c_{t,a}\right)^2\right)
  \end{equation}
where $\varphi:D_0\to\mathbb D_1$ is univalent and onto and its domain $ D_0$ is the image under the map $z\mapsto (z-c_{t,a})^2$ of $\tilde D$. Let $V_1\subset\mathbb D_1$ be the disc or radius $1/2$ centered at $2$. Consider the map 
$$
q_{t,a}:q^{-1}_{t,a}(V_1)\to V_1.
$$
Let $D_1\subset D_0$ be the disc which is the image under the map $z\mapsto (z-c_{t,a})^2$ of $q^{-1}_{t,a}(V_1)$. Then the restriction $\varphi:D_1\to V_1$ has uniformly bounded distortion, because it has univalent extension up to the unit disc $\mathbb D_1$.

Observe that $q^{-1}_{t,a}(V_1)\subset\tilde D$ has diameter proportional to the diameter of $\tilde D$ which is proportional to $1/\mu^{n/2}(t,a)$, see \eqref{eq:diameterD1}. Hence, $D_1$ has diameter proportional to $1/\mu^{n}(t,a)$. Because $\varphi$ has bounded distortion, this implies that $D\varphi(0)$ is proportional to $\mu^{n}(t,a)$. 
Define $K_{t,a}$ by
\begin{equation}\label{eq:tildevminustildec}
v_{t,a}-c_{t,a}=\frac{K_{t,a}}{D\varphi(0)}.
\end{equation}
and $K=\max_{t,a}K_{t,a}$.
Because of Lemma \ref{lem:vminusc}, we have that $K$ is uniformly bounded.

Let $s>0$ be the unique solution of $s+1=s^2K/2$ and
$$
U=\left\{z\left|\right. |z-c_{t,a}|\leq\frac{sK}{D\varphi(0)}\right\}\subset q^{-1}_{t,a}(V_1).
$$
The image of $U$ under the map $z\to (z-c_{t,a})^2$ is denoted by $U^2$. Observe that $U^2$ is a round disc, centered around zero, of radius $s^2K^2/D\varphi(0)^2$. In particular the diameter of $U^2$ is proportional to $1/\mu^{2n}(t,a)$ and the diameter of $\varphi\left(U^2\right)$ is proportional to $1/\mu^{n}(t,a)$. This implies that, if $z\in U^2$, then using the Koebe Lemma as stated in \cite{Ahl1}, \cite{Ahl2} or \cite{Bieb}, we have
\begin{equation}\label{eq:Dvarphiequaltoone}
\frac{D\varphi(z)}{D\varphi(0)}=1+O\left(\frac{1}{\mu^{n}({t,a})}\right).
\end{equation}
This implies that, for all $z\in U^2$,
\begin{equation}\label{eq:varphiz}
\varphi(z)=\varphi(0)+D\varphi(0)\int_0^z\frac{D\varphi(z)}{D\varphi(0)}dz=v_{t,a}+D\varphi(0)z+O\left(\frac{1}{\mu^{2n}(t,a)}\right).
\end{equation}

In particular $\varphi\left(U^2\right)$ contains a round disc $V$ centered around $v_{t,a}$ of radius $3s^2K^2/4D\varphi (0)$. 
Moreover, if $z\in U$, 
$$
 |z-v_{t,a}|\leq  |z- c_{t,a}|+ | c_{t,a}-v_{t,a}|\leq \frac{sK}{D\varphi(0)}+\frac{K}{D\varphi(0)}=\frac{s^2K^2}{2D\varphi(0)}<\frac{3s^2K^2}{4D\varphi(0)}.
$$
Hence,
$$
U\subset  V\subset q_{t,a}(U).
$$
Because  $sK=1+\sqrt{1+2K}$, $\text{diam}(U)D\varphi(0)=sK\geq 2$ and $\text{mod}\left(U,V\right)\geq \log\left(3/2\right)$. In particular the map $q_{t,a}:U\to V$ is a quadratic-like map with uniform modulus.

The last step is to rescale the domain $U$ to a finite size. 
Let $\alpha: x\to  c_{t,a}+x/D\varphi(0)$, $U_{t,a}=\alpha^{-1}\left(U\right)$ and $V_{t,a}=\alpha^{-1}\left(V\right)$. Moreover, let $Z_{t,a}:\mathbb C^2\to\mathbb C^2$ defined as 
$$
Z_{t,a}\left(
\begin{matrix}
x\\y
\end{matrix}
\right)=\left(
\begin{matrix}
\alpha(x)\\\alpha(y)
\end{matrix}
\right).
$$
We define now the domain of the H\'enon-like map $HF_{t,a}$ as
$$
\text{Dom}(HF_{t,a})=U_{t,a}\times U_{t,a}
$$
and
$$
HF_{t,a}=Z^{-1}_{t,a}\circ \tilde F_{t,a}\circ Z_{t,a}:\text{Dom}( HF_{t,a})\to\mathbb C\times\mathbb C.
$$
By construction 
$$
HF_{t,a}\left(\begin{matrix}
x\\y
\end{matrix}\right)
=\left(\begin{matrix}
f_{t,a}(x,y)\\
x
\end{matrix}\right).
$$
\begin{theo}\label{firstreturnmapanalytic}
For every $(t,a)\in\mathcal H_n$, there exists $\nu_{t,a}\in\mathbb R$ such that 
\begin{eqnarray*}
HF_{t,a}\left(\begin{matrix}
x\\y
\end{matrix}\right)=\left(\begin{matrix}
f_{t,a}(x,y)\\
x
\end{matrix}\right)
&=&\left(\begin{matrix}
x^2+\nu_{t,a} +\epsilon_{t,a}(x,y)\\
x
\end{matrix}\right),
\end{eqnarray*}
 with $\epsilon_{t,a}(0,0)=0$ and
$$|\epsilon_{t,a}|=O\left(\frac{1}{\mu(t,a)^n}\right).$$
Moreover, for every $y\in U_{t,a}$, the map $f_{t,a}\left(\cdot,y\right): U_{t,a}\to\mathbb C$ is quadratic-like with
$$
 V_{t,a}\subset f_{t,a}\left(U_{t,a},y\right),
$$
$$
\text{\rm mod}\left( U_{t,a},  V_{t,a}\right)\geq \log\left(\frac{3}{2}\right),
$$
and
$$
\text{diam} \left(U_{t,a}\right)\geq 2.
$$
\end{theo}
\begin{proof}
The construction of $HF_{t,a}$ gives the formula for the $y$-coordinate. Observe that, by using \eqref{eq:varphiz},
\begin{eqnarray*}
f_{t,a}(x,0)&=& \alpha^{-1}\circ q_{t,a}\circ \alpha= D\varphi(0)\left[\varphi\left(\left[\frac{x}{D\varphi(0)}\right]^2\right)- c_{t,a}\right]\\&=&D\varphi(0)\left[ v_{t,a}+\frac{x^2}{D\varphi(0)}+O\left(\frac{1}{\mu^{2n}}\right)-c_{t,a}\right]\\&=&x^2+\nu_{t,a}+O\left(\frac{1}{\mu^{n}}\right),
\end{eqnarray*}
where, by \eqref{eq:tildevminustildec},
\begin{equation}\label{eq:nutaformula}
\nu_{t,a}=D\varphi(0)\left( v_{t,a}- c_{t,a}\right)=O(1).
\end{equation}
 By construction, for $y=0$, we have $V_{t,a}\subset f_{t,a}\left( U_{t,a},0\right)$. Take $(x_0,y_0)\in U_{t,a}\times U_{t,a} $ and compare its image with the image of $(x_0,0)$. Let 
$$
Z_{t,a}\left(
\begin{matrix}
x_0\\y_0
\end{matrix}
\right)=\left(
\begin{matrix}
\tilde x\\\tilde y
\end{matrix}
\right)
\text{
and }
Z_{t,a}\left(
\begin{matrix}
x_0\\0
\end{matrix}
\right)=\left(
\begin{matrix}
\tilde x\\c_{t,a}
\end{matrix}
\right),
$$
then 
$$
\tilde y=\frac{y_0}{D\varphi(0)}+c_{t,a}.
$$
Let 
$$
\left(
\begin{matrix}
x\\y
\end{matrix}
\right)=\sigma_{t,a}^{-1}\left(
\begin{matrix}
\tilde x\\\tilde y
\end{matrix}
\right)
\text{
and }
\left(
\begin{matrix}
x^{(0)}\\y^{(0)}
\end{matrix}
\right)=\sigma_{t,a}^{-1}\left(
\begin{matrix}
\tilde x\\c_{t,a}
\end{matrix}
\right),
$$
then $\left(\begin{matrix}
x, y
\end{matrix}
\right)$ and $\left(\begin{matrix}
x^{(0)},y^{(0)}
\end{matrix}\right)$ are on the same leaf of the almost horizontal foliation and
\begin{equation}\label{lem:x2minusx2zero}
\left |x-x^{(0)}\right|=\frac{y_0}{D\varphi(0)}. 
\end{equation}
Let 
$$
\left(
\begin{matrix}
x_2\\y_2
\end{matrix}
\right)=F_{t,a}^n\left(
\begin{matrix}
x\\y
\end{matrix}
\right)
\text{
and }
\left(
\begin{matrix}
x_2^{(0)}\\y_2^{(0)}
\end{matrix}
\right)=F^n_{t,a}\left(\begin{matrix}
x^{(0)}\\y^{(0)}
\end{matrix}\right),
$$
then, because of Lemma \ref{DFnC}, the fact that $\left(\begin{matrix}
x_2, y_2
\end{matrix}
\right)$ and $\left(\begin{matrix}
x_2^{(0)},y_2^{(0)}
\end{matrix}\right)$ are on the same leaf of the foliation and \eqref{lem:x2minusx2zero}, we have
$$
\text{dist}\left(\left(
\begin{matrix}
x_2\\y_2
\end{matrix}
\right),\left(\begin{matrix}
x_2^{(0)}\\y_2^{(0)}
\end{matrix}\right)\right)=O\left({\lambda^n}\right),
$$
where we used that $D\varphi(0)$ is proportional to $\mu^{n}$. Let 
$$
\left(
\begin{matrix}
x_3\\y_3
\end{matrix}
\right)=F_{t,a}^N\left(
\begin{matrix}
x_2\\y_2
\end{matrix}
\right)
\text{
and }
\left(
\begin{matrix}
x_3^{(0)}\\y_3^{(0)}
\end{matrix}
\right)=F^N_{t,a}\left(\begin{matrix}
x_2^{(0)}\\y_2^{(0)}
\end{matrix}\right),
$$
then 
\begin{equation}\label{eq:x4y4}
x_3=x_3^{(0)}\text{ and }y_3-y_3^{(0)}=O\left({\lambda^n}\right). 
\end{equation}
The next step is to obtain an estimate for the difference of the $x$ coordinates of $\sigma_{t,a}\left(
\begin{matrix}
x_3,y_3
\end{matrix}
\right)$ and $\sigma_{t,a}\left(
\begin{matrix}
x_3^{(0)},y_3^{(0)}
\end{matrix}
\right)$. This can be estimated by iterating by $F_{t,a}^{n+N}$. Lemma \ref{DFnC} says that this map extends distance at most by $\mu^n$. Let $x_5=\left(\sigma_{t,a}\left(
\begin{matrix}
x_3,y_3
\end{matrix}
\right)\right)_x$ and $x^{(0)}_5=\left(\sigma_{t,a}\left(
\begin{matrix}
x^{(0)}_3,y^{(0)}_3
\end{matrix}
\right)\right)_x$ and we obtain 
$$
\left|x_5-x_5^{(0)}\right|=O\left(\left(\lambda\mu\right)^n\right),
$$
where we used \eqref{eq:x4y4}. Finally, by applying $\alpha ^{-1}$ one obtains
$$
\left|f_{t,a}(x_0,y_0)-f_{t,a}(x_0,0)\right|=O\left(\left(\lambda\mu^2\right)^n\right),
$$
where we used that $D\varphi(0)$ is proportional to $\mu^{n}$. The proposition follows by recalling that $\lambda\mu^3<1$.
\end{proof}

\section{Monotonicity of the normalized family}
In the previous section, we constructed families of normalized maps $HF_{t,a}$ with $(t,a)\in\mathcal H_n$. In this section we study the dependence of the normalized maps on the parameters. 
\subsection{Monotonicity in the $a$ direction}
 Fix $t\in [-t_0,t_0]$ and consider the real-analytic one-parameter family  $$\left[-{E},E\right]\ni\beta\mapsto HF_{\beta}=HF_{t,sa_n(t)+{\beta}/{|\mu(t, sa_n(t))|^{2n}}}.$$ 
 From Theorem \ref{firstreturnmapanalytic}, we know that 
 \begin{eqnarray*}
HF_{\beta}\left(\begin{matrix}
x\\y
\end{matrix}\right)
=\left(\begin{matrix}
x^2+\nu +\epsilon_{\beta}(x,y)\\
x
\end{matrix}\right)=\left(\begin{matrix}
f_{\beta}(x,y)\\
x
\end{matrix}\right),
\end{eqnarray*}
 where $\nu$ is a function of $\beta$.
 \begin{theo}\label{prop:HFclosetoFa}
Fix $t\in [-t_0,t_0]$. There exists a real-analytic function $\left[-{E},E\right]\ni\beta\mapsto L(\beta)=D\varphi(0)B_4(c)/\mu^n\neq 0$ such that 
$$
\frac{\partial f_{\beta}}{\partial\beta}=L(\beta)+O\left(\frac{n}{\mu^n(t,sa_n(t))}\right).
$$
\end{theo}
The proof of Theorem \ref{prop:HFclosetoFa} requires the following lemmas. We use the notation of the previous section. 
\begin{lem}\label{lem:ddadyderivatives}
The mixed partial derivatives at the critical point $c$ of $F_{t,a}$ satisfy the following.
\begin{itemize}
\item[-]$\partial^2\alpha/\partial a\partial y,\partial^2\beta/\partial a\partial y,\partial^2\gamma/\partial a\partial y,\partial^2\delta/\partial a\partial y=O(1)$,
\item[-]$\partial^2 A_2/\partial a\partial y,\partial^2 B_2/\partial a\partial y,\partial^2 C_2/\partial a\partial y,\partial^2 D_2/\partial a\partial y=O(n\mu^n)$,
\item[-]$\partial^2 A_4/\partial a\partial y,\partial^2 B_4/\partial a\partial y,\partial^2 C_4/\partial a\partial y,\partial^2 D_4/\partial a\partial y=O(n\mu^{2n})$,
\item[-]$\partial^2\alpha_5/\partial a\partial y,\partial^2\beta_5/\partial a\partial y,\partial^2\gamma_5/\partial a\partial y,\partial^2\delta_5/\partial a\partial y=O(n\mu^{2n})$.
\end{itemize}
\end{lem}
\begin{proof}
The first set of estimates follows by the smoothness of the family of maps $h_{t,a}$. For the second set of estimates, observe that, by \eqref{eq:dunderlinex2dy} and Lemma \ref{lem:apartialderivatives},
$$
\frac{\partial^2\underline x_2}{\partial a\partial y}=O\left(n{\mu^n}\right).
$$
As a consequence,
\begin{eqnarray*}
\frac{\partial^2 A_2}{\partial a\partial y}&=&\frac{\partial}{\partial a}\left[\frac{\partial A}{\partial x}\frac{\partial x_2}{\partial y}+\frac{\partial A}{\partial y}\frac{\partial y_2}{\partial y}\right]\\&=&
\frac{\partial x_2}{\partial y}\left[\frac{\partial^2 A}{\partial x^2}\frac{\partial x_2}{\partial a}+\frac{\partial^2 A}{\partial y\partial x}\frac{\partial y_2}{\partial a}+\frac{\partial^2 A}{\partial a\partial x}\right]+\frac{\partial A}{\partial x}\frac{\partial^2 x_2}{\partial a\partial y}\\
&+&\frac{\partial y_2}{\partial y}\left[\frac{\partial^2 A}{\partial x\partial y}\frac{\partial x_2}{\partial a}+\frac{\partial^2 A}{\partial y^2}\frac{\partial y_2}{\partial a}+\frac{\partial^2 A}{\partial a\partial y}\right]+\frac{\partial A}{\partial y}\frac{\partial^2 y_2}{\partial a\partial y}\\
&=&O\left(n\mu^n\right),
\end{eqnarray*}
where we also used \eqref{eq:dunderlinex2dy}, \eqref{eq:x2aderivatives} and \eqref{eq:y2aderivatives}. Similarly the other estimates in the second statement follows. 
For the third set of estimates, observe that, by \eqref{eq:dx4dy} and Lemma \ref{lem:apartialderivatives}
$$
\frac{\partial^2\underline x_4}{\partial a\partial y}=O\left(n{\mu^{2n}}\right),
$$
and by \eqref{eq:dx4dy}, the fact that $D_2=O\left(\left(\lambda/\mu\right)^n\right)$, see Lemma \ref{lem:D2xypartialderivatives}, we get,
\begin{equation}\label{eq:dx4dybetter}
\frac{\partial \underline x_4}{\partial y}=O\left(\left(\lambda{\mu}\right)^n\right).
\end{equation}
As a consequence,
\begin{eqnarray*}
\frac{\partial^2 A_4}{\partial a\partial y}&=&
\frac{\partial x_4}{\partial y}\left[\frac{\partial^2 A}{\partial x^2}\frac{\partial x_4}{\partial a}+\frac{\partial^2 A}{\partial y\partial x}\frac{\partial y_4}{\partial a}+\frac{\partial^2 A}{\partial a\partial x}\right]+\frac{\partial A}{\partial x}\frac{\partial^2 x_4}{\partial a\partial y}\\
&+&\frac{\partial y_4}{\partial y}\left[\frac{\partial^2 A}{\partial x\partial y}\frac{\partial x_4}{\partial a}+\frac{\partial^2 A}{\partial y^2}\frac{\partial y_4}{\partial a}+\frac{\partial^2 A}{\partial a\partial y}\right]+\frac{\partial A}{\partial y}\frac{\partial^2 y_4}{\partial a\partial y}\\
&=&O\left(n\mu^{2n}\right),
\end{eqnarray*}
where we also used \eqref{eq:x4aderivatives} and \eqref{eq:y4aderivatives}. Similarly the other estimates in the third statement follows. 
Finally, by \eqref{eq:dx5dy}, Lemma \ref{lem:D2xypartialderivatives} and Lemma \ref{lem:apartialderivatives},
$$
\frac{\partial^2\underline x_5}{\partial a\partial y}=O\left(n{\mu^{2n}}\right),
$$
and by \eqref{eq:dx5dy}, the fact that $D_2=O\left(\left(\lambda/\mu\right)^n\right)$, see Lemma \ref{lem:D2xypartialderivatives}, we get,
\begin{equation}\label{eq:dx5dybetter}
\frac{\partial \underline x_5}{\partial y}=O\left(\left(\lambda{\mu}\right)^n\right).
\end{equation}
As a consequence,
\begin{eqnarray*}
\frac{\partial^2 \alpha_5}{\partial a\partial y}
&=&O\left(n\mu^{2n}\right),
\end{eqnarray*}
where we also used \eqref{eq:x5aderivatives} and \eqref{eq:y5aderivatives}. Similarly the other estimates in the last statement follows. 
\end{proof}

The following lemma refers to the function $\Phi$. Recall that $\Phi_2$ is the second component of $\Phi$ and $\phi_{22}=\partial\Phi_2/\partial y$, see \eqref{eq:Dphi}.
\begin{lem}\label{lem:ddaphi22}
Let $(t,a)\in\mathcal H_n$. Then, 
$$
\frac{d\left(\phi_{22}\right)}{da}=O\left(\mu^{4n}\right).
$$
\end{lem}
\begin{proof}
We calculate ${d\left(\phi_{22}\right)}/{da}=\partial^2\Phi_2/\partial a\partial y$ using \eqref{eq:phisecondcomponent}. Observe that, by Lemma \ref{lem:apartialderivatives}, Lemma \ref{lem:ypartialderivatives} and Lemma \ref{lem:ddadyderivatives}, the terms in \eqref{eq:phisecondcomponent} containing a $\lambda^n$ will contribute to the estimate of the $a$-derivative of $\phi_{22}$ with an order of at most $\left(\lambda\mu^4\right)^n$.
Hence, 
\begin{eqnarray*}
\frac{d\left(\phi_{22}\right)}{da}=\frac{\partial^2}{\partial a\partial y}\left[\delta D_2B_4\mu^{2n}+\delta D_2D_4\beta_5\mu^{2n}+\delta D_2B_4\alpha_5\mu^{n}\right]+O\left(\left(\lambda\mu^4\right)^n\right).
\end{eqnarray*}
Observe that the mixed partial derivatives of the coefficients of $\mu^{2n}$, $\mu^n$ are of the same order. Hence, 
\begin{eqnarray*}
\frac{d\left(\phi_{22}\right)}{da}&=&O\left(\frac{\partial^2}{\partial a\partial y}\left[\delta D_2B_4\mu^{2n}+\delta D_2D_4\beta_5\mu^{2n}\right]\right)+				O\left(\left(\lambda\mu^4\right)^n\right).
\end{eqnarray*}
Because $D_2=O\left(\left(\lambda/\mu\right)^n\right)$, see Lemma \ref{lem:D2xypartialderivatives}, Lemma \ref{lem:apartialderivatives}, Lemma \ref{lem:ypartialderivatives} and Lemma \ref{lem:ddadyderivatives} all the terms which do not involve a partial derivative of $D_2$, give a contribution of order at most $\left(n\lambda\mu^3\right)^n$. Hence
\begin{eqnarray*}
\frac{d\left(\phi_{22}\right)}{da}&=&O\left(\frac{\partial}{\partial a}\left[\delta \frac{\partial D_2}{\partial y}B_4\mu^{2n}+\delta \frac{\partial D_2}{\partial y}D_4\beta_5\mu^{2n}\right]\right)\\&+&O\left(\frac{\partial}{\partial y}\left[\delta \frac{\partial D_2}{\partial a}B_4\mu^{2n}+\delta \frac{\partial D_2}{\partial a}D_4\beta_5\mu^{2n}\right]\right)+	O\left(\left(\lambda\mu^4\right)^n\right).
\end{eqnarray*}
Observe that, by \eqref{eq:dx4dybetter} and \eqref{eq:dx5dybetter}, we have $\partial B_4/\partial y,\partial D_4/\partial y,\partial \beta_5/\partial y=O\left(\left(\lambda\mu\right)^{n}\right)$. Using Lemma \ref{lem:apartialderivatives}, Lemma \ref{lem:ypartialderivatives} and Lemma \ref{lem:ddadyderivatives} we reduce to 
\begin{eqnarray*}
\frac{d\left(\phi_{22}\right)}{da}=O\left(\frac{\partial}{\partial a}\left[\delta \frac{\partial D_2}{\partial y}B_4\mu^{2n}+\delta \frac{\partial D_2}{\partial y}D_4\beta_5\mu^{2n}\right]\right)+O\left(n\mu^{3n}\right).
\end{eqnarray*}
The lemma follows by applying again Lemma \ref{lem:apartialderivatives}, Lemma \ref{lem:ypartialderivatives} and Lemma \ref{lem:ddadyderivatives}.
\end{proof}
The following lemma refers to the univalent function $\varphi$ introduced in \eqref{eq:qtaphi}.
\begin{lem}\label{lem:dDphida}
Let $(t,a)\in\mathcal H_n$. Then 
$$
\frac{d\left(D\varphi(0)\right)}{d\beta}=O\left(1\right).
$$
\end{lem} 
\begin{proof}
Let $(t,a)\in\mathcal H_n$ and observe that the map $\tilde F_{t,a}$ takes the horizontal segment $y=c_{t,a}$ to a curve which is tangent to the vertical segment $x=v_{t,a}$. The tangency occurs in the point $(v_{t,a},c_{t,a})$. The curve near the tangency point is given by 
$$
x=\tilde Qy^2+O\left(y^3\right),
$$
with coordinates centered at the tangency point and $\tilde Q\neq 0$. Recall that $\tilde F_{t,a}$ restricted to $y=c_{t,a}$ is given by  $q(x)=\varphi(x-c_{t,a})^2$, see \eqref{eq:qtaphi}. We describe now this map with coordinates in a domain centered around $c_{t,a}$ and in the image around $v_{t,a}$. Then, $$\tilde Qx^2+O(x^3)=\varphi(x^2)=D\varphi(0)x^2+O(x^4),$$ where we used \eqref{eq:varphiz}. Hence, 
\begin{equation}\label{eq:DphiequalQ}
D\varphi(0)=\tilde Q.
\end{equation}
In order to determine $\tilde Q$, consider the point $\underline x=(c_x,c_y+\Delta y)$ where $(c_x, c_y)$ is the critical point of $F_{t,a}$ and $\Delta y$ is small. Then, $$\sigma_{t,a}(\underline x)=\left(c_{t,a}+\Delta s,c_x\right),$$ and
$$
\sigma_{t,a}\left(F^{n+N}_{t,a}(\underline x)\right)=\left(v_{t,a}+X\left(\Delta y\right), c_{t,a}+\Delta s\right),
$$
where $\left(v_{t,a}+X\left(\Delta y\right), c_{t,a}+\Delta s\right)$ is a point on the graph of $\tilde F_{t,a}$ restricted to the line $y=c_{t,a}$. As a consequence,
\begin{eqnarray*}
\tilde Q\left(\Delta s\right)^2+O\left(\left(\Delta s\right)^3\right)=X\left(\Delta y\right)&=&\frac{\partial X}{\partial y}(c)\Delta y+\frac{\partial ^2 X}{\partial y^2}(c)\left(\Delta y\right)^2+O\left(\left(\Delta y\right)^3\right)\\
&=&\Phi_2(c)\Delta y+\frac{\partial  \Phi_2}{\partial y}(c)\left(\Delta y\right)^2+O\left(\left(\Delta y\right)^3\right)\\
&=&\phi_{22}\left(\Delta y\right)^2+O\left(\left(\Delta y\right)^3\right),
\end{eqnarray*}
where we used that $c$ is the critical point, $\Phi_2$ is the second component of $\Phi$ and $\phi_{22}=\partial\Phi_2/\partial y$, see \eqref{eq:Dphi}. From this and \eqref{eq:DphiequalQ} we get
\begin{equation}\label{eq:Qphi22dydssquared}
D\varphi(0)=\phi_{22}\left(\frac{dy}{ds}\right)^2.
\end{equation}
We are now ready to estimate the $a$-derivative of $D\varphi(0)$. Observe that, using the notation from the previous section, we have  $(c_{t,a}+\Delta s)=\left(h^{-1}_{t,a}(\underline x_3)\right)_x$ and by \eqref{eq:dx3dy},
$$
\frac{ds}{dy}=\left(1+\frac{\alpha_3}{\mu^n}\right)\left(B_2\delta\mu^n+A_2\beta\lambda^n\right)+{\beta_3}\left(D_2\delta\mu^n+C_2\beta\lambda^n\right).
$$
It follows, using Lemma \ref{lem:D2xypartialderivatives}, that 
\begin{equation}\label{eq:dyds}
\frac{dy}{ds}=\frac{1}{B_2\delta\mu^n}+O\left(\frac{1}{\mu^{2n}}\right).
\end{equation}
Using the above expression for $ds/dy$ and Lemma \ref{lem:apartialderivatives}, 
\begin{equation}\label{eq:ddadsdy}
\frac{\partial}{\partial a}\left(\frac{ds}{dy}\right)=O\left(n\mu^n\right).
\end{equation}
Using $$0=\frac{\partial}{\partial a}\left(\frac{ds}{dy}\frac{dy}{ds}\right)=\frac{dy}{ds}\frac{\partial}{\partial a}\left(\frac{ds}{dy}\right)+\frac{ds}{dy}\frac{\partial}{\partial a}\left(\frac{dy}{ds}\right),$$
\eqref{eq:dyds} and \eqref{eq:ddadsdy} we have 
\begin{equation}\label{eq:ddadyds}
\frac{\partial}{\partial a}\left(\frac{dy}{ds}\right)=O\left(\frac{n}{\mu^n}\right).
\end{equation}
By taking the $a$-derivative of \eqref{eq:Qphi22dydssquared} and using Lemma \ref{lem:ddaphi22}, \eqref{eq:phi22}, \eqref{eq:dyds}, \eqref{eq:ddadyds}, the fact that $da/d\beta=1/\mu^{2n}(t,sa_n(t))$ and $\mu^{2n}(t,a)/\mu^{2n}(t,sa_n(t))=1+O(n/\mu^{n}(t,sa_n(t)))$, the lemma follows.
\end{proof}
\begin{lem}\label{lem:D2B4underlinex}
For all $\underline x\in\sigma^{-1}\circ Z \left(\text{Dom}\left(HF_{\beta}\right)\right)$,
\begin{eqnarray*}
|\underline x-c|&=&O\left({1}/{\mu^n}\right),\\
D_2(\underline x)&=&O\left({1}/{\mu^n}\right),\\
D_4(\underline x)&=&O\left({1}/{\mu^n}\right),\\
B_4(\underline x)&=&B_4(c)+O\left({1}/{\mu^n}\right).
\end{eqnarray*}
\end{lem}
\begin{proof}
Take $\underline x_0\in\text{Dom}\left(HF_{\beta}\right)$ and compare its image with the image of $(0,0)$. Let 
$Z\left(
\underline x_0
\right)=
\underline{\tilde x}$ and $Z(0,0)=\underline{\tilde x}^{(0)}.$
Because $Z$ contracts distances by $O\left(1/\mu^n\right)$ we get 
\begin{equation}\label{eq:distanceswithcriticalpoint}
\left |\underline{\tilde x}-\underline{\tilde x}^{(0)}\right|,\left |\tilde F\left(\underline{\tilde x}\right)-\tilde F\left(\underline{\tilde x}^{(0)}\right)\right|, \left |\tilde F^2\left(\underline{\tilde x}\right)-\tilde F^2\left(\underline{\tilde x}^{(0)}\right)\right|=O\left(\frac{1}{\mu^n}\right).
\end{equation}
Let $\underline x=(x,y)=\sigma^{-1}\left(\underline{\tilde x}\right)$, $\underline x^{(0)}=\sigma^{-1}\left(\underline{\tilde x}^{(0)}\right)=\left(c_x,c_y\right)$. Because $\underline x\in HB$, we have that $y$ is proportional to $1/\mu^n$. Hence, by \eqref{eq:distanceswithcriticalpoint} and the definition of $\sigma$, the first estimate follows. Let
$\underline x_2=(x_2,y_2)=F^n\left(\underline{ x}\right)$, $\underline x_2^{(0)}=(x_2^{(0)},y_2^{(0)})=F^n\left(\underline{x}^{(0)}\right)$, $\underline x_3=(x_3,y_3)=F^{N}\left(\underline{ x}_2\right)$, $\underline x_3^{(0)}=F^{N}\left(\underline{ x}_2^{(0)}\right)$, $\underline x_4=(x_4,y_4)=F^{n}\left(\underline{ x}_3\right)$ and $\underline x_5=(x_5,y_5)=F^{N}\left(\underline{ x}_4\right)$.
By applying $F^n$ to $\underline x$, $\underline x_3$ and  $\underline x^{(0)}$, \eqref{eq:distanceswithcriticalpoint} and Lemma \ref{DFnC} imply that 
\begin{equation}\label{eq:x2x4distanceswithcriticalpoint}
\left |x_2-{x}_2^{(0)}\right|, \left |x_4-{x}_2^{(0)}\right|=O\left({\lambda^n}\right).
\end{equation}
Because of \eqref{eq:distanceswithcriticalpoint} we have
\begin{equation}\label{eq:x3x5distanceswithcriticalpoint}
\left |x_3-c_x\right|, \left |x_5-c_x\right|=O\left(\frac{1}{\mu^n}\right).
\end{equation}
By \eqref{eq:x2x4distanceswithcriticalpoint}, \eqref{eq:x3x5distanceswithcriticalpoint} and by taking the preimages under $F^N$ of the points $\underline x_3$, $\underline x_5$ and  $\underline x_3^{(0)}$ we obtain
\begin{equation}\label{eq:y2y4distanceswithcriticalpoint}
\left |y_2-{y}_2^{(0)}\right|, \left |y_4-{y}_2^{(0)}\right|=O\left(\frac{1}{\mu^n}\right).
\end{equation}
From \eqref{eq:x2x4distanceswithcriticalpoint} and \eqref{eq:y2y4distanceswithcriticalpoint} we get 
\begin{equation*}
\left |\underline x_2-\underline{x}_2^{(0)}\right|, \left |\underline x_4-\underline {x}_2^{(0)}\right|=O\left(\frac{1}{\mu^n}\right).
\end{equation*}
Observe that the image of the points $\underline x_2$ and $\underline x_4$ under the linearization $h$ coincides with the points $\underline x_2$, $\underline x_4$ as introduced before in the real domain. Because $h$ is a diffeomorphism, the distances are also of the order $1/\mu^n$. Finally, since $D_2(\underline x)=D(\underline x_2)$, $D_4(\underline x)=D(\underline x_4)$, the fact that $D_2(c)=O\left(\lambda^n/{\mu^n}\right)$, see Lemma \ref{lem:D2xypartialderivatives}, and $B_4(\underline x)=B(\underline x_4)$ the lemma follows.
\end{proof}
\begin{lem}\label{lem:dx5daunderlinex}
For all $\underline x\in\sigma^{-1}\circ Z \left(\text{Dom}\left(HF_{\beta}\right)\right)$,
\begin{eqnarray*}
\frac{\partial x_5}{\partial a}(\underline x)&=&B_4(\underline x)\mu^n+O(n).
\end{eqnarray*}
\end{lem}
\begin{proof}
Recall that $q_3=q_3(t,0)$ is the point where $D(q_3,t,0)=0$ and that $F_{t,0}^N(q_3)=q_1$ where $q_1$ is the tangency at $a=0$. Because $\partial D(q_3,t,0)/\partial y\neq 0$, there is a curve trough $q_3$ transversal to the $y$-axis where $D(\underline x,t,0)=0$ for all points $\underline x$ in the curve. The transversality implies that there is a curve $D(\underline x,t,a)=0$, transversally intersecting the $y$-axis in the point $q_3(t,a)$ at the distance of $O(1/\mu^n)$, because $(t,a)\in\mathcal H_n$. From Lemma \ref{lem:D2B4underlinex} we know that $D_2(\underline x)=O\left(1/\mu^n\right)$. This implies that the distance from $\underline x_2$ to this curve is of the $O\left(1/\mu^n\right)$. Moreover $x_2=O(\lambda^n)$. Hence, 
\begin{equation}\label{eq:distx2q3}
\text{dist}\left(\underline x_2,q_3\right)=O\left(\frac{1}{\mu^n}\right).
\end{equation}
Because $\partial\left(F^N_{t,0}(q_3)\right)_y/\partial a=1$, we get 
$$
\frac{\partial\left(F^N_{t,a}(\underline x_2)\right)_y}{\partial a}=1+O\left(\frac{1}{\mu^n}\right).
$$
Now \eqref{eq:y3aderivatives}, becomes 
$$
\frac{\partial y_3}{\partial a}(\underline x)=C_2\frac{\partial x_2}{\partial a}+D_2\frac{\partial y_2}{\partial a}+1+O\left(\frac{1}{\mu^n}\right)=1+O\left(\frac{n}{\mu^n}\right),
$$
where we also used \eqref{eq:x2aderivatives}, \eqref{eq:y2aderivatives} and Lemma \ref{lem:D2B4underlinex}.
Now \eqref{eq:y4aderivatives}, becomes 
$$
\frac{\partial y_4}{\partial a}(\underline x)=\mu^n+O(n),
$$ 
and \eqref{eq:x5aderivatives}, becomes 
$$ 
\frac{\partial x_5}{\partial a}(\underline x)=B_4(\underline x)\mu^n+O(n),
$$
where we used \eqref{eq:x4aderivatives}.
\end{proof}
\noindent
{\it Proof of Theorem \ref{prop:HFclosetoFa}.}
Let $\underline x_0=(x_0,y_0)\in\text{Dom}\left(HF_{\beta}\right)$, $\underline{\tilde x}=(\tilde x,\tilde y)=Z(\underline x_0)$ and $\underline x=(x,y)=\sigma^{-1}\left(\underline{\tilde x}\right)\in\sigma^{-1}\circ Z \left(\text{Dom}\left(HF_{\beta}\right)\right)$. Because $\underline{\tilde x}=(c_{t,a},c_{t,a})+\underline x_0/D\varphi(0)$, using Proposition \ref{prop:aderivativescriticalpointscriticalvalue}, the fact that $D\varphi(0)$ is proportional to $\mu^n$, Lemma \ref{lem:dDphida} and the fact that $da/d\beta=O\left(1/\mu^{2n}\right)$ we have
\begin{equation}\label{eq:dtildexdbeta}
\frac{d\underline{\tilde x}}{d\beta}=O\left(\frac{n}{\mu^{2n}}\right),
\end{equation}
and because $x=\tilde y$,
\begin{equation}\label{eq:dxdbeta}
\frac{d{ x}}{d\beta}=O\left(\frac{n}{\mu^{2n}}\right).
\end{equation}
The estimate for $dy/d\beta$ needs some preparation.
By \eqref{eq:dx3dx}, \eqref{eq:dxdbeta}, \eqref{eq:dx3dy}, \eqref{eq:x3aderivatives} and the fact that $da/d\beta=O\left(1/\mu^{2n}\right)$ we have
\begin{equation}\label{eq:dx3dbeta}
\frac{d{ x_3}}{d\beta}=\frac{\partial { x_3}}{\partial x}\frac{d{ x}}{d\beta}+\frac{\partial { x_3}}{\partial y}\frac{d{ y}}{d\beta}+\frac{\partial { x_3}}{\partial a}\frac{d{ a}}{d\beta}=\left(B_2\delta\mu^n+O(\lambda^n)\right)\frac{d{ y}}{d\beta}+O\left(\frac{n}{\mu^{2n}}\right),
\end{equation}
where, by Remark \ref{rem:deltaBawayfromzero}, $B_2$ and $\delta$ are uniformly away from zero. Similarly, by \eqref{eq:dx3dx}, \eqref{eq:dxdbeta}, \eqref{eq:dx3dy}, \eqref{eq:y3aderivatives} and Lemma \ref{lem:D2B4underlinex} we have 
\begin{equation}\label{eq:dy3dbeta}
\frac{d{ y_3}}{d\beta}=O\left(1\right)\frac{d{ y}}{d\beta}+O\left(\frac{1}{\mu^{2n}}\right). 
\end{equation}
Observe now that 
$
\tilde x=\left(h^{-1}(\underline x_3)\right)_x
$
which implies, using \eqref{eq:dx3dbeta} and \eqref{eq:dy3dbeta},
\begin{eqnarray*}
\frac{d {\tilde x}}{d\beta}&=&\left(1+\frac{\alpha_3}{\mu^{n}}\right)\left[\left(B_2\delta\mu^n+O(\lambda^n)\right)\frac{d{ y}}{d\beta}+O\left(\frac{n}{\mu^{2n}}\right)\right]+{\beta_3}O\left(1\right)\frac{d{ y}}{d\beta}+O\left(\frac{1}{\mu^{2n}}\right)\\&=&\left(B_2\delta\mu^n+O(1)\right)\frac{d{ y}}{d\beta}+O\left(\frac{n}{\mu^{2n}}\right).
\end{eqnarray*}
By \eqref{eq:dtildexdbeta}, we get 
\begin{equation}\label{eq:dydbeta}
\frac{d{ y}}{d\beta}=O\left(\frac{n}{\mu^{3n}}\right).
\end{equation}
Observe now that, by \eqref{eq:dx5dx}, Lemma \ref{lem:D2B4underlinex}, \eqref{eq:dxdbeta}, \eqref{eq:dx5dy}, again Lemma \ref{lem:D2B4underlinex}, \eqref{eq:dydbeta} and Lemma \ref{lem:dx5daunderlinex},
\begin{equation}\label{eq:dx5dbeta}
\frac{d{ x_5}}{d\beta}=\frac{\partial { x_5}}{\partial x}\frac{d{ x}}{d\beta}+\frac{\partial { x_5}}{\partial y}\frac{d{ y}}{d\beta}+\frac{\partial { x_5}}{\partial a}\frac{d{ a}}{d\beta}=\frac{B_4(\underline x)}{\mu^n}+O\left(\frac{n}{\mu^{2n}}\right).
\end{equation}
Similarly, by \eqref{eq:dx5dx}, Lemma \ref{lem:D2B4underlinex}, \eqref{eq:dxdbeta}, \eqref{eq:dx5dy}, again Lemma \ref{lem:D2B4underlinex}, \eqref{eq:dydbeta} and \eqref{eq:y5aderivatives}
\begin{equation}\label{eq:dy5dbeta}
\frac{d{ y_5}}{d\beta}=O\left(\frac{n}{\mu^{2n}}\right).
\end{equation}
Because 
$
\left(\tilde F(\underline{\tilde x})\right)_x=\left(h^{-1}(\underline x_5)\right)_x
$, using \eqref{eq:dx5dbeta} and \eqref{eq:dy5dbeta}
\begin{eqnarray}\label{eq:dtildeFdbeta}\nonumber
\frac{d {\left(\tilde F(\underline{\tilde x})\right)_x}}{d\beta}&=&\left(1+\frac{\alpha_5}{\mu^{n}}\right)\left[\frac{B_4(\underline x)}{\mu^n}+O\left(\frac{n}{\mu^{2n}}\right)\right]+{\beta_5}O\left(\frac{n}{\mu^{2n}}\right)+O\left(\frac{\partial{h^{-1}}}{\partial a}\right)\frac{1}{\mu^{2n}}\\&=&\frac{B_4(\underline x)}{\mu^n}+O\left(\frac{n}{\mu^{2n}}\right).
\end{eqnarray}
Finally observe that $$f_{\beta}(\underline x_0)=\alpha^{-1}\left(\left(\tilde F(\underline{\tilde x})\right)_x\right)=\left(\left(\tilde F(\underline{\tilde x})\right)_x-c_{t,a}\right)D\varphi(0).$$
As a consequence, 
\begin{eqnarray*}
\frac{d f_{\beta}(\underline x_0)}{d\beta}&=&D\varphi(0)\frac{d {\left(\tilde F(\underline{\tilde x})\right)_x}}{d\beta}-D\varphi(0)\frac{\partial c_{t,a}}{\partial a}\frac{da}{d\beta}+ \left(\left(\tilde F(\underline{\tilde x})\right)_x-c_{t,a}\right)\frac{d D\varphi(0)}{d\beta}\\&=&B_4(\underline x)\frac{D\varphi(0)}{\mu^n}+O\left(\frac{n}{\mu^{n}}\right),
\end{eqnarray*}
where we used \eqref{eq:dtildeFdbeta}, the fact that $D\varphi(0)$ is proportional to $\mu^n$, Proposition \ref{prop:aderivativescriticalpointscriticalvalue}, Lemma \ref{lem:dDphida} and the fact that $\left(\tilde F(\underline{\tilde x})\right)_x\in\alpha(U_{t,a})$ which has diameter proportional to $1/\mu^n$. 
In conclusion, using Lemma \ref{lem:D2B4underlinex},
\begin{eqnarray*}
\frac{d f_{\beta}}{d\beta}&=&B_4(c)\frac{D\varphi(0)}{\mu^n}+O\left(\frac{n}{\mu^{n}}\right),
\end{eqnarray*}
where $B_4(c){D\varphi(0)}/{\mu^n}$ is uniformly bounded away from zero.
\qed

\subsection{Monotonicity in the $t$ direction}
The aim of this subsection is  to study the dependence of the normalized map on the $t$-parameter. In order to achieve that, we start by calculating the movement of the critical point while changing the parameter $t$.
\begin{prop}\label{prop:tderivativescriticalpointscriticalvalue}
For every $(t,a)\in\mathcal H_n$, 
\begin{eqnarray*}
\frac{\partial c_{t,a}}{\partial t}&=&\frac{\partial c_{x}}{\partial t}=\frac{\partial v_{x}}{\partial t}=O\left(\frac{n}{\mu^n(t,a)}\right), 
\\\frac{\partial v_{t,a}}{\partial t}&=&B_4n\mu^{n-1}\frac{\partial\mu}{\partial t}y_3+O\left(\frac{n}{\mu^n(t,a)}\right).
\end{eqnarray*}
\end{prop}
The proof of this proposition needs some preparation. Fix $(t,a)\in\mathcal H_n$. Let $(x,y)\in HB(t,a)\cap\mathbb R^2$ and  $\underline x=(x,y)$, $\underline x_1$,
$\underline x_2$, $\underline x_3$, $\underline x_4$, $\underline x_5$ as introduced before. Observe that all coefficients introduced, $\alpha,\beta,\dots,A_2,B_2,\dots,\alpha_3,\beta_3,\dots,A_4,B_4,\dots,\alpha_5,\dots$ are all functions of the point $(x,y)\in HB(t,a)\cap\mathbb R^2$ and $t$. We will vary the parameter $t$. 
\begin{lem}\label{lem:dhdt}
 The coordinate change $h_{t,a}$ restricted to $\sigma_{t,a}^{-1}\circ Z_{t,a} \left(\text{Dom}\left(HF_{t,a}\right)\right)$ satisfies
$$
\frac{\partial \left(h_{t,a}\right)_y}{\partial t}=O\left(\frac{1}{\mu^{2n}(t,a)}\right).
$$
\end{lem}
\begin{proof} As in the proof of Lemma \ref{lem:Dh}, write the Taylor expansion of $\left(h_{t,a}\right)_y$ centered around the point $(2,0)$. Then 
$$\left(h_{t,a}\right)_y=y( 1+x\varphi_x+y\varphi_y).$$
It follows that 
$$
\frac{\partial \left(h_{t,a}\right)_y}{\partial t}=y\left(x\frac{\partial \varphi_x}{\partial t}+y\frac{\partial \varphi_y}{\partial t}\right)=O\left(\frac{1}{\mu^{2n}}\right)
$$
where we used the first equality in Lemma \ref{lem:D2B4underlinex} with $c_x=2$.
\end{proof}
\begin{lem}\label{lem:dFNdt}
For all $\underline x=(x,y)\in F_{t,a}^n\circ\sigma_{t,a}^{-1}\circ Z_{t,a} \left(\text{Dom}\left(HF_{t,a}\right)\right)$,
\begin{eqnarray*}
\frac{\partial D(\underline x)}{\partial t}&=&O\left(\frac{1}{\mu^{n}(t,a)}\right),\\
\frac{\partial \left(F_{t,a}^N(\underline x)\right)_x}{\partial t}&=&O\left(\frac{1}{\mu^{n}(t,a)}\right),\\
\frac{\partial \left(F_{t,a}^N(\underline x)\right)_y}{\partial t}&=&O\left(\frac{1}{\mu^{2n}(t,a)}\right).
\end{eqnarray*}
\end{lem}
\begin{proof}
From the first statement of Lemma \ref{lem:D2B4underlinex} and the fact that $a=O\left(1/\mu^n\right)$, we have 
\begin{equation}\label{eq:distq3tox}
|q_3(t)-\underline x|=O\left(\frac{1}{\mu^n}\right),
\end{equation}
and by the hypothesis $q_3(t,0)=(0,1)$
\begin{equation}\label{eq:distq3toy}
|1-y|=O\left(\frac{1}{\mu^n}\right).
\end{equation}
Moreover, by Lemma \ref{lem:Dh},
\begin{equation}\label{eq:xorderlambdamu}
| x|=O\left(\left(\lambda\mu\right)^n\right).
\end{equation}
Observe that $D(q_3(t,0),t,0)=0$ when $a=0$. Hence, for every $t$, $\partial D(q_3(t,0),t,0)/\partial t=0$. The first statement follows by applying \eqref{eq:distq3tox}. Write $\left(F_{t,a}^N(\underline x)\right)_x$ in coordinates centered in the domain around $q_3$ and in the image around $q_1$. One gets 
$$
\left(F_{t,a}^N(\underline x)\right)_x=Ax+By+aH,
$$
where $A,B,H$ are $\Cd$ functions. Hence,
$$
\frac{\partial \left(F_{t,a}^N(\underline x)\right)_x}{\partial t}=O\left(|x|+|y|+|a|\right).
$$
The second statement follows.
Similarly, write $\left(F_{t,a}^N(\underline x)\right)_y$ in coordinates centered in the domain around $q_3$ and in the image around $q_1$. One gets 
$$
\left(F_{t,a}^N(\underline x)\right)_y=Cx+Q_{11}x^2+Q_{12}xy+Q_{22}y^2+a\left(1+V_1x+V_2y\right)
$$
where $C,Q_{11},Q_{12},Q_{22},V_1,V_2$ are $\Cd$ functions. Hence
$$
\frac{\partial \left(F_{t,a}^N(\underline x)\right)_x}{\partial t}=O\left(|x|+|xy|+|y^2|+|ax|+|ay|\right).
$$
The last statement follows by using \eqref{eq:distq3tox}, \eqref{eq:xorderlambdamu} and the fact that $\lambda\mu^3<1$.
\end{proof}
\begin{lem}\label{lem:tpartialderivatives}
The partial derivatives at the critical point $c$ of $F_{t,a}$ satisfy the following.
\begin{itemize}
\item[-]$\partial\alpha/\partial t,\partial \beta/\partial t,\partial\gamma/\partial t,\partial \delta/\partial t=O(1)$,
\item[-]$\partial A_2/\partial t,\partial B_2/\partial t,\partial C_2/\partial t,\partial D_2/\partial t=O(n)$,
\item[-]$\partial\alpha_3/\partial t,\partial \beta_3/\partial t,\partial\gamma_3/\partial t,\partial \delta_3/\partial t=O(n)$,
\item[-]$\partial A_4/\partial t,\partial B_4/\partial t,\partial C_4/\partial t,\partial D_4/\partial t=O(n)$,
\item[-]$\partial\alpha_5/\partial t,\partial \beta_5/\partial t,\partial\gamma_5/\partial t,\partial \delta_5/\partial t=O(n)$.
\end{itemize}
\end{lem}
\begin{proof}
The first set of estimates follows by the smoothness of the family of maps $h_{t,a}$. For the second set of estimates, observe that, because $h_{t,a}(x,0)=(x,0)$ and $y=O\left(1/\mu^n\right)$, we have 
$$
\frac{\partial x_1}{\partial t}=O\left(\frac{1}{\mu^n}\right),
$$
and by Lemma \ref{lem:dhdt}, 
$$
\frac{\partial y_1}{\partial t}=O\left(\frac{1}{\mu^{2n}}\right).
$$
Recall that $\underline x_2=\hat F^n_{t,a}\left(\underline x_1,t,a\right)$. It follows that,
\begin{equation}\label{eq:x2tderivatives}
\frac{\partial x_2}{\partial t}=n\lambda^{n-1}\frac{\partial\lambda}{\partial t}x_1+\lambda^n\frac{\partial x_1}{\partial t}=O\left(n\lambda^n\right),
\end{equation}
and 
\begin{equation}\label{eq:y2tderivatives}
\frac{\partial y_2}{\partial t}=n\mu^{n-1}\frac{\partial\mu}{\partial t}y_1+\mu^n\frac{\partial y_1}{\partial t}=n\mu^{n-1}\frac{\partial\mu}{\partial t}y_1+O\left(\frac{1}{\mu^n}\right).
\end{equation}
Observe that, by our initial hypotheses $\partial\mu/\partial t\neq 0$ and that $y_1$ is proportional to $1/\mu^n$. As a consequence, ${\partial y_2}/{\partial t}$ is proportional to $n$. 
Moreover
$$
\frac{\partial A_2}{\partial t}=\frac{\partial A}{\partial x}\frac{\partial x_2}{\partial t}+\frac{\partial A}{\partial y}\frac{\partial y_2}{\partial t}+\frac{\partial A}{\partial t}=O\left(n\right).
$$
Similarly the other estimates in the second statement follows. 
Observe that, using Lemma \ref{lem:dFNdt}, 
\begin{equation}\label{eq:D2tderivatives}
\frac{\partial D_2}{\partial t}=\frac{\partial D}{\partial y}n\mu^{n-1}\frac{\partial\mu}{\partial t}y_1+O\left(\frac{1}{\mu^n}\right),
\end{equation}
and ${\partial D}/{\partial y}\neq 0$.
Recall that $\underline x_3=\hat F^N_{t,a}\left(\underline x_2,t,a\right)$. By Lemma \ref{lem:dFNdt}, 
\begin{equation}\label{eq:x3tderivatives}
\frac{\partial x_3}{\partial t}=A_2\frac{\partial x_2}{\partial t}+B_2\frac{\partial y_2}{\partial t}+O\left(\frac{1}{\mu^n}\right)=B_2 n\mu^{n-1}\frac{\partial\mu}{\partial t}y_1+O\left(\frac{1}{\mu^n}\right),
\end{equation}
and 
\begin{eqnarray}\label{eq:y3tderivativesfor}
\nonumber\frac{\partial y_3}{\partial t}&=&C_2\frac{\partial x_2}{\partial t}+D_2\frac{\partial y_2}{\partial t}+O\left(\frac{1}{\mu^{2n}}\right)\\&=&
C_2\frac{\partial x_2}{\partial t}+D_2n\mu^{n-1}\frac{\partial\mu}{\partial t}y_1+D_2O\left(\frac{1}{\mu^n}\right)+O\left(\frac{1}{\mu^{2n}}\right)\\\label{eq:y3tderivativesorder} 
&=&O\left(\frac{1}{\mu^{2n}}\right),
\end{eqnarray}
where we used \eqref{eq:x2tderivatives}, \eqref{eq:y2tderivatives} and that $D_2=O\left(\lambda^n/\mu^n\right)$, see Lemma \ref{lem:D2xypartialderivatives}. The third set of estimates follow.
Recall now that $\underline x_4=\hat F^n_{t,a}\left(\underline x_3,t,a\right)$. Hence,
\begin{equation}\label{eq:x4tderivatives}
\frac{\partial x_4}{\partial t}=n\lambda^{n-1}\frac{\partial\lambda}{\partial t}x_3+\lambda^n\frac{\partial x_3}{\partial t}=O\left(n\lambda^n\right),
\end{equation}
and 
\begin{equation}\label{eq:y4tderivatives}
\frac{\partial y_4}{\partial t}=n\mu^{n-1}\frac{\partial\mu}{\partial t}y_3+\mu^n\frac{\partial y_3}{\partial t}=n\mu^{n-1}\frac{\partial\mu}{\partial t}y_3+O\left(\frac{1}{\mu^n}\right),
\end{equation}
where we used \eqref{eq:x3tderivatives}, \eqref{eq:y3tderivativesorder}. Observe that, by our initial hypotheses $\partial\mu/\partial t\neq 0$ and that $y_3$ is proportional to $1/\mu^n$. As a consequence ${\partial y_4}/{\partial t}$ is proportional to $n$. 
Hence,
$$
\frac{\partial A_4}{\partial t}=\frac{\partial A}{\partial x}\frac{\partial x_4}{\partial t}+\frac{\partial A}{\partial y}\frac{\partial y_4}{\partial t}+\frac{\partial A}{\partial t}=O\left(n\right).
$$
Similarly the other estimates in the fourth statement follows. 
Finally, by Lemma \ref{lem:dFNdt},
\begin{equation}\label{eq:x5tderivatives}
\frac{\partial x_5}{\partial t}=A_4\frac{\partial x_4}{\partial t}+B_4\frac{\partial y_4}{\partial t}+O\left(\frac{1}{\mu^n}\right)=B_4n\mu^{n-1}\frac{\partial\mu}{\partial t}y_3+O\left(\frac{1}{\mu^n}\right),
\end{equation}
and 
\begin{equation}\label{eq:y5tderivatives}
\frac{\partial y_5}{\partial t}=C_4\frac{\partial x_4}{\partial t}+D_4\frac{\partial y_4}{\partial t}+O\left(\frac{1}{\mu^{2n}}\right)=O\left(\frac{n}{\mu^n}\right),
\end{equation}
where we used \eqref{eq:x4tderivatives}, \eqref{eq:y4tderivatives} and Lemma \ref{lem:D2B4underlinex}.
As before the last set follows.
\end{proof}

\noindent
{\it Proof of Proposition \ref{prop:tderivativescriticalpointscriticalvalue}.}
Observe that 
$$
\frac{\partial c}{\partial t}=\left(D\Phi\right)^{-1}\frac{\partial\Phi}{\partial t}.
$$
Let $\Phi=(\Phi_1,\Phi_2)$. We start by calculating ${\partial\Phi_1}/{\partial t}$. Observe that $\Phi_1=\left(h_{t,a}^{-1}(\underline x_3)\right )_x$. As a consequence
\begin{equation}\label{eq:dphi1dt}
\frac{\partial\Phi_1}{\partial t}=\left(1+\frac{\alpha_3}{\mu^n}\right)\frac{\partial x_3}{\partial t}+{\beta_3}\frac{\partial y_3}{\partial t}+\frac{\partial \left(h_{t,a}^{-1}\right)_x}{\partial t}=B_2 n\mu^{n-1}\frac{\partial\mu}{\partial t}y_1+O\left(\frac{n}{\mu^n}\right).
\end{equation}
where we used \eqref{eq:x3tderivatives} and \eqref{eq:y3tderivativesorder}. In order to calculate ${\partial\Phi_2}/{\partial t}$ we take the $t$-derivatives of \eqref{eq:phisecondcomponent}. Using Lemma \ref{lem:tpartialderivatives}, the partial derivatives of all terms in \eqref{eq:phisecondcomponent} with a factor $\lambda^n$ will give a contribution of $O\left(n\left(\lambda\mu\right)^n\right)$. We get
\begin{eqnarray*}
\frac{\partial\Phi_2}{\partial t}&=&\frac{\partial}{\partial t}\left[\delta D_2B_4\mu^{2n}+\delta D_2D_4\beta_5\mu^{2n}+\delta D_2B_4\alpha_5\mu^{n}\right]\\&+&O\left(n\left(\lambda\mu\right)^n\right).
\end{eqnarray*}
Using Lemma \ref{lem:tpartialderivatives}, \eqref{eq:D2tderivatives}, the fact that $D_2=O\left(\lambda^n/\mu^n\right)$, see Lemma \ref{lem:D2xypartialderivatives} and Lemma \ref{lem:D2B4underlinex}, we get 
\begin{eqnarray}\label{eq:dphi2dt}
\frac{\partial\Phi_2}{\partial t}&=&\delta B_4\frac{\partial D}{\partial y}n\mu^{n-1}\frac{\partial\mu}{\partial t}y_1\mu^{2n}+O\left(n\mu^{n}\right).
\end{eqnarray}
By \eqref{eq:phi11}, \eqref{eq:phi12}, \eqref{eq:phi21}, \eqref{eq:phi22}, using Lemma \ref{lem:D2B4underlinex} and \eqref{eq:dD2dyexact} we have 
$$ 
D\Phi=\left(\begin{matrix}
-1+O(\gamma)& \delta B_2\mu^n+O(1)\\
\delta B_4\frac{\partial D_2}{\partial x}\mu^{2n}+O(\mu^n)& \delta^2 B_4\frac{\partial D}{\partial y}\mu^{3n}+O(\mu^{2n})
\end{matrix}
\right),
$$
and $$\text{det}\left(D\Phi\right)=-\delta^2B_4\frac{\partial D}{\partial y}\mu^{3n}\left[1+\frac{B_2{\partial D_2}/{\partial x}}{{\partial D}/{\partial y}}-{O(\gamma)}\right]=-\delta^2B_4\frac{\partial D}{\partial y}\mu^{3n}\left[1+\chi\right],$$
where $\chi$ is close to zero, see Lemma \ref{lem:D2B4underlinex} and Remark \ref{rem:alphabetagammadeltasmall}.
Moreover 
\begin{eqnarray*}
\left(\left(D\Phi\right)^{-1}\right)_{11}&=&-\left[\frac{1+O(1/\mu^n)}{1+\chi}\right]\\
\left(\left(D\Phi\right)^{-1}\right)_{12}&=&\frac{B_2}{\delta B_4{\partial D}/{\partial y}}\frac{1}{\mu^{2n}}\left[\frac{1+O(1/\mu^n)}{1+\chi}\right].
\end{eqnarray*}
As a consequence, using \eqref{eq:dphi1dt} and \eqref{eq:dphi2dt} a cancellation happens and leads to   
$$
\frac{\partial c_x}{\partial t}=\frac{\partial c_{t,a}}{\partial t}=\frac{\partial v_{x}}{\partial t}=\left(\left(D\Phi\right)^{-1}\right)_{11}\frac{\partial \Phi_1}{\partial t}+\left(\left(D\Phi\right)^{-1}\right)_{12}\frac{\partial \Phi_2}{\partial t}=O\left(\frac{n}{\mu^n}\right).
$$
The first equation follows. 
Observe that $v_{t,a}=\left(h^{-1}_{t,a}\left(x_5\right)\right)_x$. Hence
$$
\frac{\partial v_{t,a}}{\partial t}=\left(1+\frac{\alpha_5}{\mu^n}\right)\frac{\partial x_5}{\partial t}+{\beta_5}\frac{\partial y_5}{\partial t}+O\left(\frac{1}{\mu^n}\right)=B_4n\mu^{n-1}\frac{\partial\mu}{\partial t}y_3+O\left(\frac{n}{\mu^n}\right),
$$
where we used \eqref{eq:x5tderivatives} and \eqref{eq:y5tderivatives}. 
\qed
\bigskip

We are now ready to state the main theorem on the monotonicity of the normalized family in the $t$ direction.
 \begin{theo}\label{prop:dftdt}
Fix $(t,a)\in \mathcal H_n$. There exists a real-analytic function $M:\mathcal H_n\to\mathbb R$ which is bounded away from zero, $$M(t,a)=\frac{D\varphi(0)B_4(c)y_3}{\mu}\frac{\partial\mu}{\partial t},$$ such that 
$$
\frac{\partial f_{t,a}}{\partial t}=M(t,a)n\mu^n(t,a)+O\left(n\right).
$$
\end{theo}
The proof of Theorem \ref{prop:dftdt} requires the following lemmas. We use the notation of the previous section. 
\begin{lem}\label{lem:ddtdyderivatives}
The mixed partial derivatives at the critical point $c$ of $F_{t,a}$ satisfy the following.
\begin{itemize}
\item[-]$\partial^2\alpha/\partial t\partial y,\partial^2\beta/\partial t\partial y,\partial^2\gamma/\partial t\partial y,\partial^2\delta/\partial t\partial y=O(1)$,
\item[-]$\partial^2 A_2/\partial t\partial y,\partial^2 B_2/\partial t\partial y,\partial^2 C_2/\partial t\partial y,\partial^2 D_2/\partial t\partial y=O(n\mu^n(t,a))$,
\item[-]$\partial^2 A_4/\partial t\partial y,\partial^2 B_4/\partial t\partial y,\partial^2 C_4/\partial t\partial y,\partial^2 D_4/\partial t\partial y=O(n\mu^{2n}(t,a))$,
\item[-]$\partial^2\alpha_5/\partial a\partial y,\partial^2\beta_5/\partial a\partial y,\partial^2\gamma_5/\partial a\partial y,\partial^2\delta_5/\partial a\partial y=O(n\mu^{2n}(t,a))$.
\end{itemize}
\end{lem}
\begin{proof}
The first set of estimates follows by the smoothness of the family of maps $h_{t,a}$. For the second set of estimates, observe that, by \eqref{eq:dunderlinex2dy} and Lemma \ref{lem:tpartialderivatives}
$$
\frac{\partial^2\underline x_2}{\partial t\partial y}=O\left(n{\mu^n}\right).
$$ 
As a consequence, 
\begin{eqnarray*}
\frac{\partial^2 A_2}{\partial t\partial y}&=&\frac{\partial}{\partial t}\left[\frac{\partial A}{\partial x}\frac{\partial x_2}{\partial y}+\frac{\partial A}{\partial y}\frac{\partial y_2}{\partial y}\right]\\&=&
\frac{\partial x_2}{\partial y}\left[\frac{\partial^2 A}{\partial x^2}\frac{\partial x_2}{\partial t}+\frac{\partial^2 A}{\partial y\partial x}\frac{\partial y_2}{\partial t}+\frac{\partial^2 A}{\partial t\partial x}\right]+\frac{\partial A}{\partial x}\frac{\partial^2 x_2}{\partial t\partial y}\\
&+&\frac{\partial y_2}{\partial y}\left[\frac{\partial^2 A}{\partial x\partial y}\frac{\partial x_2}{\partial t}+\frac{\partial^2 A}{\partial y^2}\frac{\partial y_2}{\partial t}+\frac{\partial^2 A}{\partial t\partial y}\right]+\frac{\partial A}{\partial y}\frac{\partial^2 y_2}{\partial t\partial y}\\
&=&O\left(n\mu^n\right),
\end{eqnarray*} 
where we also used \eqref{eq:dunderlinex2dy}, \eqref{eq:x2tderivatives} and \eqref{eq:y2tderivatives}. Similarly the other estimates in the second statement follows. 
For the third set of estimates, observe that, by \eqref{eq:dx4dy} and Lemma \ref{lem:tpartialderivatives}
$$
\frac{\partial^2\underline x_4}{\partial t\partial y}=O\left(n{\mu^{2n}}\right), 
$$
and by \eqref{eq:dx4dy}, the fact that $D_2=O\left(\left(\lambda/\mu\right)^n\right)$, see Lemma \ref{lem:D2xypartialderivatives}, we get
$$
\frac{\partial \underline x_4}{\partial y}=O\left(\left(\lambda{\mu}\right)^n\right).
$$
As a consequence,
\begin{eqnarray*}
\frac{\partial^2 A_4}{\partial t\partial y}&=&
\frac{\partial x_4}{\partial y}\left[\frac{\partial^2 A}{\partial x^2}\frac{\partial x_4}{\partial t}+\frac{\partial^2 A}{\partial y\partial x}\frac{\partial y_4}{\partial t}+\frac{\partial^2 A}{\partial t\partial x}\right]+\frac{\partial A}{\partial x}\frac{\partial^2 x_4}{\partial t\partial y}\\
&+&\frac{\partial y_4}{\partial y}\left[\frac{\partial^2 A}{\partial x\partial y}\frac{\partial x_4}{\partial t}+\frac{\partial^2 A}{\partial y^2}\frac{\partial y_4}{\partial t}+\frac{\partial^2 A}{\partial t\partial y}\right]+\frac{\partial A}{\partial y}\frac{\partial^2 y_4}{\partial t\partial y}\\
&=&O\left(n\mu^{2n}\right),
\end{eqnarray*}
where we also used \eqref{eq:x4tderivatives} and \eqref{eq:y4tderivatives}. Similarly the other estimates in the third statement follows. 
Finally, by \eqref{eq:dx5dy}, Lemma \ref{lem:D2xypartialderivatives} and Lemma \ref{lem:tpartialderivatives}
$$
\frac{\partial^2\underline x_5}{\partial t\partial y}=O\left(n{\mu^{2n}}\right),
$$
and by \eqref{eq:dx5dy}, the fact that $D_2=O\left(\left(\lambda/\mu\right)^n\right)$, see Lemma \ref{lem:D2xypartialderivatives}, we get
$$
\frac{\partial \underline x_5}{\partial y}=O\left(\left(\lambda{\mu}\right)^n\right).
$$
As a consequence,
\begin{eqnarray*}
\frac{\partial^2 \alpha_5}{\partial t\partial y}
&=&O\left(n\mu^{2n}\right),
\end{eqnarray*}
where we also used \eqref{eq:x5tderivatives} and \eqref{eq:y5tderivatives}. Similarly the other estimates in the last statement follows.  
\end{proof}

The following lemma refers to the function $\Phi$. Recall that $\Phi_2$ is the second component of $\Phi$ and $\phi_{22}=\partial\Phi_2/\partial y$, see \eqref{eq:Dphi}.
\begin{lem}\label{lem:ddtphi22}
Let $(t,a)\in\mathcal H_n$. Then 
$$
\frac{d\left(\phi_{22}\right)}{dt}=O\left(n\mu^{3n}(t,a)\right).
$$
\end{lem}
\begin{proof}
We calculate ${d\left(\phi_{22}\right)}/{dt}=\partial^2\Phi_2/\partial t\partial y$ using \eqref{eq:phisecondcomponent}. Observe that, by Lemma \ref{lem:tpartialderivatives}, Lemma \ref{lem:ypartialderivatives} and Lemma \ref{lem:ddtdyderivatives}, the terms in \eqref{eq:phisecondcomponent} containing a $\lambda^n$ will contribute to the estimate of the $t$-derivative of $\phi_{22}$ with an order of at most $\left(\lambda\mu^3\right)^n$.
Hence, 
\begin{eqnarray*}
\frac{d\left(\phi_{22}\right)}{dt}=\frac{\partial^2}{\partial t\partial y}\left[\delta D_2B_4\mu^{2n}+\delta D_2D_4\beta_5\mu^{2n}+\delta D_2B_4\alpha_5\mu^{n}\right]+O\left(n\left(\lambda\mu^3\right)^n\right).
\end{eqnarray*}
Observe that the mixed partial derivatives of the coefficients of $\mu^{2n}$, $\mu^n$ are of the same order. Hence, 
\begin{eqnarray*}
\frac{d\left(\phi_{22}\right)}{dt}&=&O\left(\frac{\partial^2}{\partial t\partial y}\left[\delta D_2B_4\mu^{2n}+\delta D_2D_4\beta_5\mu^{2n}\right]\right)+				O\left(n\left(\lambda\mu^3\right)^n\right).
\end{eqnarray*}
Because $D_2=O\left(\left(\lambda/\mu\right)^n\right)$, see Lemma \ref{lem:D2xypartialderivatives}, Lemma \ref{lem:tpartialderivatives}, Lemma \ref{lem:ypartialderivatives} and Lemma \ref{lem:ddtdyderivatives} all the terms which do not involve a partial derivative of $D_2$, give a contribution of order at most $n\left(\lambda\mu^3\right)^n$. Hence,
\begin{eqnarray*}
\frac{d\left(\phi_{22}\right)}{dt}&=&O\left(\frac{\partial}{\partial t}\left[\delta \frac{\partial D_2}{\partial y}B_4\mu^{2n}+\delta \frac{\partial D_2}{\partial y}D_4\beta_5\mu^{2n}\right]\right)\\&+&O\left(\frac{\partial}{\partial y}\left[\delta \frac{\partial D_2}{\partial t}B_4\mu^{2n}+\delta \frac{\partial D_2}{\partial t}D_4\beta_5\mu^{2n}\right]\right)+	O\left(n\left(\lambda\mu^3\right)^n\right).
\end{eqnarray*}
Observe that, by \eqref{eq:dx4dybetter} and \eqref{eq:dx5dybetter}, we have $\partial B_4/\partial y,\partial D_4/\partial y,\partial \beta_5/\partial y=O\left(\left(\lambda\mu\right)^{n}\right)$. Using Lemma \ref{lem:tpartialderivatives}, Lemma \ref{lem:ypartialderivatives} and Lemma \ref{lem:ddtdyderivatives} we reduce to 
\begin{eqnarray*}
\frac{d\left(\phi_{22}\right)}{dt}=O\left(\frac{\partial}{\partial t}\left[\delta \frac{\partial D_2}{\partial y}B_4\mu^{2n}+\delta \frac{\partial D_2}{\partial y}D_4\beta_5\mu^{2n}\right]\right)+O\left(n\mu^{3n}\right).
\end{eqnarray*}
The lemma follows by applying again Lemma \ref{lem:tpartialderivatives}, Lemma \ref{lem:ypartialderivatives} and Lemma \ref{lem:ddtdyderivatives}.
\end{proof}
The following lemma refers to the univalent function $\varphi$ introduced in \eqref{eq:qtaphi}.
\begin{lem}\label{lem:dDphidt}
Let $(t,a)\in\mathcal H_n$. Then 
$$
\frac{d\left(D\varphi(0)\right)}{d t}=O\left(n\mu^n(t,a)\right).
$$
\end{lem} 
\begin{proof}
Let $(t,a)\in\mathcal H_n$. Following the notation as in the proof of Lemma \ref{lem:dDphida}, by \eqref{eq:Qphi22dydssquared} we have
\begin{equation}\label{eq:Qphi22dydssquared2}
D\varphi(0)=\phi_{22}\left(\frac{dy}{ds}\right)^2.
\end{equation}
We are now ready to estimate the $t$-derivative of $D\varphi(0)$. Recall that  $(c_{t,a}+\Delta s)=\left(h^{-1}_{t,a}(\underline x_3)\right)_x$ and by \eqref{eq:dx3dy},
$$
\frac{ds}{dy}=\left(1+\frac{\alpha_3}{\mu^n}\right)\left(\delta B_2\mu^n+A_2\beta{\lambda}^n\right)+{\beta_3}\left(D_2\delta\mu^n+C_2\beta{\lambda}^n\right).
$$ 
Using the above expression for $ds/dy$ and Lemma \ref{lem:tpartialderivatives}, 
\begin{equation}\label{eq:ddtdsdy}
\frac{\partial}{\partial t}\left(\frac{ds}{dy}\right)=O\left(n\mu^n\right).
\end{equation}
Using $$0=\frac{\partial}{\partial t}\left(\frac{ds}{dy}\frac{dy}{ds}\right)=\frac{dy}{ds}\frac{\partial}{\partial t}\left(\frac{ds}{dy}\right)+\frac{ds}{dy}\frac{\partial}{\partial t}\left(\frac{dy}{ds}\right),$$
\eqref{eq:dyds} and \eqref{eq:ddtdsdy} we have 
\begin{equation}\label{eq:ddtdyds}
\frac{\partial}{\partial t}\left(\frac{dy}{ds}\right)=O\left(\frac{n}{\mu^n}\right).
\end{equation}
By taking the $t$-derivative of \eqref{eq:Qphi22dydssquared} and using Lemma \ref{lem:ddtphi22}, \eqref{eq:phi22}, \eqref{eq:dyds}, and \eqref{eq:ddtdyds} the lemma follows. 
\end{proof}

\begin{lem}\label{lem:dx5dtunderlinex}
For all $\underline x\in\sigma_{t,a}^{-1}\circ Z_{t,a} \left(\text{Dom}\left(HF_{t,a}\right)\right)$,
\begin{eqnarray*}
\frac{\partial x_5}{\partial t}(\underline x)&=&B_4(\underline x)n\mu^{n-1}(t,a)\frac{\partial\mu}{\partial t}\left(y_3+\mu^n(t,a)D_2\left(\underline x\right)y_1\right)+O\left(\frac{1}{\mu^n(t,a)}\right).
\end{eqnarray*} 
\end{lem}
\begin{proof}
By \eqref{eq:y3tderivativesfor}, 
$$
\frac{\partial y_3}{\partial t}(\underline x)=C_2\frac{\partial x_2}{\partial t}+D_2n\mu^{n-1}\frac{\partial\mu}{\partial t}y_1+D_2O\left(\frac{1}{\mu^{n}}\right)+O\left(\frac{1}{\mu^{2n}}\right)=D_2(\underline x)n\mu^{n-1}\frac{\partial\mu}{\partial t}y_1+O\left(\frac{1}{\mu^{2n}}\right) ,
$$ 
where we also used \eqref{eq:x2tderivatives}, and Lemma \ref{lem:D2B4underlinex}. 
Now \eqref{eq:y4tderivatives}, becomes 
$$
\frac{\partial y_4}{\partial t}(\underline x)=n\mu^{n-1}\frac{\partial\mu}{\partial t}\left(y_3+\mu^nD_2\left(\underline x\right)y_1\right)+O\left(\frac{1}{\mu^n}\right),
$$
and \eqref{eq:x5tderivatives}, becomes 
$$
\frac{\partial x_5}{\partial t}(\underline x)=B_4(\underline x)n\mu^{n-1}\frac{\partial\mu}{\partial t}\left(y_3+\mu^nD_2\left(\underline x\right)y_1\right)+O\left(\frac{1}{\mu^n}\right),
$$
where we used \eqref{eq:x4tderivatives}.
\end{proof}
\noindent
{\it Proof of Theorem \ref{prop:dftdt}.} 
Let $\underline x_0=(x_0,y_0)\in\text{Dom}\left(HF_{t,a}\right)$, $\underline{\tilde x}=(\tilde x,\tilde y)=Z_{t,a}(\underline x_0)$ and $\underline x=(x,y)=\sigma_{t,a}^{-1}\left(\underline{\tilde x}\right)\in\sigma_{t,a}^{-1}\circ Z_{t,a} \left(\text{Dom}\left(HF_{t,a}\right)\right)$. Because $\underline{\tilde x}=(c_{t,a},c_{t,a})+\underline x_0/D\varphi(0)$, using Proposition \ref{prop:tderivativescriticalpointscriticalvalue}, the fact that $D\varphi(0)$ is proportional to $\mu^n$ and Lemma \ref{lem:dDphidt} we have
\begin{equation}\label{eq:dtildexdt}
\frac{d\underline{\tilde x}}{dt}=O\left(\frac{n}{\mu^n}\right),
\end{equation}
and because $x=\tilde y$,
\begin{equation}\label{eq:dxdt}
\frac{d{ x}}{dt}=O\left(\frac{n}{\mu^n}\right). 
\end{equation}
The estimate for $dy/dt$ needs some preparation.
By \eqref{eq:dx3dx}, \eqref{eq:dxdt}, \eqref{eq:dx3dy} and \eqref{eq:x3tderivatives} we have
\begin{eqnarray}\label{eq:dx3dt}
\nonumber \frac{d{ x_3}}{dt}&=&\frac{\partial { x_3}}{\partial x}\frac{d{ x}}{d t}+\frac{\partial { x_3}}{\partial y}\frac{d{ y}}{dt}+\frac{\partial { x_3}}{\partial t}\\&=&\left(\delta B_2\mu^n+O(\lambda^n)\right)\frac{d{ y}}{dt}+B_2n\mu^{n-1}\frac{\partial\mu}{\partial t}y_1+O\left(\frac{n}{\mu^n}\right),
\end{eqnarray}
where $B_2$ and $\delta$ as uniformly away from zero, see Remark \ref{rem:deltaBawayfromzero}. Similarly, by \eqref{eq:dx3dx}, \eqref{eq:dxdt}, \eqref{eq:dx3dy}, \eqref{eq:y3tderivativesfor} we have 
\begin{equation}\label{eq:dy3dt}
\frac{d{ y_3}}{dt}=O\left(1\right)\frac{d{ y}}{dt}+O\left(\frac{n}{\mu^{2n}}\right),
\end{equation}
where we also used Lemma \ref{lem:D2B4underlinex}.
Observe now that 
$ 
\tilde x=\left(h^{-1}(\underline x_3)\right)_x
$
which implies, using \eqref{eq:dx3dt} and \eqref{eq:dy3dt},
\begin{eqnarray*}
\frac{d {\tilde x}}{dt}&=&\left(1+\frac{\alpha_3}{\mu^{n}}\right)\left(\left(\delta B_2\mu^n+O(\lambda^n)\right)\frac{d{ y}}{dt}+B_2n\mu^{n-1}\frac{\partial\mu}{\partial t}y_1+O\left(\frac{n}{\mu^n}\right)\right)+O(1)\frac{d{ y}}{dt}\\&=&\left(\delta B_2\mu^n+O(1)\right)\frac{d{ y}}{dt}+B_2n\mu^{n-1}\frac{\partial\mu}{\partial t}y_1+O\left(\frac{n}{\mu^n}\right).
\end{eqnarray*}
By \eqref{eq:dtildexdt}, we get 
\begin{equation}\label{eq:dydt}
\frac{d{ y}}{dt}=-\frac{ny_1}{\delta\mu}\frac{\partial\mu}{\partial t}\left(1+O\left(\frac{1}{\mu^n}\right)\right).
\end{equation} 
Observe now that, by \eqref{eq:dx5dx}, Lemma \ref{lem:D2B4underlinex}, \eqref{eq:dxdt}, \eqref{eq:dx5dy}, again Lemma \ref{lem:D2B4underlinex}, \eqref{eq:dydt} and Lemma \ref{lem:dx5dtunderlinex},
\begin{eqnarray}\label{eq:dx5dt}
\nonumber \frac{d{ x_5}}{dt}&=&\frac{\partial { x_5}}{\partial x}\frac{d{ x}}{dt}+\frac{\partial { x_5}}{\partial y}\frac{d{ y}}{dt}+\frac{\partial { x_5}}{\partial t}\\\nonumber &=&B_4\mu^{2n}\left(\delta D_2+O\left(\left(\frac{\lambda}{\mu}\right)^n\right)\right)\left(-\frac{ny_1}{\delta\mu}\frac{\partial\mu}{\partial t}\left(1+O\left(\frac{1}{\mu^n}\right)\right)\right)\\\nonumber &+&B_4n\mu^{n-1}\frac{\partial\mu}{\partial t}\left(y_3+\mu^n D_2y_1\right)+O\left(\frac{n}{\mu^n}\right)\\&=&B_4n\mu^{n-1}\frac{\partial\mu}{\partial t}y_3+O\left(\frac{n}{\mu^n}\right).
\end{eqnarray}
Similarly, by \eqref{eq:dx5dx}, Lemma \ref{lem:D2B4underlinex}, \eqref{eq:dxdt}, \eqref{eq:dx5dy}, again Lemma \ref{lem:D2B4underlinex}, \eqref{eq:dydt} and \eqref{eq:y5tderivatives}
\begin{equation}\label{eq:dy5dt}
\frac{d{ y_5}}{dt}=O\left(\frac{n}{\mu^{n}}\right). 
\end{equation}
Because
$
\left(\tilde F_{t,a}(\underline{\tilde x})\right)_x=\left(h_{t,a}^{-1}(\underline x_5)\right)_x
$, using \eqref{eq:dx5dt} and \eqref{eq:dy5dt}
\begin{eqnarray}\label{eq:dtildeFdt}\nonumber
\frac{d {\left(\tilde F_{t,a}(\underline{\tilde x})\right)_x}}{dt}&=&\left(1+\frac{\alpha_5}{\mu^{n}}\right)\left[B_4n\mu^{n-1}\frac{\partial\mu}{\partial t}y_3+O\left(\frac{n}{\mu^{n}}\right)\right]+O\left(\frac{n}{\mu^{n}}\right)\\&=&B_4n\mu^{n-1}\frac{\partial\mu}{\partial t}y_3+O\left(\frac{n}{\mu^{n}}\right).
\end{eqnarray}
Finally observe that $$f_{t,a}(\underline x_0)=\alpha^{-1}\left(\left(\tilde F(\underline{\tilde x})\right)_x\right)=\left(\left(\tilde F(\underline{\tilde x})\right)_x-c_{t,a}\right)D\varphi(0).$$
As a consequence, 
\begin{eqnarray*}
\frac{d f_{t,a}(\underline x_0)}{dt}&=&D\varphi(0)\frac{d {\left(\tilde F(\underline{\tilde x})\right)_x}}{d t}-D\varphi(0)\frac{\partial c_{t,a}}{\partial t}+ \left(\left(\tilde F(\underline{\tilde x})\right)_x-c_{t,a}\right)\frac{d D\varphi(0)}{dt}\\&=&D\varphi(0)B_4(\underline x)n\mu^{n-1}\frac{\partial\mu}{\partial t}y_3+O(n),
\end{eqnarray*}
where we used \eqref{eq:dtildeFdt}, the fact that $D\varphi(0)$ is proportional to $\mu^n$, Proposition \ref{prop:tderivativescriticalpointscriticalvalue}, Lemma \ref{lem:dDphidt} and the fact that $\left(\tilde F(\underline{\tilde x})\right)_x\in\alpha(U_{t,a})$ which has diameter proportional to $1/\mu^n$. 
In conclusion, using Lemma \ref{lem:D2B4underlinex},
\begin{eqnarray*}
\frac{d f_{t,a}}{dt}&=&D\varphi(0)B_4(c)n\mu^{n-1}\frac{\partial\mu}{\partial t}y_3(c)+O(n),
\end{eqnarray*}
where $B_4(c){D\varphi(0)}y_3(c)\partial\mu/\partial t$ is uniformly bounded away from zero.
\qed

\section{The period doubling curve}\label{sec:perioddoublingcurve}
In this section we prove the existence of a curve $PD_n$ contained in the strip $\mathcal H_n$ such that, each map corresponding to a parameter point in $PD_n$ has a period doubling Cantor attractor. We start by recalling the following definition.
\begin{defin}\label{cantorA} Let $M$ be a manifold and $f:M\to M$. An invariant Cantor set $A\subset M$ is called a period doubling Cantor attractor of $f$ if $f|A$ is conjugated to a $2$-adic  adding machine and there is a neighborhood $M\supset U\supset A$ such that the orbit of almost every point in $U$ accumulates at $A$.
 
\end{defin}

\begin{rem} A period doubling Cantor attractor has zero topological entropy. It carries a unique invariant probability measure. Strongly dissipative H\'enon-like maps at the boundary of chaos have period doubling Cantor attractors, see \cite{CLM}.
\end{rem}

\bigskip

We define the period doubling locus $PD_n$ as,
$$
PD_n=\left\{(t,a)\in\mathcal H_n\left|\right. HF_{t,a} \text{ has a period doubling Cantor attractor}\right\}.
$$
\begin{prop}\label{pdanalytic}
Let $F:\mathcal P\times M\to M$ be a real-analytic two dimensional unfolding of a map $f$ with a strong homoclinic tangency, then the period doubling locus $PD_n$ contains the graph of a real-analytic function $[-t_0,t_0]\ni t\mapsto PD_n(t)\in\mathbb R$. Moreover 
$$
\frac{dPD_n}{dt}=-\frac{n}{\mu^{n+1}(t,a)}\frac{\partial\mu}{\partial t}+O\left(\frac{n}{\mu^{2n}(t,a)}\right).
$$
\end{prop}

The proof of Proposition \ref{pdanalytic} needs some preparation. Denote the disk of radius $2$ centered at $0$ by $\mathbb D_2$.
Let $Y$ be the vector space of holomorphic and real symmetric functions $\epsilon:\mathbb D_2\times\mathbb D_2\to\mathbb D$ with $\epsilon(0,0)=0$. $Y$ is equipped with the $\Co$ norm.

A {\it H\'enon-like map} is a map $HF:\left[-2,2\right]^2\to\mathbb R^2$ of the form
\begin{eqnarray}\label{eq:Henonlikemapsdef}
HF\left(\begin{matrix}
x\\y
\end{matrix}\right)=\left(\begin{matrix}
x^2+\nu +\epsilon(x,y)\\
x
\end{matrix}\right),
\end{eqnarray}
where $\nu\in\mathbb R$ and $\epsilon\in Y$. The space of H\'enon-like maps is $\mathcal H=\mathbb R\times Y$. Observe that the maps defined in Theorem \ref{firstreturnmapanalytic} are in $\mathcal H$. In particular the set $\mathcal H_n$ defines a two-dimensional family in $\mathcal H$.

\bigskip

In the space of H\'enon-like maps, there exists a codimension one manifold $PD$ of maps which have a period doubling Cantor attractor, see \cite{CLM}. In particular the degenerate H\'enon family is transversal to the $PD$ manifold. This manifold is locally the graph of a real-analytic function 
$$
Ch:Y_{\rho}\to\mathbb R
$$
where $Y_{\rho}=\left\{\epsilon\in Y\left|\right.|\epsilon|\leq\rho\right\}$ with $\rho$ sufficiently small. In particular,  
$$
HF=(\nu,\epsilon)\in PD\cap \left(\mathbb R\times Y_{\rho}\right)\iff \nu-Ch(\epsilon)=0.
$$
Observe that the maps defined in Theorem \ref{firstreturnmapanalytic} are in $\mathcal H_{\rho}=\mathbb R\times Y_{\rho}$. In particular the set $\mathcal H_n$ defines a two-dimensional family in $\mathcal H_{\rho}$.
\begin{lem}\label{lem:closetofullfamily}
There exists an $E>0$ such that, for $n$ large enough and every $t\in\left[-t_0,t_0\right]$, the family $$\left[-{E},E\right]\ni\beta\mapsto HF_{\beta}=HF_{t,sa_n(t)+{\beta}/{|\mu(t, a_n(t))|^{2n}}}$$ contains $HF_{0}$ with entropy $0$ and $HF_{-E}$ with entropy at least $\log 2$. 
\end{lem}
\begin{proof}
The degenerate H\'enon map $(x,y)\mapsto(x^2,x)$ has a sink at $(0,0)$ and entropy zero. Because this is an open condition, for $n$ large, $HF_0$ being arbitrarily close to this degenerate H\'enon map has also entropy zero. 

The degenerate H\'enon map $(x,y)\mapsto(x^2+\nu,x)$, with $\nu\leq -3$ has a horseshoe and in particular it has entropy $\log 2$. Let 
\begin{equation}\label{eq:Echoice}
E=\frac{4}{\text{min}_{\cup_{n>0}\mathcal H_n}\left(B_4(c)D\varphi(0)/\mu^n\right)}.
\end{equation}
Observe that $E$ is a positive finite number because $B_4(c)$ is uniformly away from zero, see Remark \ref{rem:deltaBawayfromzero} and $D\varphi(0)$ is proportional to $\mu^n$. From \eqref{eq:dnuda}, $\nu_{-E}\leq -3$ and in particular, for $n$ large enough, $HF_{-E}$ has entropy at least $\log 2$.
\end{proof}
\begin{lem}\label{lem:transversaltoPD}
The family $$\mathcal H_n\ni(t,a)\mapsto HF_{t,a}\in\mathcal H_{\rho}$$ is transversal to $PD$ and intersects $PD$ in a single curve. 
\end{lem}
\begin{proof}
Fix $(t,a)\in\mathcal H_n$. From \eqref{eq:nutaformula}, we get 
$$
\frac{\partial\nu_{t,a}}{\partial a}=\frac{\partial D\varphi(0)}{\partial a}(v_{t,a}-c_{t,a})+ D\varphi(0)\left(\frac{\partial v_{t,a}}{\partial a}-\frac{\partial c_{t,a}}{\partial a}\right),
$$
and by Lemma \ref{lem:dDphida}, Proposition \ref{prop:aderivativescriticalpointscriticalvalue}, Lemma \ref{lem:vminusc}, we get
\begin{equation}\label{eq:dnuda}
\frac{\partial\nu_{t,a}}{\partial a}= D\varphi(0)B_4(c)\mu^{n}+O(\mu^n).
\end{equation}
As a consequence,
\begin{equation}\label{eq:depsilonda}
\frac{\partial\epsilon_{t,a}}{\partial a}=\frac{\partial f_{t,a}}{\partial \beta}\frac{\partial \beta}{\partial a} -\frac{\partial\nu_{t,a}}{\partial a}=O(n\mu^n),
\end{equation}
where we used Theorem \ref{prop:HFclosetoFa}. 
Fix $t\in [-t_0,t_0]$ and consider the family $\left[-E,E\right]\ni\beta\mapsto HF_{\beta}=(\nu_{\beta},\epsilon_{\beta})\in\mathcal H_{\rho}.$   
The function 
$$\left[-E,E\right]\ni\beta\mapsto\nu_{\beta}-Ch(\epsilon_{\beta})\in\mathbb R$$ is strictly monotone, because 
$$
\frac{\partial}{\partial\beta}\left(\nu_{\beta}-Ch(\epsilon_{\beta})\right)=\frac{D\varphi(0)B_4(c)}{\mu^n}+O\left(\frac{n}{\mu^n}\right)\neq 0,
$$
where we used \eqref{eq:dnuda}, \eqref{eq:depsilonda} and the fact that $\partial a/\partial\beta$ is proportional to $1/\mu^{2n}$. This means that the curve $\beta\mapsto HF_{\beta}$ is transversal to the level set of $\nu-Ch(\epsilon)$. In particular, the preimage of a level set of $\nu-Ch(\epsilon)$ is the graph of a function over the $t$-axes.

From Lemma \ref{lem:closetofullfamily}, we know that $HF_0$ has entropy zero and $HF_{-E}$ has entropy at least $\log 2$. The $PD$ manifold in $\mathcal H_{\rho}$ is the graph of the function $Ch$. Hence, $\mathcal H_{\rho}\setminus PD$ has two connected components, one containing $HF_0$ and the other $HF_{-E}$. In particular there exists a unique $\beta_{\infty}$ such that $HF_{\beta_{\infty}}\in PD$. In particular, the two dimensional family $HF_{t,a}$ intersects $PD$ transversally in a single curve.
\end{proof}
\noindent
{\it Proof of Proposition \ref{pdanalytic}.} 
 Consider the preimage of $PD$ under the family $\mathcal H_n\ni(t,a)\mapsto HF_{t,a}$. By Lemma \ref{lem:transversaltoPD} and its proof, this preimage is the graph of a real-analytic function and it is contained in $PD_n$. Abusing the notation we denote this function by $[-t_0,t_0]\mapsto PD_n(t)$. Moreover by \eqref{eq:nutaformula}, we get 
$$
\frac{\partial\nu_{t,a}}{\partial t}=\frac{\partial D\varphi(0)}{\partial t}(v_{t,a}-c_{t,a})+ D\varphi(0)\left(\frac{\partial v_{t,a}}{\partial t}-\frac{\partial c_{t,a}}{\partial t}\right),
$$
and by Lemma \ref{lem:dDphidt}, Proposition \ref{prop:tderivativescriticalpointscriticalvalue}, Lemma \ref{lem:vminusc}, we get,
\begin{equation}\label{eq:dnudt}
\frac{\partial\nu_{t,a}}{\partial t}= D\varphi(0)B_4n\mu^{n-1}\frac{\partial\mu}{\partial t}y_3+O(n).
\end{equation}
As a consequence,
\begin{equation}\label{eq:depsilondt}
\frac{\partial\epsilon_{t,a}}{\partial t}=\frac{\partial f_{t,a}}{\partial t} -\frac{\partial\nu_{t,a}}{\partial t}=O(n),
\end{equation}
where we used Theorem \ref{prop:dftdt}. 
For every $t\in\left[-t_0,t_0\right]$, one has $$\nu_{t,PD_n(t)}-Ch\left(\epsilon_{t,PD_n(t)}\right)=0.$$
Hence,
$$
0=\frac{\partial\nu}{\partial t}+\frac{\partial\nu}{\partial a}\frac{d PD_n}{d t}-DCh\left(\frac{\partial\epsilon}{\partial t}+\frac{\partial\epsilon}{\partial a}\frac{d PD_n}{d t}\right),
$$
and by using \eqref{eq:dnuda}, \eqref{eq:depsilonda}, \eqref{eq:dnudt}, \eqref{eq:depsilondt} we have 
$$
\frac{dPD_n}{dt}=-\frac{n }{\mu}\frac{\partial\mu}{\partial t}y_3+O\left(\frac{n}{\mu^{2n}(t,a)}\right).
$$
Observe that, using \eqref{eq:distq3toy}, $y_3=1/\mu^n(y_4)=1/\mu^n\left(1+O(1/\mu^n)\right)$. The proposition follows.
\qed

\section{Coexistence of non-periodic attractors}
In this section we prove that each real-analytic two-dimensional unfolding contains countably many maps with two period doubling Cantor attractors as well as countably many maps with finitely many sinks and two period doubling Cantor attractors.

\begin{theo}\label{theo:2PDattarctors}
Let $F:\mathcal P\times M\to M$ be a real-analytic two dimensional unfolding of a map $f$ with a strong homoclinic tangency, then there exists a countable set $2PD\subset\mathcal P$, such that, each map in $2PD$ has two period doubling Cantor attractors. The homoclinic tangency persists along a curve in $\mathcal P$ and it is contained in the closure of $2PD$. 
\end{theo}
\begin{proof}
By Proposition $4$ in \cite{BMP}, each $\mathcal H_n$ contains finitely many curves which are graphs of functions $b_{n,n_0}:[t^{-}_{n,n_0},t^+_{n,n_0}]\mapsto\R$. The maps corresponding to points in these curves have a secondary homoclinic tangency. Moreover, for all $E>0$ and for $n$ large enough, $$b_{n,n_0}(t^+_{n,n_0})=sa_n(t)+\frac{E}{{|\mu(t,sa_n(t))|^{2n}}}\text{ and }b_{n,n_0}(t^-_{n,n_0})=sa_n(t)-\frac{E}{{|\mu(t,sa_n(t))|^{2n}}}.$$ As consequence, each curve $b_{n,n_0}$ crosses the strip $\mathcal H_n$ and in particular the curve $PD_n$. Let 
$$
\mathcal {P}_{n,n_0}=\left\{(t,a)\in [-t_0,t_0]\times [-a_0,a_0] \left|\right.  t\in\left[t^{-}_{n,n_0},t^{+}_{n,n_0}\right],  |a-sa_n(t)|\leq\frac{E}{|\mu(t,sa_n(t))|^{2n}}\right\}.
$$
Proposition $5$ in \cite{BMP} says that $F:\mathcal P_{n,n_0}\times M\mapsto M$ can be reparametrized to become an unfolding. We apply now the previous sections to the restricted unfoldings and we get new curves $PD^{(n,n_0)}_m$, corresponding to maps with a period doubling Cantor attractor which, for $m$ large enough, accumulate to the curves $b_{n,n_0}$, see Proposition \ref{pdanalytic}. Because the curves $b_{n,n_0}$ intersect the curve $PD_n$, then also the curves $PD^{(n,n_0)}_m$ intersect $PD_n$ and the intersection points correspond to maps with $2$ period doubling Cantor attractors, see Figure \ref{2PDattr}. Finally, by Proposition $3$ in \cite{BMP}, the distance between $\mathcal P_{n,n_0}$ and $\mathcal P_{n,n_0+1}$ is of the order $1/n$. Hence, the tangency curve at $a=0$ is contained in the closure of the set of maps with two period doubling Cantor attractors.
\end{proof}
\vskip .2 cm
\begin{figure}
\centering
\includegraphics[width=1\textwidth]{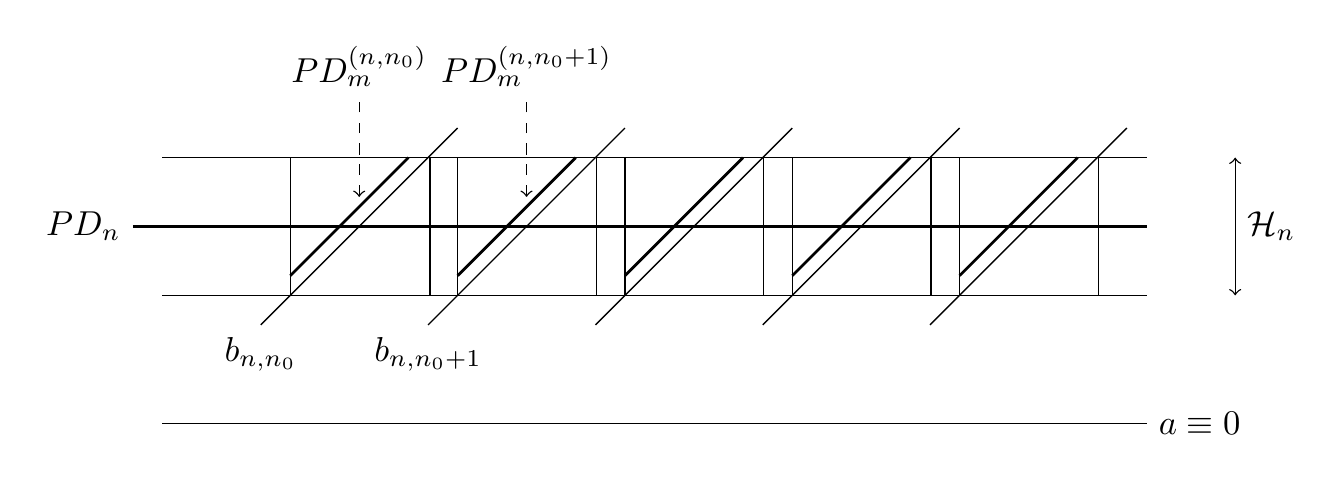}
\caption{Intersections of curves of period doubling attractors}
\label{2PDattr}
\end{figure}
\begin{theo}\label{sinksand2PDattarctors}
Fix $S\in\N$. Let $F:\mathcal P\times M\to M$ be a real-analytic two dimensional unfolding of a map $f$ with a strong homoclinic tangency, then there exists a countable set $S2PD \subset \mathcal P$, such that, each map in $S2PD$ has at least $S$ sinks and two period doubling Cantor attractors. The homoclinic tangency persists along a curve in $\mathcal P$ and it is contained in the closure of $S2PD$. 
\end{theo}
\begin{proof}
This is a consequence of Theorem A in \cite{BMP} and Theorem \ref{theo:2PDattarctors}. It is enough to stop the inductive procedure in the proof of Theorem A in \cite{BMP} at the step $S$. At this moment there are boxes $\mathcal {P}^S_{n,n_0}\subset \mathcal P$ which are crossed diagonally by curves of secondary homoclinic tangencies, $b^S_{n,n_0}$. The family restricted to each of these boxes, $\mathcal {P}^S_{n,n_0}$, is an unfolding of a map with a strong homoclinic tangency given by the curve $b^S_{n,n_0}$ and all maps in $\mathcal {P}^S_{n,n_0}$  have at least $S$ sinks, see Proposition 5 in \cite{BMP}. By applying Theorem \ref{theo:2PDattarctors} to each of these restricted families, we find countable sets $S2PD_{n,n_0}\subset \mathcal {P}^S_{n,n_0}$ consisting of maps with at least $S$ sinks and two period doubling Cantor attractors. The closure of each $S2PD_{n,n_0}$ contains the curve of secondary tangencies $b^S_{n,n_0}$. Because, by Proposition $3$ in \cite{BMP}, the distance between $\mathcal P^S_{n,n_0}$ and $\mathcal P^S_{n,n_0+1}$ is of the order $1/n$ the set $S2PD=\cup_{n,n_0}S2PD_{n,n_0}$ contains in its closure the tangency curve at $a=0$.
\end{proof}

\section{Laminations of multiple attractors}
We are now ready to prove the coexistence of multiple attractors and study their stability. In families with at least three parameters, we construct maps with several period doubling Cantor attractors as well as maps with sinks and period doubling Cantor attractors forming laminations. We split the discussion in two subsections in which the main theorems and their meanings are carefully presented.
\subsection{Laminations in general unfoldings}

In higher dimensional families with a two dimensional section which is an unfolding of a homoclinic tangency, the sinks and the two period doubling attractors constructed in Theorem \ref{theo:2PDattarctors}, they start to move simultaneously creating codimension two laminations, see Theorem \ref{theo:KPDlaminations2dimensional}. The leafs of the laminations are real-analytic and they have uniform positive diameter. The same phenomenon holds for the finitely many sinks and the two period doubling Cantor attractors constructed in Theorem \ref{sinksand2PDattarctors}, see Theorem \ref{theo:SsinksKperioddoubling}.

 Given a two-dimensional family $F$ which is an unfolding of a map with a strong homoclinic tangency, recall that there exist curves  $PD_n$, in parameter space, corresponding to maps with one period doubling Cantor attractor. Furthermore, from the proof of Theorem \ref{theo:2PDattarctors}, for every $n$ large enough, there are finitely many curves $b_{n,n_0}$ consisting of maps with secondary tangencies and they cross the curve $PD_n$. The curves $b_{n,n_0}$  are accumulated by new period doubling curves $PD_m^{(n,n_0)}$. In the next proposition we compare the angle of the curves $b_{n,n_0}$ and $PD_m^{(n,n_0)}$ with the angle of the curve $PD_n$. 

\begin{prop}\label{transangle}
The curves $b_{n,n_0}$ and $PD_m^{(n,n_0)}$ cross the curve $PD_n$ transversally and the angle is larger than $Vn\mu_{\text{min}}^{-3n/2}$ where $V$ is a uniform constant and $\mu_{\text{min}}=\min_{(t,a)}|\mu(t,a)|$.
\end{prop}
\begin{proof}
By Proposition $4$ in \cite{BMP}, 
$$\frac{db_{n,n_0}}{dt}=-\frac{n}{\mu^{n+1}}\frac{\partial\mu}{\partial t}+V_{t,a}{n \lambda^{\theta n}}+O\left(|\lambda|^{\theta n}\right),$$ where $V_{t,a}$ is uniformly away from zero, $0<\theta<1/2$ and $|\lambda|^{2\theta}|\mu|^{3}>1$. In particular, by Proposition \ref{pdanalytic}, the curve $b_{n,n_0}$ crosses the curve $PD_n$ transversally and the angle is larger than, $Vn\mu^{-3n/2}$ where $V$ is a uniform constant. Because, for $m$ large, the curve $PD_m^{(n,n_0)}$ is $\Cuno$ close to $b_{n,n_0}$, see Proposition \ref{pdanalytic}, the same angle estimate holds for $PD_m^{(n,n_0)}$ and the transversality with $PD_n$ follows.
\end{proof}
\begin{rem}\label{Vdep}
Observe that, by Proposition \ref{transangle}, the angle formed by the intersection of $PD_m^{(n,n_0)}$ with the curve $PD_n$ is larger than $Vn\mu_{\text{min}}^{-3n/2}$, for $n\geq n_0$. By Remark 10 in \cite{BMP}, $n_0$ and $V$ are locally constant, i.e. they depends continuously on the family.
\end{rem}

\begin{theo}\label{theo:KPDlaminations2dimensional}

Let $M$, $\mathcal P$ and $\mathcal T$ be real-analytic manifolds and $F:\left(\mathcal P\times\mathcal T\right)\times M\to M$ a real-analytic   family with $\text{dim}(\mathcal P)=2$ and  $\text{dim}(\mathcal T)\geq 1$. If there exists $\tau_0\in\mathcal T$ such that $F_0:\left(\mathcal P\times\left\{\tau_0\right\}\right)\times M\to M$ is an unfolding of a map $f_{\tau_0}$ with a strong homoclinic tangency, then for $k=1,2$, there exists a codimension $k$ lamination of maps with at least $k$ period doubling Cantor attractors which persist along the leafs. The homoclinic tangency persists along a global codimension one manifold in $\mathcal P\times \mathcal T$ and this tangency locus is contained in the closure of the lamination. Moreover, the leafs of the lamination are real-analytic and they have a uniform positive diameter.
\end{theo}
\begin{proof} Without loss of generality we may assume that $\mathcal P\times \mathcal T=[-1,1]^2\times [-1,1]^{r-2} $ where $r=\text{dim}(\mathcal P\times \mathcal T)$. A point in parameter space is given by  $(t,a,\tau)$ and $a=0$ corresponds to the tangency locus. Moreover, we may also assume that, for all $\tau\in\mathcal T$, the family restricted to $[-1,1]^2\times\left\{\tau\right\}$, $F_{\tau}$, is an unfolding of a strong homoclinic tangency of the map $f_{\tau}$. Given $\tau$, denote the $n^{\text{th}}$ H\'enon strip of the family restricted to $[-1,1]^2\times\left\{\tau\right\}$ by $\mathcal H_n(\tau)$.

 For $k=1$, Theorem \ref{theo:KPDlaminations} is a reformulation of Lemma \ref{lem:transversaltoPD}. Namely, for every $n$, there is a real-analytic function $PD_n:(t,\tau)\mapsto PD_n(t,\tau)$ such that the graph of $PD_n$ in $\mathcal P\times \mathcal T$, i.e. points of the form $(t, PD_n(t,\tau),\tau)$, consists of maps with at least one period doubling Cantor attractor. Because, for a fixed $\tau$, these graphs are contained in $\mathcal H_n(\tau)$ and the distance of $\mathcal H_n(\tau)$ to the tangency locus at $a=0$ is of the order $1/\mu^n$, the graphs of $PD_n$ contain the tangency locus in their closure.  

Let us consider now the case $k=2$. As in the proof of Theorem \ref{theo:2PDattarctors}, from Proposition \ref{transangle}, we get that in the unfolding $F_{\tau}$, there are period doubling curves $PD^{(n,n_0)}_m(\tau)$ and $PD_n(\tau)$ which intersect transversally in the point $\left(t_m^{(n,n_0)}(\tau),a_m^{(n,n_0)}(\tau),\tau\right)$, see Figure \ref{2PDattr}. From Proposition \ref{transangle} we get a lower bound for the angle between these curves which is independent of the parameter $\tau$, see Remark \ref{Vdep}. This transversality implies that this intersection persists for all $\tau$ as the graph of a real-analytic function.  In particular the two period doubling attractors at $\left(t_m^{(n,n_0)}(\tau),a_m^{(n,n_0)}(\tau),\tau\right)$ have their continuation in all unfoldings $F_{\tau}$, creating a codimension two real-analytic leaf of the lamination. Let 
$$2PD(\tau)=\left\{ PD^{(n,n_0)}_m(\tau)\cap PD_n(\tau)\left|\right.n,n_0,m>0\right\}.$$ According to Theorem \ref{theo:2PDattarctors}, the set $2PD(\tau)$ consists of countably many maps in the family $F_{\tau}$ which have two period doubling attractors. Moreover, it accumulates at the tangency curve in $\mathcal P\times \left\{\tau\right\}$. The set $2PD=\cup_{\tau}2PD(\tau)$ is the required lamination. In particular for a given $\tau$, the set $2PD(\tau)$ moves along the leafs of $2PD$ while varying $\tau$ and all leafs project onto $[-1,1]^{r-2}$, they have uniform diameter. 
\end{proof}

\subsection{Laminations in saddle deforming unfoldings}
We analyze here coexisting phenomena in saddle deforming unfoldings. In these families, whose definition and meaning is explained in the following, we find codimension three laminations of coexisting attractors and we give a precise description of the asymptotic direction of the leafs of the laminations. We start with some basic definition and explanation. 

Two crucial invariants associate to a saddle point are indeed its eigenvalues. One of the properties of an unfolding is that the unstable eigenvalue changes with one parameter, see Definition \ref{unfolding}. We introduce here a notion of unfolding for families with at least three parameters. We require that the unfolding of the homoclinic tangency is also able to change both eigenvalues independently. In these so called {\it saddle deforming unfoldings}, the level sets of the eigenvalue pair define a codimension two foliation of the tangency locus, called the {\it eigenvalue foliation}. A  saddle deforming unfolding contains a three dimensional subfamily transversal to the codimension three leafs of the eigenvalue foliation. 
\begin{defin}\label{defn:sdunfoldings}
Let $M$, $\mathcal P$ and $\mathcal T$ be real-analytic manifolds and $F:\left(\mathcal P\times\mathcal T\right)\times M\to M$ a real-analytic   family with $\text{dim}(\mathcal P)=2$ and  $\text{dim}(\mathcal T)\geq 1$. If there exists $\tau_0\in\mathcal T$ such that $F_0:\left(\mathcal P\times\left\{\tau_0\right\}\right)\times M\to M$ is an unfolding of a map with a strong homoclinic tangency and the eigenvalues at the saddle point, after a reparametrization, satisfy $$\frac{\partial\mu}{\partial\tau}= 0\text{,   }\frac{\partial\lambda}{\partial\tau}\neq 0, $$
then the family $(t,a,\tau)\mapsto F_{t,a,\tau}$ is called a saddle deforming unfolding.
\end{defin}

\begin{theo}\label{theo:KPDlaminations}
Let $M$, $\mathcal P$ and $\mathcal T$ be real-analytic manifolds and $F:\left(\mathcal P\times\mathcal T\right)\times M\to M$ be a saddle deforming unfolding then there exists a codimension $3$ lamination $3PD$ of maps with at least $3$ period doubling Cantor attractors which persist along the leafs.  The leafs of the lamination are real-analytic and when $\text{dim}(\mathcal T)\geq 2$ they have a uniform positive diameter. The homoclinic tangency persists along a global codimension one manifold in $\mathcal P\times \mathcal T$ and for each leaf of the eigenvalue foliation, there is a sequence of leafs of $3PD$ which accumulate at this eigenvalue leaf.  
\end{theo}
\begin{proof} Without loss of generality we may assume that $\mathcal P\times \mathcal T=[-1,1]^2\times [-1,1]^{r-2} $ where $r=\text{dim}(\mathcal P\times \mathcal T)$. A point in parameter space is given by  $(t,a,\underline\tau)$ and $a=0$ corresponds to the tangency locus. Moreover, we may also assume that, for all $\underline \tau\in\mathcal T$, the family restricted to $[-1,1]^2\times\left\{\underline\tau\right\}$, $F_{\underline \tau}$, is an unfolding of a strong homoclinic tangency. Given $\underline \tau$, there are period doubling curves $PD_n(\underline \tau)\subset [-1,1]^2\times\left\{\underline\tau\right\}$ contained in the $n^{\text{th}}$ H\'enon strip of the family restricted to $[-1,1]^2\times\left\{\underline\tau\right\}$ which we denote here by $\mathcal H_n(\underline\tau)$.
 By Proposition $4$ in \cite{BMP}, each $\mathcal H_n(\underline\tau)$ contains finitely many curves which are graphs of functions $b_{n,n_0}(\underline\tau):\left[t^{-}_{n,n_0}(\underline\tau),t^+_{n,n_0}(\underline\tau)\right]\mapsto\R$. The maps corresponding to points in these curves have a strong homoclinic tangency. Moreover, by Proposition \ref{transangle}, each curve $b_{n,n_0}(\underline\tau)$ crosses transversally the curve $PD_n(\underline\tau)$, in a point denoted by $\mathcal B_{n,n_0}(\underline\tau)$. Because of the transversality, the points $\mathcal B_{n,n_0}(\underline\tau)$ define a real-analytic codimension two manifold $\mathcal B_{n,n_0}\subset\mathcal P\times \mathcal T$. As in the proof of Theorem \ref{theo:KPDlaminations2dimensional}, let $PD_n=\cup_{\underline\tau}PD_n(\underline\tau)=\mathcal P_n\subset\mathcal P\times \mathcal T$ be a codimension one manifold. Observe that, maps corresponding to points in $\mathcal P_n$ have a period doubling Cantor attractor and $\mathcal B_{n,n_0}\subset \mathcal P_n$ is a codimension one manifold consisting of maps with also a strong homoclinic tangency.

In the following we prove that  the restriction of the initial family $F:\mathcal P_n\times M\to M$ contains an unfolding of the strong homoclinic tangency of the maps in $\mathcal B_{n,n_0}$. The only condition which is not straightforward is that $d\mu$ is non zero along $\mathcal B_{n,n_0}$. After a reparametrization of the coordinates $\underline \tau=(\tau,\tau_2\dots,\tau_{r-2})$, we may assume that $\partial\lambda/\partial \tau=1$ and $\partial\lambda/\partial \tau_j=0$ for all $2\leq j\leq r-2$. This is possible because, by hypothesis, our family is a saddle deforming unfolding. It suffices to prove that the two dimensional restriction of $F:\mathcal P_n\times M\to M$ to the coordinates $(t, a(t,\tau,0,\dots,0),\tau,0,\dots,0)$ is an unfolding. In this two dimensional family, the tangency locus $\mathcal B_{n,n_0}$ is a curve parametrized by $\tau$, say the points in this curve are of the form $(t(\tau),a(\tau),\tau)$. The aim is to prove that $d\mu/d\tau\neq 0$ along this curve. Observe that 
$$
\frac{d\mu}{d\tau}=\frac{\partial\mu}{\partial t}\frac{dt}{d\tau}+\frac{\partial\mu}{\partial a}\frac{da}{d\tau}+\frac{\partial\mu}{\partial\tau}
$$
where ${\partial\mu}/{\partial t}\neq 0$ and ${\partial\mu}/{\partial \tau}=0$ because our initial family is a saddle deforming unfolding.  Moreover, using Proposition \ref{pdanalytic} and the fact that $\partial\mu/\partial\tau=0$, one has ${da}/{d \tau}=O\left(n/\mu^{2n}\right)$. Hence,
\begin{equation}\label{eq:dmudtau}
\frac{d\mu}{d\tau}=\frac{\partial\mu}{\partial t}\frac{dt}{d\tau}+O\left(\frac{n}{\mu^{2n}}\right).
\end{equation}
Let $\left(\Delta t,\Delta a, \Delta\tau\right)$ be a tangent vector to the curve $\mathcal B_{n,n_0}$. From the proof of Theorem $B$ in \cite{BMP}, we have
\begin{eqnarray*}
\left(1+O\left(\frac{1}{n}\right)\right)\left(V\Delta t+ V_1\Delta\tau \right)=0.
\end{eqnarray*}
Hence,
\begin{equation}\label{eq:dtdtau}
\frac{dt}{d\tau}=\frac{\Delta t}{\Delta\tau}=-\frac{V_1}{V}\left(1+O\left(\frac{1}{n}\right)\right),
\end{equation}
with $V$ a function uniformly away from zero and by $(4.47)$ in \cite{BMP}, there are functions, $K$ and $K'$ uniformly bounded away from zero, such that
\begin{equation}\label{eq:Vuno}
V_{1}=K \frac{\partial \lambda}{\partial \tau}+K'\frac{\partial\mu}{\partial \tau}=K \frac{\partial \lambda}{\partial \tau}\neq 0,
\end{equation}
where we used the property $\partial\mu/\partial\tau=0$ of saddle deforming unfolding. Combining \eqref{eq:dmudtau}, \eqref{eq:dtdtau} and \eqref{eq:Vuno} we get that $d\mu/d\tau\neq 0$ and our family is an unfolding.
By Theorem \ref{theo:KPDlaminations2dimensional}, $\mathcal P_n$ contains a codimension $ 2$ lamination, denoted by $3PD_{n,n_0}$, of maps with $2$ period doubling Cantor attractors which persist along the real-analytic leafs of the lamination. Moreover, the lamination contains $\mathcal B_{n,n_0}$ in its closure. Observe that, all maps in $\mathcal P_n$ have at least one period doubling Cantor attractor of a fixed combinatorial type. Hence, every map in $3PD_{n,n_0}$ actually has at least $3$ period doubling Cantor attractors.  

By Proposition $3$ in \cite{BMP}, the distance between $b_{n,n_0}(\underline\tau)$ and $b_{n,n_0+1}(\underline\tau)$ is of the order $1/n$. Hence, the tangency locus at $a=0$ is contained in the closure of the lamination $3PD=\cup_{n,n_0}3PD_{n,n_0}$ formed by maps with at least $3$ period doubling Cantor attractors.

For understanding the diameter of the leafs, write $\mathcal P\times \mathcal T=[-1,1]^2\times [-1,1] \times [-1,1]^{r-3}$, the parameters as $(t,a,\tau,\tau')$ and denote the set of parameters of the form $(\cdot,\cdot,\cdot,\tau')$ by $\mathcal P(\tau')$. For a given $\tau'$ the lamination $3PD$ intersects $\mathcal P(\tau')$ in countably many points which move along the leafs of $3PD$ while varying $\tau'$. In particular, all leafs project onto $[-1,1]^{r-3}$, they have uniform diameter. 

In order to complete the proof, observe that each family $\mathcal P_n$ is the graph of a function $[-1,1]^2\times [-1,1]^{r-3}\ni (t,\tau,\tau')\mapsto\mathcal P_n(t,\tau,\tau')\in [-1,1]$. In particular, we can reparametrize $\mathcal P_n$ by $[-1,1]^2\times [-1,1]^{r-3}$ using the parameters $(t,\tau,\tau')$. For every $\tau'\in [-1,1]^{r-3}$, we can identify the two dimensional family $\mathcal P_n(\tau')=\mathcal P_n\cap \mathcal P(\tau')$ with $[-1,1]^2$. This family contains the curve $\mathcal P_n(\tau')\cap \mathcal B_{n,n_0}$ of strong homoclinic tangencies and, as shown before, it is an unfolding of these tangencies.  For each  $\tau'\in [-1,1]^{r-3}$ we can choose a reparametrization of $\mathcal P_n(\tau')$, depending real-analytically on $\tau'$, obtaining coordinates $(t',a',\tau')\in [-1,1]^2\times [-1,1]^{r-3}$ such that $a'=0$ corresponds to the tangency locus. Moreover we can assume
\begin{equation}\label{eq:partialmupartiallambda}
\frac{\partial\mu}{\partial t'}=1, \frac{\partial\mu}{\partial \tau'}=0, \frac{\partial\lambda}{\partial \tau'}=0.
\end{equation}

 We apply now the previous sections to the restricted unfolding and we get new curves $PD'_m(\tau')\subset \mathcal P_n(\tau')$, corresponding to maps with a period doubling Cantor attractor, see Proposition \ref{pdanalytic}. Using Proposition $4$ in \cite{BMP} and Proposition \ref{transangle}, there are finitely many curves $b'_{m,m_0}(\tau')\subset \mathcal P_n(\tau')$ which cross $PD'_m(\tau')$ transversally. Because this intersection depends real-analytically on $\tau'$, the intersection points form a codimension two manifold $\mathcal B'_{m,m_0}\subset \mathcal P_n$. 

Let $\left(\Delta t',\Delta a', \Delta\tau'\right)$ be a tangent vector to the manifold $\mathcal B'_{m,m_0}$. From the proof of Theorem $B$ in \cite{BMP}, we have
\begin{eqnarray*}
\left(V'+O\left(\frac{1}{m}\right)\right)\Delta t'+ \sum_{i=1}^{r-3}\left( V'_i+O\left(\frac{1}{m}\right)\right)\Delta\tau'_i =0,
\end{eqnarray*}
and 
\begin{eqnarray*}
O\left(\frac{m}{\mu^m}\left(\Delta t'+ \Delta\tau' \right)\right)=\Delta a',
\end{eqnarray*}
with $V'$ a function uniformly away from zero and by $(4.47)$ in \cite{BMP}, there are functions, $K_i$ and $K'_i$ uniformly bounded away from zero, such that
\begin{equation*}\label{eq:V1}
V'_{i}=K_i \frac{\partial \lambda}{\partial \tau'}+K'_i\frac{\partial\mu}{\partial \tau'}= 0,
\end{equation*}
where we used \eqref{eq:partialmupartiallambda}. Hence, 
\begin{equation}\label{eq:Bnalligne}
\Delta t'=O\left(\frac{1}{m}\Delta\tau'\right)\text{ and }\Delta a'=O\left(\frac{m}{\mu^m}\Delta\tau'\right).
\end{equation}
Choose a point $(t'_0,a'_0,\tau'_0)\in\mathcal B'_{m,m_0}$. Observe that $a'_0=O\left(1/\mu^m\right)$. As a consequence of \eqref{eq:Bnalligne}, the manifold $\mathcal B'_{m,m_0}$ is at the distance of the order $1/m$ to a codimension two subspace defined by $t'=t'_0$ and  $a'=a'_0$. By \eqref{eq:partialmupartiallambda} this subspace is a level set for the eigenvalues. 

As final remark, according to Proposition \ref{transangle} for every $\tau'\in [-1,1]^{r-3}$ there exists a sequence of curves $PD_k^{(m,m_0)}(\tau')$ accumulating in $\Cuno$ at  $\mathcal B'_{m,m_0}(\tau')$. The transversal intersection points with $PD_m(\tau')$ varies real-analitycally with $\tau'$ forming a leaf of the lamination $3PD$. Hence, there are leafs of $3PD$ which accumulates at $\mathcal B'_{m,m_0}$ which by them self accumulate at the leafs of the eigenvalue foliation.
\end{proof}

As before one also gets the coexistence of finitely many sinks and finitely many period doubling Cantor attractor. The proof is similar to the proof of Theorem \ref{sinksand2PDattarctors}.
\begin{theo}\label{theo:SsinksKperioddoubling}
Let $M$, $\mathcal P$ and $\mathcal T$ be real-analytic manifolds and $F:\left(\mathcal P\times\mathcal T\right)\times M\to M$ be a saddle deforming unfolding, then for every $1\leq k\le 3$ and $S\in\mathbb N$, there exists a codimension $k$ lamination $SkPD$ of maps with at least $S$ sinks and $k$ period doubling Cantor attractors which persist along the leafs.  The homoclinic tangency persists along a global codimension one manifold in $\mathcal P\times\mathcal T$ and this tangency locus is contained in the closure of the lamination. The leafs of the lamination are real-analytic and when $1\leq k<\text{dim}(\mathcal P\times\mathcal T)$ they have a uniform positive diameter. Moreover, for each leaf of the eigenvalue foliation, there is a sequence of leafs of $S3PD$ which accumulate at this eigenvalue leaf.  
\end{theo}
 \begin{proof}
 Without loss of generality we may assume that $\mathcal P\times \mathcal T=[-1,1]^2\times [-1,1]^{r-2} $ where $r=\text{dim}(\mathcal P\times \mathcal T)$. A point in parameter space is given by  $(t,a,\tau)$ and $a=0$ corresponds to the tangency locus. Moreover, we may also assume that, for all $\tau\in\mathcal T$, the family reastricted to $[-1,1]^2\times\left\{\tau\right\}$, $F_{\tau}$, is an unfolding of a strong homoclinic tangency of the map $f_{\tau}$. We apply the inductive procedure in the proof of Theorem A in \cite{BMP} $S$ times. At this moment there are boxes $\mathcal {P}^S_{n,n_0}(\tau)\subset  \mathcal P\times \left\{\tau\right\}$ which are crossed diagonally by curves of secondary homoclinic tangencies, $b^S_{n,n_0}(\tau)$. The family $F_{\tau}$ restricted to each of these boxes $\mathcal {P}^S_{n,n_0}(\tau)$ is an unfolding of a map with a strong homoclinic tangency given by the curve $b^S_{n,n_0}(\tau)$ and all maps in the box have at least $S$ sinks, see Proposition 5 in \cite{BMP}. Because the boxes $\mathcal {P}^S_{n,n_0}(\tau)$ depends analyticaly on $\tau$, we can consider the union  $\mathcal {P}^S_{n,n_0}=\cup_{\tau}\mathcal {P}^S_{n,n_0}(\tau)$. Since the definition of saddle deforming unfolding only involves the dependence of $\lambda$ on the parameters, the restriction of the initial family to $\mathcal {P}^S_{n,n_0}$ is still a saddle deforming unfolding  and each map in this restriction has $S$ sinks. We apply  Theorem \ref{theo:KPDlaminations2dimensional} and Theorem \ref{theo:KPDlaminations} to the restricted family.  
 
 We consider first the cases $k=1,2$. For each  $\mathcal {P}^S_{n,n_0}$ we get a codimension $k$ lamination $SkPD_{n,n_0}$ of maps with at least $S$ sinks and $k$ period doubling Cantor attractors. The leafs of this lamination project onto $\mathcal T$ and they contain the tangency locus of the restricted family in their closure. Observe that, by Proposition $3$ in \cite{BMP}, the distance between $\mathcal P^S_{n,n_0}(\tau)$ and $\mathcal P^S_{n,n_0+1}(\tau)$ is of the order $1/n$ and the collection of boxes $\mathcal P^S_{n,n_0}(\tau)$ contains in its closure the tangency curve at $a=0$. Hence, the set $SkPD=\cup_{n,n_0}SkPD_{n,n_0}$ contains in its closure the tangency locus at $a=0$.
 
Consider the case $k=3$.  For each  $\mathcal {P}^S_{n,n_0}$ we get a codimension $3$ lamination $S3PD_{n,n_0}$ of maps with at least $S$ sinks and $3$ period doubling Cantor attractors. If the dimension of the family is three, then the set $S3PD_{n,n_0}$ is countable and it accumulates at the tangency locus of $\mathcal P^S_{n,n_0}$. For the same reason as for the case $k=1,2$ it contains the tangency locus of the original family in its closure. If the dimension of the family is strictly larger than three, after a reparametrization of the coordinates $\tau=(\tau_1,\dots,\tau_{r-2})$, we may assume that $\partial\lambda/\partial \tau_1=1$ and $\partial\lambda/\partial \tau_j=0$ for all $2\leq j\leq r-2$. This is possible because, by hypothesis, our family is a saddle deforming unfolding. The leafs of the lamination $S3PD_{n,n_0}$ project onto $\left\{0\right\}\times\left\{0\right\}\times\left\{0\right\}\times[-1,1]^{r-3}$ and they contain the tangency locus of the restricted family in their closure. As before, the set $S3PD=\cup_{n,n_0}S3PD_{n,n_0}$ contains in its closure the tangency locus at $a=0$.
   \end{proof}
A similar phenomenon holds for laminations of infinitely many sinks and one period doubling Cantor attractor. This is stated precisely in the following.
\begin{theo}\label{theo:infsinksPDlaminations}
Let $M$, $\mathcal P$ and $\mathcal T$ be real-analytic manifolds and $F:\left(\mathcal P\times\mathcal T\right)\times M\to M$ be a saddle deforming unfolding, then there exists a codimension $3$ lamination $NHPD$ of maps with infinitely many sinks and at least $1$ period doubling Cantor attractor which persist along the leafs. The leafs of the lamination are real-analytic and when the dimension of $\mathcal T$ is at least $2$, they have a uniform positive diameter.  Moreover, the homoclinic tangency persists along a global codimension one manifold in $\mathcal P\times \mathcal T$ and for each leaf of the eigenvalue foliation, there is a sequence of leafs of $NHPD$ which accumulate at this eigenvalue leaf.  
\end{theo}
\begin{proof} We use the notation as in the proof of Theorem \ref{theo:KPDlaminations}.  Write the parameter space as $\mathcal P\times \mathcal T=[-1,1]^2\times [-1,1] \times [-1,1]^{r-3}$ , the parameters as $(t,a,\tau,\tau')$ and denote the set of parameters of the form $(\cdot,\cdot,\cdot,\tau')$ by $\mathcal P(\tau')$. Recall that $\mathcal P_n\subset \mathcal P\times \mathcal T$ is a codimension one manifold formed of maps with one period doubling Cantor attractor. For a given $\tau'$, write $\mathcal P_n(\tau')=\mathcal P_n\cap \mathcal P(\tau')$. In the proof of Theorem \ref{theo:KPDlaminations} we proved that the restriction of our initial family, $F_{\tau'}$, to  $\mathcal P_n(\tau')$ is a two dimensional unfolding of a strong homoclinic tangency. We apply Theorem A in \cite{BMP} to $F_{\tau'}$ and we get a set $NHPD_n(\tau')$ of maps with infinitely many sinks. Each point in $NHPD_n(\tau')$ depends real-analytically on $\tau'$ and it forms a real-analytic codimension two manifold in $\mathcal P_n$. All points together form a codimension two lamination $NHPD_n$ of maps with infinitely many sinks and one period doubling Cantor attractor coming from being in $\mathcal P_n$.  By construction the lamination contains the tangency loci $\mathcal B_{n,n_0}$ in its closure. Let $NHPD=\cup_n NHPD_n$. As in the previous proofs, the lamination $NHPD$ contains its closure the tangency locus at $a=0$. 

Choose $\tau'\in  [-1,1]^{r-3}$ and let $f\in NHPD_n(\tau')$. From the proof of Theorem A in \cite{BMP}, there are nested boxes $\mathcal P^g(\tau')$ with $\left\{f\right\}=\cap^g \mathcal P^g(\tau')$. Each  $\mathcal P^g(\tau')$ contains a real-analytic curve $b^g(\tau')$ consisting of maps with a strong homoclinic tangency associated to the original saddle point and the family restricted to $\mathcal P^g(\tau')$  is an unfolding of the maps in $b^g(\tau')$. Moreover, the box $\mathcal P^g(\tau')$ contains a curve $PD^g_m(\tau')$ consisting of maps with one period doubling Cantor attractor and the curve  $b^g_{m,m_0}(\tau')\subset \mathcal P^{g+1}(\tau')$ of maps with a secondary homoclinic tangency. Similarly as in the proof of Theorem \ref{theo:KPDlaminations}, by Proposition \ref{transangle}, the  curve $b^g_{m,m_0}(\tau')$ intersects transversally the curve $PD^g_m(\tau')$ in a point $\mathcal B'_{m,m_0}(\tau')$. Denote this intersection point by $\mathcal B^{g+1}(\tau')\in \mathcal P^{g+1}(\tau') $ and notice that it depends real-analytically on $\tau'$ forming a manifold $\mathcal B^{g+1}$. Observe that the leaf of the lamination $NHPD$ trough $f$ is the limit of the leafs $\mathcal B^{g}$. As was shown in the proof of Theorem  \ref{theo:KPDlaminations},  the leafs $\mathcal B^{g}$ asymptotically align with the leafs of the eigenvalue foliation, see  \eqref{eq:Bnalligne}.

Any point in the tangency locus is in the closure of the set $NHPD$. In particular it is in the closure of a sequence of leafs of the lamination $NHPD$. In turn these leafs are in the closure of leafs of the form $\mathcal B^{g}$ which are  asymptotically align with the leafs of the eigenvalue foliation. Hence, there is a sequence of leafs of the lamination $NHPD$ which accumulates at any chosen level set of the eigenvalue foliation. This concludes the proof.   
\end{proof}
Our method allows also to find coexistence of period doubling Cantor attractors and a strange one. In this case, we do not have laminations because the stability of the strange attractors in families with at least three parameters is not yet understood. We start by recalling the formal definition of a {\it strange attractor}.
\begin{defin} \label{strange} Let $M$ be a manifold and $f:M\to M$. An open set $U\subset M$ is called a trapping region if $\overline{f(U)}\subset U$. An attractor in the sense of Conley is
$$
\Lambda=\bigcap_{j\ge 0} f^j(U).
$$
The attractor $\Lambda$ is called topologically transitive if it contains a  dense orbit. If $\Lambda$ contains a dense orbit which satisfies the Collet-Eckmann conditions, i.e. there exist a point $z$, a vector $v\in T_zM$ and a constant $\kappa>0$ such that $$|Df^n(z)v|\geq e^{\kappa n} \text{ for all } n>0,$$ then $\Lambda$ is called a strange attractor.
\end{defin}
\begin{theo}\label{theo:KPD1Strange}
Let $M$, $\mathcal P$ and $\mathcal T$ be real-analytic manifolds and $F:\left(\mathcal P\times\mathcal T\right)\times M\to M$ a saddle deforming unfolding, then the set of maps with at least $2$ period doubling Cantor attractors and one strange attractor has Hausdorff dimension at least $\text{dim}(\mathcal P\times\mathcal T)-2$. 
\end{theo}
\begin{proof} We use the notation as in the proof of Theorem \ref{theo:KPDlaminations}.  Write the parameter space as $\mathcal P\times \mathcal T=[-1,1]^2\times [-1,1] \times [-1,1]^{r-3}$ , the parameters as $(t,a,\tau,\tau')$ and denote the set of parameters of the form $(\cdot,\cdot,\cdot,\tau')$ by $\mathcal P(\tau')$. Recall that $\mathcal P_n\subset \mathcal P\times \mathcal T$ is a codimension one manifold formed of maps with one period doubling Cantor attractor. For a given $\tau'$, write $\mathcal P_n(\tau')=\mathcal P_n\cap \mathcal P(\tau')$. In the proof of Theorem \ref{theo:KPDlaminations} we proved that the restriction of our initial family, $F_{\tau'}$, to  $\mathcal P_n(\tau')$ is an unfolding of a strong homoclinic tangency at $\mathcal B_{n,n_0}(\tau')$.  We apply now the previous sections to the restricted unfolding and we get new curves $PD^{(n,n_0)}_m(\tau')$, corresponding to maps with a period doubling Cantor attractor, see Proposition \ref{pdanalytic}. Using Proposition $4$ in \cite{BMP} and Proposition \ref{transangle}, there are finitely many curves $b^{(n,n_0)}_{m,m_0}(\tau')$ which cross $PD^{(n,n_0)}_m(\tau')$ transversally. The maps in the curves $b^{(n,n_0)}_{m,m_0}(\tau')$ have a strong homoclinic tangency and $PD^{(n,n_0)}_m(\tau')$ is a one-dimensional unfolding of the homoclinic tangency of the map at the intersection of the two curves. By Theorem A in \cite{MV}, the curve $PD^{(n,n_0)}_m(\tau')$ contains a set of positive Lebesgue measure with a strange attractor. In particular we found a set, $2PDSA(\tau')$, of positive Lebesgue measure of maps with one strange attractor and at least two period doubling attractors, one from being in $\mathcal P_n(\tau')$ and the other from being in  $PD^{(n,n_0)}_m(\tau')$. The proof is concluded by varying $\tau'$ and taking $2PDSA=\cup_{\tau'}2PDSA(\tau')$.
\end{proof}

\section{The H\' enon family and other polynomial families}\label{henonfamily}
In this section we apply our theorems to the H\' enon family and other polynomial families. We start by recalling the definition of H\' enon family. The two parameter family $F:\mathbb{R}^2\times\mathbb{R}^2\to\mathbb{R}^2$, 
\begin{equation}\label{Henonmap}
F_{a,b}\left(\begin{matrix}
x\\y
\end{matrix}
\right)=\left(\begin{matrix}
a+x^2-by\\x
\end{matrix}
\right),
\end{equation}
is called the {\it H\' enon family}. It is well known that the H\' enon family is an unfolding of a map with a strong homoclinic tangency. A proof can be found for example in \cite{BMP}, or \cite{Berger}.  
\vskip .2 cm
\begin{theo}\label{theo:Henonapplication}
Fix $S\in\N$. The H\' enon family contains a countable set $S2PD$ of maps with at least $S$ sinks and two period doubling Cantor attractors. Moreover the set of H\'enon maps with $S$ sinks, one period doubling Cantor attractor and one strange attractor has Hausdorff dimension at least one. 
\end{theo}
\begin{proof} 
The first statement is an application of Theorem \ref{sinksand2PDattarctors}. For the second statement we stop the inductive procedure in the proof of Theorem A in \cite{BMP} at the step $S$. At this moment there are boxes $\mathcal {P}^S_{n,n_0}\subset \mathcal P$ which are crossed diagonally by curves of secondary homoclinic tangencies, $b^S_{n,n_0}$. The family restricted to each of these boxes $\mathcal {P}^S_{n,n_0}$ is an unfolding of a map with a strong homoclinic tangency at the curve $b^S_{n,n_0}$ and all maps in the box have at least $S$ sinks, see Proposition 5 in \cite{BMP}.  We apply now the previous sections to the restricted unfolding and we get new curves $PD^{(n,n_0,S)}_m$, corresponding to maps with a period doubling Cantor attractor, see Proposition \ref{pdanalytic}. Using Proposition $4$ in \cite{BMP} and Proposition \ref{transangle}, there are finitely many curves $b^{(n,n_0,S)}_{m,m_0}$ which cross $PD^{(n,n_0,S)}_m$ transversally. The maps in the curves $b^{(n,n_0,S)}_{m,m_0}$ have a strong homoclinic tangency and $PD^{(n,n_0,S)}_m$ is a one-dimensional unfolding of the homoclinic tangency of the map at the intersection of the two curves. By Theorem A in \cite{MV}, the curve $PD^{(n,n_0,S)}_m$ contains a set of positive Lebesgue measure of maps with a strange attractor. 
\end{proof}
Theorem \ref{theo:SsinksKperioddoubling}, Theorem \ref{theo:infsinksPDlaminations} and Theorem \ref{theo:KPD1Strange} give the following.
\begin{theo}\label{theo:polyapplication}
Let $\text{Poly}_d({\mathbb{R}^2})$ be the space of real polynomials of $\mathbb{R}^2$ of degree at most $d$, with $d\ge 2$. Given a map $f$ in  $\text{Poly}_d({\mathbb{R}^2})$ which has a strong homoclinic tangency, there is a global codimension one manifold along which this tangency persists. The following holds:
\begin{itemize}
\item[-] for every $1\leq k\le 3$ and $S\in\N$ there exists a codimension $k$ lamination $SkPD$ of maps with at least $S$ sinks and $k$ period doubling Cantor attractors which persist along the leafs,
\item[-]  there exists a codimension $3$ lamination $NHPD$ of maps with infinitely many sinks and at least $1$ period doubling Cantor attractors which persist along the leafs,
\item[-]the set of maps with at least $2$ period doubling Cantor attractors and one strange attractor has Hausdorff dimension at least $\text{dim}\left(\text{Poly}_d({\mathbb{R}^2})\right)-2$. 
\end{itemize}
  The leafs of the laminations are real-analytic and  they have uniform diameter. Moreover, for each leaf of the eigenvalue foliation, there is a sequence of leafs of $3PD$ and a sequence of leafs of $NHPD$ which accumulate at this eigenvalue leaf.  
\end{theo}


\end{document}